\documentclass[11pt,reqno]{amsart}
 \usepackage[foot]{amsaddr}
\usepackage[a4paper,margin=1in]{geometry}

\usepackage{amsmath,bm}
\usepackage{amssymb}
\usepackage{amsthm}
\usepackage{mathrsfs}
\usepackage{upgreek}
\usepackage{graphicx} 
\usepackage[dvipsnames]{xcolor}
\usepackage{mathtools}

\usepackage{caption}
\usepackage{subcaption}

\usepackage{comment}

\numberwithin{equation}{section}

\usepackage{hyperref}

\mathtoolsset{showonlyrefs}

\newcommand{\abs}[1]{\lvert #1\rvert}

\newcommand{\cc}{{\normalfont{\text{c}}}}
\newcommand{\loc}{{\normalfont{\text{loc}}}}

\usepackage{pifont}

\usepackage{fancyhdr}
\usepackage[us,12hr]{datetime} 
\fancypagestyle{plain}{
\fancyhf{}
\rhead{\footnotesize Version of \today}
\lhead{\thepage}
}
\usepackage{color}
\newcommand{\half}{\frac{1}{2}}

\newcommand{\norm}[1]{\left \lVert #1 \right \rVert}
\newcommand{\snorm}[1]{\left \lvert #1 \right \rvert}

\newcommand{\R}{\mathbb{R}}

\newcommand{\IN}{\mathbb{N}}

\newcommand{\OV}{\mathsf{V}}
\newcommand{\OK}{\mathsf{K}}

\newcommand{{\D}}{\normalfont{\text{D}}}
\newcommand{{\G}}{\normalfont{\text{G}}}
\newcommand{{\U}}{\normalfont{\text{U}}}
\newcommand{{\I}}{\normalfont{\text{I}}}
\newcommand{{\per}}{\normalfont{\text{per}}}
\newcommand{{\y}}{{\boldsymbol{{y}}}}

\newcommand{{\bx}}{{\bf x}}
\newcommand{{\by}}{{\bf y}}
\newcommand{{\bz}}{{\bf z}}
\newcommand{{\bd}}{{\bf d}}

\newcommand{\dotp}[2]{\left ( #1,#2 \right )}

\usepackage{tikz}

\newcommand\encircle[1]{%
  \tikz[baseline=(X.base)] 
    \node (X) [draw, shape=circle, inner sep=0] {\strut #1};}

\newtheorem{theorem}{Theorem}[section]
\newtheorem{corollary}[theorem]{Corollary}
\newtheorem{lemma}[theorem]{Lemma}

\newtheorem{definition}[theorem]{Definition}
\newtheorem{assumption}[theorem]{Assumption}
\newtheorem{proposition}[theorem]{Proposition}

\theoremstyle{remark}
\newtheorem{remark}{Remark}[section]
\newcommand{\fh}[1]{{\color{OliveGreen} #1}}

\begin{document}
%
%

\title{
Deep ReLU Neural Network Emulation 
\\ 
in High-Frequency Acoustic Scattering}

\author{Fernando Henr\'iquez$^\dagger$}
\author{Christoph Schwab$^\ddagger$}
\address{$^\dagger$Chair of Computational Mathematics and Simulation Science (MCSS), \'Ecole Polytechnique F\'ed\'eral de Lausanne, 1015, Lausanne, Switzerland.}
\address{$^\ddagger$Seminar for Applied Mathematics, ETH Z{\"u}rich, 8092, Z{\"u}rich, Switzerland.}
\email{fernando.henriquez@epfl.ch}
\email{schwab@math.ethz.ch}

\begin{abstract}
We obtain wavenumber-robust error bounds 
for the deep neural network (DNN) emulation
of the solution to the time-harmonic, 
sound-soft acoustic scattering problem in the 
exterior of a smooth, convex obstacle in two physical dimensions.

The error bounds are based on
a boundary reduction of the scattering problem
in the unbounded exterior region 
to its smooth, curved boundary $\Gamma$
using the so-called
combined field integral equation (CFIE), a
well-posed, second-kind boundary integral equation (BIE) 
for the field's Neumann datum on $\Gamma$.
In this setting, the continuity and stability constants of this formulation 
are explicit in terms of the (non-dimensional) wavenumber $\kappa$.

Using wavenumber-explicit asymptotics of the problem's Neumann datum,
we analyze the DNN approximation rate for this problem.
We use fully connected NNs of the feed-forward type
with Rectified Linear Unit (ReLU) activation.
Through a constructive argument we prove the existence of DNNs 
with an $\epsilon$-error bound in the $L^\infty(\Gamma)$-norm
having a small, fixed width and a depth that increases \emph{spectrally}
with the target accuracy $\epsilon>0$. 
We show that for fixed $\epsilon>0$,
the depth of these NNs should increase
\emph{poly-logarithmically} with respect to the wavenumber $\kappa$
whereas the width of the NN remains fixed.
Unlike current computational approaches 
such as wavenumber-adapted versions of the Galerkin Boundary Element Method (BEM)
with shape- and wavenumber-tailored solution \emph{ansatz} spaces, 
our DNN approximations
do not require any prior analytic information about the scatterer's shape.
\end{abstract}



\maketitle

\section{Introduction}
\label{sec:Intro}
In recent years, the use of DNNs for the numerical approximation of PDEs
has emerged as an alternative to well-established, ``classical''
discretization techniques such as Finite Element-, Finite Difference- and
Finite Volume methods.
The success of DNNs in the field of computational mathematics
is arguably based on three main pillars: (i) universality results
(e.g.~\cite{PinkusActa99} and references therein), (ii) 
a considerably stronger approximation properties
compared to classical linear approximation methods
or nonlinear (adaptive) approximation schemes, and (iii)
on the intense development of dedicated machine learning
algorithms and computer hardware for DNN-based computations.

Mounting computational evidence together with theoretical insights
unmistakably underlines the notable gains stemming 
from the superior expressivity properties of DNNs.
These favorable properties come at the price of a \emph{nonlinear dependence
of the DNNs upon their hidden parameters}.
This renders the solution approximation by DNNs of linear PDEs and BIEs a nonlinear, finite-parametric problem which must be solved numerically
in an iterative fashion during the so-called ``DNN training'' phase.
Therefore, the question arises if, in particular applications, the potential gain afforded by employing 
DNNs for numerical approximation of solutions is so high as to justify the use of nonlinear optimization in the computation of the DNN approximation (as compared to a linear system solve in, e.g. Galerkin type 
methods based on linear subspaces). 

\subsection{Previous work on numerical high-frequency scattering}
\label{sec:HighkScat}
Due to its large importance in applications, 
the numerical solution of exterior, time-harmonic
acoustic and electromagnetic scattering problems
for large wavenumber $\kappa$ has received considerable
attention by engineers and by numerical analysts.
A large body of literature is available. 
First, as $\kappa\to\infty$, it is classical that the 
solution of the problem tends to the so-called
``geometric optics'' limit, the solution of 
which is governed by the Eikonal equation, see, eg., \cite[Chap. X]{TaylorPDO} and \cite{MoraLudwig68}
and the references there. 
The Eikonal equation is a Hamilton-Jacobi PDE which,
in general, must be treated numerically.
For large, but finite $\kappa$, asymptotic expansion w.r. to
$1/\kappa$ of the solution have been obtained, and justified,
for smooth, convex scatterers \cite{LudwigUniform67,MelTaylor85}. 
These works provide recursive expressions for the terms
appearing in the asymptotic expansions w.r.~to $1/\kappa \downarrow 0$ of the solution,
in terms of powers of $\kappa^{\frac{1}{3}}$,
for the solution of the exterior problem, in any space dimension $d\geq 2$.
These results indicate that in the case of convex scatterers, and 
for large $\kappa$, the solution comprises a piecewise
smooth part plus oscillatory terms whose specific structure amounts
to a smooth modulation of oscillatory functions depending on a nontrivial
way on the scatterer's boundary $\Gamma$.
This indicates the multiscale nature of the problem, and
immediately implies a \emph{scale resolution} requirement for 
the FEM and BEM. These are typically of the type $\kappa^2h<1$ for low
order FE or of the type $\kappa h<1$ for BE. 
This kind of requirement is prohibitive for large values of $\kappa$,
especially in space dimension
$d=3$.
High order methods of polynomial degree $p>1$ have scale resolution
requirements of the type $\kappa h/p< 1$. See, e.g., \cite{IhlBook,WildeActa12}
and the references there.

Early attempts to alleviate this scale resolution requirement
by incorporating information on the asymptotic expansions
into the numerical schemes include \cite{Bourdonn1994} 
for an approach in connection with Galerkin BEM.
We refer to \cite{WildeActa12} for a survey of these developments
and a comprehensive list of references.

The past decade has seen intense developments in the numerical
analysis of FE and BE methods for the Helmholtz equation at high wavenumber.
For further details, we refer to e.g. 
\cite{BabIhl1,BabIhl2}, and also to \cite{IhlBook} and
the references there. For boundary integral formulations, 
several critical issues in FE discretizations 
such as dispersion, and scale resolution
are either absent, or strongly mitigated. 
We refer to \cite{WildeActa12} and the 
references there for detailed discussion.
Boundary reduction and BIE reformulation 
is also the basis for the present results.

\subsection{Contributions}
\label{sec:contrib}
We investigate the wavenumber-robustness of a NN-based 
numerical approximation for time-harmonic, 
sound-soft, acoustic scattering in the exterior $\mathbb{R}^2 \backslash \overline{\D}$
of a bounded scatterer occupying a smooth, bounded domain $\D$.
Adopting a (standard, for this type of problem, see, e.g., \cite{WildeActa12,SS10})
suitable (in particular, resonance-free) 
reduction to a boundary integral equation on 
the closed curve $\Gamma = \partial \D$ 
implies satisfaction of the Sommerfeld radiation condition
at infinity, and of the homogeneous Helmholtz PDE in $\mathbb{R}^2 \backslash \overline{\D}$.
Continuity and coercivity of the boundary integral operator (BIO) 
with $\kappa$-explicit coercivity and continuity bounds renders 
Ritz-Galerkin variational approximations of the unknown 
density on $\Gamma$ well-posed, and quasi-optimal.
In the present paper, we extend our work \cite{AHS23_2956} and
address \emph{wavenumber-explicit convergence rates} of deep-BEM approximations 
in $L^2(\Gamma)$ of the corresponding boundary integral equations (BIEs for short).
This is reminiscent to e.g. the so-called ``Deep Ritz Method'' 
NN approximation for variational elliptic PDEs in \cite{Yuetal17}. 
In the low-frequency case, 
in \cite{AHS23_2956}, we showed that on polygonal $\Gamma$, 
ReLU-NN approximations can resolve the corner singularities
of the surface densities in BIEs at an optimal (even exponential)
rate, in terms of the ``number of degrees of freedom'', being the 
NN weights that are active in the DNN approximation.

The results in the present paper are based on the high wavenumber asymptotics
of the unknown surface density solution to a second kind BIE reformulation
of the exterior scattering problem. 
NN approximation is preceded by
boundary reduction whereby the governing equation and
the radiation condition at infinity are satisfied exactly.
The surface density is unknown, and we study neural networks 
to approximate it in a wavenumber-robust fashion, based on high wavenumber asymptotics
proposed in \cite{SKS15} for the density on the surface of the scatterer.
Numerical approximations based on this idea
have been developed in recent years in
\cite{Bourdonn1994,DGS07,EE19,EcevitBoubAn22}, among others.
These methods were based on earlier work on high-wavenumber asymptotics 
\cite{MR0196254,LudwigUniform67,MoraLudwig68,MelTaylor85}
in PDE theory.
While in \cite{EE19}, local geometric asymptotics 
have been used to develop geometry- and wavenumber-adapted 
trial spaces for Galerkin boundary integral equations, 
such spaces necessarily depend on the wavenumber and the geometry,
with corresponding challenges to accurate numerical integration
of the Galerkin matrix for the (likewise wavenumber-dependent)
fundamental solution. 
While affording wavenumber-robust approximations these 
highly problem-adapted trial spaces for the Galerkin BEM
method entail locally different approximation spaces, 
depending on the illuminated, shadow and glancing regions
on the boundary $\Gamma$, with wavenumber-explicit localization and 
of these regions of $\Gamma$ and strongly differing functional forms
of the trial functions in them. 
While it was argued (and corroborated by sophisticated numerical tests)
in \cite{EcevitBoubAn22} that this results in a $\kappa$-robust approximation
scheme, the significant analytic overhead in specifying the 
$\kappa \to \infty$ asymptotics of the solution in dependence on the geometry of $\Gamma$. 
This is used in \cite{Bourdonn1994,DGS07,EE19,EcevitBoubAn22} 
to a-priori design ``problem adapted, non-polynomial, approximation
spaces'' on $\Gamma$.
This design, being strongly geometry dependent, 
could be viewed as a significant obstruction
to the versatility of the approach proposed in \cite{EE19,EcevitBoubAn22}
and in similar, earlier work \cite{DGS07}.

In the present paper, we prove that 
deep feedforward, strict ReLU-activated NNs of \emph{fixed width}
can achieve a corresponding approximation performance
subject to mild constraints on NN depth $L$
w.r. to the wavenumber $\kappa$ and the target accuracy $\epsilon \in (0,1]$.
In contrast to e.g.~\cite{EcevitBoubAn22},
without any a-priori adaptation of the NN to the scatterers' shape.
The main results of this work are Theorems \ref{thm:approx_V_s_k} and 
\ref{thm:approximation_phi_k}, in which we state different ReLU-NN emulation results
with wavenumber robust convergence rates. However, we remark that the results 
presented in Section \ref{ssec:approx_fock_integral}, which concern the emulation
of the so-called Fock's integral and its derivatives
(naturally appearing in geometric optics based asymptotics, also on surfaces),
are of independent interest.

The results are established for NNs with ReLU activation fuction.
Standard arguments (sketched in various places of the text) allow to extend 
to other, non-polynomial activation functions (sigmoidal, tanh, etc.), 
and to infer corresponding results also for neuromorphic approximations
by means of the conversion algorithm introduced in \cite{SWBCPG2022}.
\subsection{Layout}
\label{sec:outline}
This work is structured as follows.
In Section~\ref{sec:problem_model} we introduce the sound-soft
acoustic scattering and describe its boundary reduction via the combined field
boundary integral equation.
In addition, we recall $\kappa$-explicit stability and coercivity bounds
together with the well-posedness of the Galerkin formulation in $L^2(\Gamma)$. 
A key component of the ensuing wavenumber-explicit DNN analysis
are known asymptotic expansion of the solution with respect to the wavenumber $\kappa$.
These are thoroughly reviewed in Section~\ref{sec:AsTotField}.
In Section~\ref{sec:NN} we introduce 
definitions, notation, and results concerning DNNs. 
In Section~\ref{sec:k_robust_ReLU_NN_emulation} we present 
the main results of this work: the wavenumber-explicit ReLU-NN
emulation of the solution to the CFIE, the problem's Neumann trace.
Finally, in Section~\ref{sec:Concl} we provide concluding remarks.
\subsection{Notation}
\label{sec:Notation}
We denote by $\mathbb{N}$ the set of natural number starting with $1$,
and for any $C>1$ we set $\mathbb{N}_{\geq C} \coloneqq \mathbb{N} \cap [C,\infty)$.
Let $\Omega$ be a bounded, Lispchitz domain with boundary $\partial \Omega$
and exterior complement
$\Omega^\cc\coloneqq \mathbb{R}^d \backslash \overline{\Omega}$
in $\mathbb{R}^d$, $d \in \mathbb{N}$.
We denote by $L^p(\Omega)$, $p\in [1,\infty)$ the Lebesgue space of measurable, $p$-integrable
functions in $\Omega$, with the appropiate extension to $p= \infty$, 
and by $H^s(\Omega)$, $s\geq0$, standard Sobolev spaces in $\Omega$.

In addition, let $H^s(\partial \D)$, for $s\in[0,1]$, {be the Sobolev}
space of traces on $\partial \D$ (\cite[Section 2.4]{SS10}).
As is customary, we identify $H^0(\partial \D)$ with $L^2({\partial \D})$ and, 
for $s\in [0,1]$, $H^{-s}(\partial \D)$ with the dual space of $H^s(\partial \D)$. 
Exterior Dirichlet and Neumann trace operators are denoted by
\begin{equation}
	\gamma^\cc\colon H^1_\loc(\D^\cc)\to H^{\half}(\Gamma)
	\quad
	\text{and}
	\quad
	\frac{\partial}{\partial\mathbf{n}}\colon H^1(\Delta,\D^\cc)\to H^{-\half}(\Gamma),
\end{equation}
where the spaces $H^1_\loc(\D^\cc)$ and $H^1(\Delta,\D^\cc)$ are defined
in \cite[Definition 2.6.1 ]{SS10} and \cite[Equation (2.108)]{SS10}, respectively.
These operators are continuous as stated in \cite[Theorem 2.6.8 and Theorem 2.7.7]{SS10}.
For a Banach space $X$, we denote by $\mathscr{L}(X)$ the space of bounded linear operators 
from $X$ into itself equipped with the standard operator norm, which we denote by
$\norm{\cdot}_{\mathscr{L}(X)}$.
Let $f(s)$ be a $\ell$-times continuously differentiable function depending on
a single variable $s$. We denote by $f'(s)$ the first derivative of $f(s)$ with respect
to $s$. 
For higher-order derivatives, 
we use both the notation $f^{(\ell)}(s)$ and $D^\ell_s f(s)$
to denote the $\ell$-th order differential of $f(s)$ with respect to the argument $s$.
\section{High-Frequency Acoustic Scattering Formulation}
\label{sec:problem_model}
\subsection{Sound-Soft Scattering Problem}
\label{sec:SoundSftScat}
Let $\D\subset \mathbb{R}^2$ be a bounded domain with 
smooth boundary $\Gamma\coloneqq\partial \D$.
We denote by $\D^\cc\coloneqq \mathbb{R}^2 \backslash \overline{\D}$ 
the corresponding exterior domain. 
Given $\kappa>0$ and $\hat{{\bf d}} \in \mathbb{R}^2$ a unit vector, 
we define 
$u^\text{i}(\bx)\coloneqq\exp(\imath \kappa \bx \cdot \hat{{\bf d}})$, $\bx \in \mathbb{R}^2$,
and consider the following problem: 
Find the scattered field $u^\text{s}\in H^1_\loc(\D^\cc)$ 
such that the total field 
$u \coloneqq u^\text{i}+ u^\text{s}$ satisfies
\begin{subequations}\label{eq:sound_soft_problem}
\begin{align}
	\Delta u + \kappa^2 u &=0 \quad \text{in } \D^\cc,  \quad \text{and}\label{eq:HHeqn} 
        \\
	\gamma^\cc u &= 0 \quad \text{on } \Gamma, \label{eq:SoundSftBc}
\end{align}
\end{subequations}
In addition to \eqref{eq:HHeqn} and \eqref{eq:SoundSftBc}, as it is customary,
we impose the Sommerfeld radiation condition on the scattered
field $u^{\text{s}}$
\begin{align}\label{eq:Sommerf}
	\frac{\partial u^{\text{s}}}{\partial r}(\bx) - \imath \kappa u^{\text{s}}(\bx) 
	= 
	o
	\left(
		\norm{\bx}^{-\half}
	\right)
\end{align}
as $\norm{\bx}\rightarrow \infty$, uniformly in $\hat{\bx} \coloneqq \bx/\norm{\bx}$. 

\subsection{Boundary Integral Formulation}
\label{ssec:BIF}
We reduce the exterior acoustic scattering problem \eqref{eq:sound_soft_problem}
to an integral equation on the boundary $\Gamma = \partial \D$. 
In doing so, we satisfy exactly the radiation condition \eqref{eq:Sommerf} 
and the homogeneous PDE \eqref{eq:HHeqn}.

Let $\text{G}_{\kappa}(\bz)\coloneqq\frac{\imath}{4}H^{(1)}_0(\kappa \norm{\bz})$ 
for $\bz \in \R^2 \backslash \{{\bf 0}\}$ denote the fundamental solution of
the Helmholtz equation in two spatial dimensions and of wavenumber $\kappa \in  \R_+$.
Here $H^{(1)}_0$ corresponds to the Hankel function of the first kind and of order zero.
We proceed to recast \eqref{eq:sound_soft_problem} using BIOs by means of boundary potentials. 
To this end, we introduce 
{\it single layer potential}
\begin{align}
	\left(
		\mathcal{S}_\kappa\, \psi
	\right)(\bx) 
	\coloneqq
	\int\limits_{\Gamma}
	\text{G}_\kappa(\bx-\by) \psi(\by) 
	\text{d}s_{\by}, 
	\quad \bx \in \D^\cc,
\label{eq:SLP}
\end{align}
and the {\it double layer potential}
\begin{align}\label{eq:DLP}
	\left(
		\mathcal{D}_{\kappa}
		\phi
	\right)
	(\bx)
	\coloneqq
	\int\limits_\Gamma
	\mathbf{n}(\by)
	\cdot
  	\nabla_{\by}
	\text{G}_\kappa(\bx-\by)
	\phi(\by)
	\text{d}\text{s}_\by,
	\quad
	\bx
	\in 
	\D^\cc,
\end{align}
where $\psi,\phi \in L^1(\Gamma)$ are referred to as \emph{surface densities},
which are defined on the boundary $\Gamma$.
The application of exterior Dirichlet and Neumann traces to 
we obtain the single layer operator
\begin{align}\label{eq:SLO}
	\OV_{\kappa}
	\coloneqq
	\gamma^\cc
	\mathcal{S}_{\kappa}
	\colon H^{-\half}(\Gamma)\to H^{\half}(\Gamma),
\end{align}
the double layer operator
\begin{align}\label{eq:DLO}
	\half \mathsf{Id}
	+
	\OK_{\kappa}
	\coloneqq
	\gamma^\cc\mathcal{D}_{\kappa}
	\colon 
	H^{\half}(\Gamma)\to H^{\half}(\Gamma),
\end{align}
and the adjoint double layer operator
\begin{align}\label{eq:ALO}
	\half \mathsf{Id}-\OK'_{\kappa}
	\coloneqq
	\frac{\partial}{\partial\mathbf{n}}
	\mathcal{S}_{\kappa}
	\colon H^{-\half}(\Gamma)\to H^{-\half}(\Gamma).
\end{align}
The double layer operator and the adjoint double layer operator
are 
mutually adjoint 
in the bilinear $L^2(\Gamma)$-duality pairing (e.g. \cite{SS10}).

The starting point of the reformulation of \eqref{eq:sound_soft_problem} is Green's representation formula, 
which provides the following representation of the total field $u$: 
\begin{align}\label{eq:green_rep_formula}
	u(\bx)
	=
	u^\text{i}(\bx)
	-
	\mathcal{S}_{\kappa}
	\left(
		\frac{\partial u}{\partial\mathbf{n}}
	\right)(\bx),
	\quad
	\bx \in \D^\cc,
\end{align}
where $\frac{\partial u}{\partial\mathbf{n}}$ corresponds to the 
Neumann trace of the scattered field $u$.
We remark that Sommerfeld radiation conditions are built-in in the single layer potential.
The application of both Dirichlet and Neumann traces to 
\eqref{eq:green_rep_formula} yields the BIEs
\begin{align}
	\OV_\kappa  \frac{\partial u}{\partial\mathbf{n}} 
	&= 
	\gamma^\cc u^\text{i},\quad \text{and}, \label{eq:IE1}\\
	\left(\half \mathsf{I}+ \mathsf{K}'_\kappa \right)\frac{\partial u}{\partial\mathbf{n}}
	&= 
	\frac{\partial u^\text{i}}{\partial\mathbf{n}}\label{eq:IE2}.
\end{align}
Both \eqref{eq:IE1} and \eqref{eq:IE2} can be proven to be well--posed,
except for those $\kappa$ that are eigenvalues of the Laplace operator
in $\D$ with Dirichlet and Neumann boundary conditions, respectively.
A linear combination of \eqref{eq:IE1} and \eqref{eq:IE2}
leads to the resonance-free, well-posed BIE
\begin{align}\label{eq:cfie_eq}
	\mathsf{A}'_{\kappa,\eta} 
	\varphi_\kappa 
	= 
	\frac{\partial u^\text{i}}{\partial\mathbf{n}} - \imath \eta \gamma^\cc u^\text{i}
        \eqcolon f_{\kappa,\eta} \in L^2(\Gamma),
\end{align}
where $\varphi_\kappa \coloneqq \frac{\partial u}{\partial\mathbf{n}} $, and
$\mathsf{A}_{\kappa,\eta}$ and its adjoint are defined as
\begin{align}\label{eq:cfie}
	\mathsf{A}_{\kappa,\eta}
	\coloneqq
	\half \mathsf{I}+ \mathsf{K}_\kappa  - \imath \eta \OV_\kappa
	\quad
	\text{and}
	\quad
	\mathsf{A}'_{\kappa,\eta}
	\coloneqq
	\half \mathsf{I}+ \mathsf{K}'_\kappa  - \imath \eta \OV_\kappa,
\end{align}
respectively.
Here $\eta \in \mathbb{R} \backslash \{0\}$ is a coupling parameter to be specified.
It is well established that $\frac{\partial u}{\partial\mathbf{n}} \in H^{-\half}(\Gamma)$. 
However, if $\D$ is a domain of class $\mathscr{C}^2$ standard regularity results assert
that $\frac{\partial u}{\partial\mathbf{n}}\in L^2(\Gamma)$. 
Indeed, this also holds true even if $\D$ is a Lipschitz domain
(see e.g. \cite[Section 5.1.2]{Nec1967}). 
Therefore, as it is customary, we consider \eqref{eq:cfie_eq} as an operator equation posed in $L^2(\Gamma)$.
\subsection{{Well-posedness of the CFIE and $\kappa$-Explicit Bounds}}
\label{sec:CondCFIE}
The following two propositions address the boundedness and coercivity
of $\mathsf{A}_{\kappa,\eta}$ and its adjoint, with bounds that are 
explicit in the wavenumber.

\begin{proposition}[{\cite[Theorem 3.6]{CGL09}}]
\label{thm:continuity_cfie}
It holds
\begin{align}
	\norm{\mathsf{A}_{\kappa,\eta}}_{\mathscr{L}(L^2(\Gamma))}
	=
	\norm{\mathsf{A}'_{\kappa,\eta}}_{\mathscr{L}(L^2(\Gamma))}
	\lesssim
	1+\kappa^{\frac{1}{2}} + \snorm{\eta} \kappa^{-\frac{1}{2}},
\end{align}
where the implied constant is independent of $\eta$ and $\kappa$.
\end{proposition}

\begin{proposition}[{\cite[Theorem 1.2]{SKS15}}]
\label{thm:coercivity_cfie}
Let $\D$ be a convex domain in either two dimensions or three dimensions whose 
boundary, $\Gamma$, has a strictly positive curvature and is 
both $\mathscr{C}^3$ and piecewise analytic. 
Then there exists a constant $\eta_0>0$ such that, given $\delta \in (0, 1/2)$, 
there exists $\kappa_0$ (depending on $\delta$) such that 
for $\kappa \geq \kappa_0$ and $\eta\geq \eta_0 \kappa$, 
\begin{equation}
	\mathfrak{Re}
	\left \{ 
		\dotp{\mathsf{A}'_{\kappa,\eta} \phi}{\phi}_{L^2(\Gamma)}
	\right\}
	\geq
	\left( \frac{1}{2}-\delta\right) 
	\norm{\phi}^2_{L^2(\Gamma)}
\end{equation}
for all $\phi \in L^2(\Gamma)$. 
The bound also holds with $\mathsf{A}'_{\kappa,\eta}$
replaced by $\mathsf{A}_{\kappa,\eta}$. 
\end{proposition}

As a consequence of Propositions \ref{thm:continuity_cfie} and 
\ref{thm:coercivity_cfie}, together with the Lax-Milgram lemma, we may conclude
well-posedness of the CFIE, stated before in \eqref{eq:cfie_eq}, 
as an operator equation set in $L^2(\Gamma)$.
In addition, as thoroughly discussed in \cite[Section 1.3]{SKS15}, 
there holds
$\kappa$-explicit quasi-optimality of the Galerkin discretization 
method. 
More precisely, given a finite-dimensional subspace
$V_N$ of $L^2(\Gamma)$, we consider the following Galerkin
approximation of the variational form of \eqref{eq:cfie_eq}:
Find ${\varphi}_{\kappa,N} \in V_N$ such that
\begin{equation}\label{eq:cfieN}
	\dotp{
		\mathsf{A}'_{\kappa,\eta} {\varphi}_{\kappa,N}
	}{
		\psi_N
	}_{L^2(\Gamma)}
	=
	\dotp{
		f_{\kappa,\eta} 
	}{
		\psi_N
	}_{L^2(\Gamma)}
	\quad
	\forall \psi_N \in V_N,
\end{equation}
with $f_{\kappa,\eta} $ as in the right-hand side of \eqref{eq:cfie_eq}.
Quasi-optimality of the Galerkin BEM reads as follows: 
Provided that $\eta$ is chosen such that
$\eta_0 \kappa \leq \eta \lesssim \kappa$ with $\eta_0$
and $\kappa_0$ as in Proposition \ref{thm:coercivity_cfie},
then for all $\kappa \geq \kappa_0$
it holds
\begin{equation}\label{eq:quasi_opt_galerkin_BEM}
	\norm{
		{\varphi}_{\kappa,N}
		-
		{\varphi}_{\kappa}
	}_{L^2(\Gamma)}
	\lesssim
	\kappa^{\frac{d-1}{2}}
	\inf_{{\phi}_{N}\in V_N}
	\norm{
		{\varphi}_{\kappa,N}
		-
		{\phi}_{N}
	}_{L^2(\Gamma)},
\end{equation}
where $d\in \mathbb{N}$ corresponds to the problem's
physical dimension.
\subsection{Galerkin vs DNN Approximations}
\label{sec:GalDNN}
In view of \eqref{eq:quasi_opt_galerkin_BEM},
and as it is usual in the Galerkin BEM, the convergence analysis of \eqref{eq:cfieN}
boils down to the choices of the subspaces $V_N$ and the regularity
of the solution $\varphi_\kappa$ to the CFIE \eqref{eq:cfie_eq}.
Concerning the latter aspect, in Section \ref{sec:AsTotField} ahead we recall
relevant results addressing the asymptotic regularity of $\varphi_\kappa$
in \eqref{eq:cfie}, which are explicit in the wavenumber $\kappa$. 

These results have been used in \cite{EcevitBoubAn22} 
and the references therein to design \emph{solution-adapted
subspaces} $V_{N}$ affording wavenumber-robust spectral convergence
in the approximation of the discrete Galerkin BEM problem stated in \eqref{eq:cfieN}.
The subspaces $V_N$ obtained in \cite{EcevitBoubAn22} depend on the wavenumber $\kappa$,
the incident angle $\hat{\bf d}$, and on $\Gamma$ itself in a nontrivial way.

In the present work, we use wavenumber-explicit asymptotic regularity
for the solution $\varphi_\kappa$ of the CFIE \eqref{eq:cfie_eq}
in our deep ReLU-NN emulation construction.
The NN architecture in the present paper is a deep feedforward NN 
with strict ReLU activation and \emph{fixed width}.

Contrary to the construction in \cite{EcevitBoubAn22}, 
the presently obtained DNN emulation 
results indicate no need to include 
explicit asymptotic information of $\varphi_\kappa$ into the DNN
architecture, while enjoying, as we shall show, similar convergence properties.
In practice, NN training would account via feature optimization
for the dependence of the approximation on the problem's wavenumber and 
on the scatterer's geometry.
As usual in many DNN based PDE solvers,
this comes with the price of a nonlinear dependence of the 
DNN approximation on its constitutive parameters. 
In practice, for the computation of these parameters, iterative
training algorithms are required, as opposed to the solution of 
the Galerkin approximation \eqref{eq:cfieN}, which 
merely requires
assembling and solution of a linear system of equations.

\subsection{Asymptotics of the normal derivative of the total field}
\label{sec:AsTotField}
We assume that the scatterer's boundary $\Gamma$ is given by
$$
	\Gamma
	\coloneqq
	\left \{
		\bx \in \mathbb{R}^2:
		\bx  = \boldsymbol{\gamma}(s) :\; s\in [0,2\pi]
	\right \},
$$
where $\boldsymbol{\gamma} :[0,2\pi] \rightarrow \mathbb{R}^2$ is a $2\pi$--periodic
parametrization of $\Gamma$.  
Following \cite[Section 5]{DGS07} (which in turn relies 
on the work by Melrose and Taylor \cite{MelTaylor85},
and references therein), 
the solution $\varphi_\kappa \in L^2(\Gamma)$ to \eqref{eq:cfie_eq}
admits the following decomposition
\begin{equation}\label{eq:decomp_neumann_trace}
	\hat{\varphi}_{\kappa}(s) 
	\coloneqq 
	\left(\tau_{\boldsymbol{\gamma}} \varphi_\kappa \right)(s)
	=
	\kappa V_\kappa(s) \exp(\imath \kappa \boldsymbol{\gamma}(s)\cdot \hat{{\bf d}}),
	\quad
	s \in [0,2\pi],
\end{equation}
where $\tau_{\boldsymbol{\gamma}}:L^2(\Gamma) \rightarrow L^2((0,2\pi)): u \mapsto u \circ \boldsymbol{\gamma}$
defines a bounded linear operator with a bounded inverse. 

Throughout this manuscript, we work under the assumption stated below.
\begin{assumption}\label{assumption:smooth_curve}
We require $\Gamma$ to be $\mathscr{C}^\infty$,
simply connected and strictly convex, i.e.~with non-vanishing curvature.
\end{assumption}

When Assumption \ref{assumption:smooth_curve} is satisfied,
the following result provides an asymptotic expansion 
of $V_\kappa$ that is explicit in the wavenumber.

\begin{figure}
     \centering
     \begin{subfigure}[b]{0.45\textwidth}
         \centering
         \includegraphics[width=0.8\textwidth]{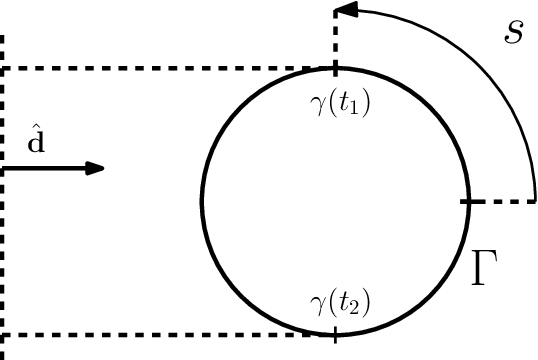}
         \caption{Geometrical setting.}
         \label{fig:geometry}
     \end{subfigure}
     \hfill
     \begin{subfigure}[b]{0.45\textwidth}
         \centering
         \includegraphics[width=0.8\textwidth]{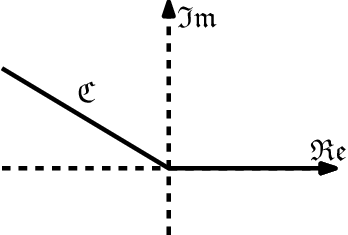}
         \caption{Complex integration contour used in the definition of Fock's integral.}
         \label{fig:contour}
     \end{subfigure}
        \caption{Figure \ref{fig:geometry} portrays the geometrical setting under consideration, whereas
        Figure \ref{fig:contour} depicts the complex integration contour used in the definition of Fock's
        integral state in \eqref{eq:fock_integral}.}
        \label{fig:geometry_contour}
\end{figure}

\begin{theorem}[{\cite[Theorem 5.1]{DGS07}},{\cite[Theorem 9.27]{MelTaylor85}}]
\label{thm:expansion_V}
Let Assumption \ref{assumption:smooth_curve} be satisfied.
There exists $\delta>0$ such that $V_\kappa$ has the asymptotic expansion:
\begin{align}\label{eq:decomp_V_frequency}
	V_\kappa(s) 
	\sim
	\sum_{\ell,m\geq 0}
	\kappa^{-\frac{1}{3}-\frac{2}{3}\ell-m} 
	b_{\ell,m}(s)
	\Psi^{(\ell)}\left(\kappa^{\frac{1}{3}}Z(s)\right)
\end{align}
valid for 
$s\in I_\delta \coloneqq(t_1-\delta,t_1+\delta)\cup(t_2-\delta,t_2+\delta)$,
where $\boldsymbol{\gamma}(t_1)$ 
and 
$\boldsymbol{\gamma}(t_2)$ are the tangency points (see Figure \ref{fig:geometry}).
The functions $b_{\ell,m}$, $\Psi$ and $Z$ have the following properties:
\begin{itemize}
\item[(i)] $b_{\ell,m}$ are $\mathscr{C}^\infty$ complex-valued functions on $I_\delta$.
\item[(ii)] $Z$ is a $\mathscr{C}^\infty$ real-valued function on $I_\delta$ with simple zeros at $t_1$,
and $t_2$, which is positive--valued on $(t_1,t_2)\cap I_\delta$and negative-valued on $(t_2-2\pi)\cap I_\delta$.
\item[(iii)] $\Psi:\mathbb{C}\rightarrow \mathbb{C}$ is an entire function specified by
\begin{align}\label{eq:fock_integral}
\Psi(\tau)\coloneqq
\exp(-\imath \tau^3/3)\int\limits_{\mathfrak{C}} \frac{\exp(-\imath z \tau)}{\normalfont\text{Ai}(\exp(2\pi \imath/3)z)} \normalfont\text{d}z,
\end{align}
where $\normalfont\text{Ai}$ is the Airy function and $\mathfrak{C}$ is the contour depicted in 
Figure \ref{fig:contour}. 
This function is often referred to as “Fock’s integral”, 
see e.g. \cite{fok1965electromagnetic,buslaev1964short}.
\end{itemize}
The asymptotics of $\Psi(\tau)$ for large $\tau$ is known. 
We have that
\begin{align}\label{eq:asymptotic_psi_tau_large}
	\Psi(\tau)
	=
	\sum_{j=0}^{n}a_j \tau^{1-3j} + \mathcal{O}(\tau^{1-3(n+1)}),
	\quad
	\text{as}\quad \tau\rightarrow\infty,
\end{align}
where $a_0 \neq 0$. 
This expansion remains valid for all derivatives of $\Psi$ by formally differentiating
each term on the right hand side, including the error term. 
Furthermore, there exists $\beta>0$ and $c_0 \neq 0$
such that for any $n\in \mathbb{N}_0$
\begin{align}
	D^n_\tau \Psi(\tau)
	=
	c_0 
	D^n_\tau 
	\left\{
		\exp(-\imath \tau^3/3-\imath \tau \alpha_1) 
	\right \}
	(1+ \mathcal{O}(\exp(-\beta \snorm{\tau})),
	\quad
	\text{as}\quad \tau\rightarrow-\infty,
\end{align}
where $\alpha_1=\exp(-2\pi/3 \imath)\nu_1$ and $\nu_1\in \mathbb{R}_{-}$ is the right-most root of $\normalfont\text{Ai}$. 
\end{theorem}

\begin{remark}\label{rmk:periodicity_Z_symbol}
We make the following observations concerning Theorem \ref{thm:expansion_V}.
\begin{itemize}
	\item[(i)]
	In \eqref{eq:decomp_V_frequency}, the symbol ``$\sim$''
	refers to the symbol classes of H\"{o}rmander. 
	For details we refer to  \cite[p. 236, Definition 7.8.1]{hormander2015analysis}, 
	and we do not elaborate on this any further.
	\item[(ii)]
	The function $Z$ can be extended (non-uniquely) to a
	$\mathscr{C}^\infty$, $2\pi$--periodic function that 
	is positive-valued on $(t_1,t_2)$ and negative-valued
	on $(t_2-2\pi,t_2)$. 
	From now on we assume that this extension
	has been made.
\end{itemize}
\end{remark}

\begin{proposition}[{\cite[Corollary 5.3]{DGS07}}]
\label{prop:asymptotic_V_remainder}
The functions $b_{\ell,m}$ can be extended to $2\pi$--periodic $\mathscr{C}^\infty$
functions such that, for all $L$, $M\in \mathbb{N}_0$, the decomposition
\begin{align}\label{eq:decom_V_k}
	V_\kappa(s)
	=
	\sum_{\ell,m=0}^{L,M} \kappa^{-\frac{1}{3}-2\ell/3-m}b_{\ell,m}(s) \Psi^{(\ell)} \left(\kappa^\frac{1}{3} Z(s)\right)
	+
	R_{L,M,\kappa}(s)
\end{align}
holds for all $s\in [0,2\pi]$, with the remainder term satisfying, for all $n\in \mathbb{N}_0$,
\begin{align}\label{eq:residual_R}
	\snorm{D^n_sR_{L,M,\kappa}(s)}
	\leq
	C_{L,M,n} (1+\kappa)^{\mu+\frac{n}{3}},
\end{align}
where
\begin{align}\label{eq:residual_R_mu}
	\mu
	\coloneqq
	-\min \left\{\frac{2}{3}(L+1),(M+1) \right\},
\end{align}
and $C_{L,M,n}$ is independent of $\kappa$.
\end{proposition}

\begin{remark}
Proposition \ref{prop:asymptotic_V_remainder} and Remark \ref{rmk:periodicity_Z_symbol} allow 
us to conclude that the remainder $R_{L,M,\kappa}(s)$ can also be extend (non-uniquely) 
to a $\mathscr{C}^\infty$, $2\pi$--periodic function.
\end{remark}

\begin{remark}\label{rmk:bound_psi}
For $\tau \in \mathbb{R}$ it holds
\begin{equation}
	\snorm{\Psi(\tau)} \leq C_0(1+\snorm{\tau}),
	\quad
	\snorm{\Psi'(\tau)} \leq C_1,
	\quad
	\snorm{\Psi^{(\ell)}(\tau)} \leq C_\ell(1+\snorm{\tau})^{-2-\ell}, 
	\quad
	\text{for } \ell\geq2,
\end{equation}
with $C_\ell$ independent of $\kappa$.
\end{remark}

The next result provides 
wavenumber-explicit bounds for the derivatives of $V_\kappa$.

\begin{lemma}[{\cite[Corollary 5.5]{DGS07}}]
\label{eq:smoothness_neuman_trace}
For all $n\geq1$ there exists a constant $C_n>0$ independent of $\kappa$ such that, for $\kappa$
sufficiently large, $\snorm{D^n_s V_\kappa(s)} \leq C_n(1+\kappa)^{\frac{(n-1)}{3}}$, $s\in[0,2\pi]$.
\end{lemma}

\begin{remark}\label{rmk:correction_lemma_V_k}
As originally stated in \cite[Corollary 5.5]{DGS07}, 
Lemma \ref{eq:smoothness_neuman_trace} holds for $n\in \mathbb{N}$.
An inspection of \cite[Corollary 5.3]{DGS07} reveals that 
one has
$\snorm{V_\kappa(s)} \leq C$,  $s\in[0,2\pi]$, 
for some constant $C>0$ independent of $\kappa$.
\end{remark}

\section{Deep ReLU Neural Networks}
\label{sec:NN}
We recall notation, terminology, and results concerning feedforward, fully connected
DNNs, as subsequently required, consistent with \cite{elbrachter2021deep}.

\begin{definition}[Deep Neural Network]
\label{def:DNN}
Let $L \in \IN$, $\ell_0,\dots, \ell_L \in \IN$
and let
$\varrho:\mathbb{R} \rightarrow\mathbb{R}$ be a
non-linear function, referred to in the following as
the \emph{activation function}.
Given $({\bf W}_k,{\bf b}_k)_{k=1}^{L}$,
${\bf W}_k\in\mathbb{R}^{\ell_k\times\ell_{k-1}}$, ${\bf b}_k\in\mathbb{R}^{\ell _k}$,
define the affine transformation
 $\mathcal{A}_k: \mathbb{R}^{\ell_{k-1}}\rightarrow \mathbb{R}^{\ell_k}: 
 {\bf x} \mapsto {\bf W}_k {\bf x}+{\bf b}_k$
for $k\in\{1,\ldots,L\}$. We define a \emph{Deep Neural Network (DNN)}
with activation function $\varrho:\mathbb{R}\rightarrow \mathbb{R}$ as a map
$\Phi_{\mathcal{N\!N}}: \mathbb{R}^{\ell_0}\rightarrow \mathbb{R}^{\ell_L}$ with
\begin{align}\label{eq:ann_def}
	\Phi_{\mathcal{N\!N}}({\bf x})
	\coloneqq
	\begin{cases}
	\mathcal{A}_1(\bx), & L=1, \\
	\left(
		\mathcal{A}_L
		\circ
		\varrho
		\circ
		\mathcal{A}_{L-1}
		\circ
		\varrho
		\cdots
		\circ
		\varrho
		\circ
		\mathcal{A}_1
	\right)(\bx),
	& L\geq2,
	\end{cases}
\end{align}
where the activation function $\varrho$ is applied component-wise to vector inputs.
We define the depth and the width
of an DNN as
\begin{equation}
	\mathcal{M}(\Phi_{\mathcal{N\!N}})\coloneqq\max\{\ell_1,\ldots, \ell_L\}
	\quad
	\text{and}
	\quad
	\mathcal{L}(\Phi_{\mathcal{N\!N}})\coloneqq L
\end{equation}
together with the weight bounds
\begin{equation}
	\mathcal{B}
	\left(
		\Phi_{\mathcal{N\!N}}
	\right)
	=
	\max_{\ell =1,\dots,L}
	\max
	\left\{
		\norm{{\bf W}_k}_{\infty},
		\norm{{\bf b}_k}_{\infty}
	\right\}.
\end{equation}
In the above definition, $\ell_0$ is the dimension of the \emph{input layer},
$\ell_L$ denotes the dimension of the \emph{output layer}.

In addition, we denote by $\mathcal{N\!N}^{\varrho}_{L,M,\ell_0,\ell_L}$
the set of all DNNs $\Phi_{\mathcal{N\!N}}:\mathbb{R}^{\ell_0}\rightarrow \mathbb{R}^{\ell_L}$ 
with input dimension $\ell_0$, output dimension $\ell_L$, a depth of at most $L$ layers,
maximum width $M$, and activation function $\varrho$.
\end{definition}

Unless explicitly stated otherwise, we assume that
the activation is by the ReLU defined as
$\varrho(x)\coloneqq \max\{ x,0 \}$, and we drop the explicit
dependence on the activation function in $\mathcal{N\!N}^{\varrho}_{L,M,\ell_0,\ell_L}$.

We proceed to recall results concerning the
approximation properties of ReLU-NNs following \cite{elbrachter2021deep}.
Firstly, we introduce the multiplication ReLU-NN.

\begin{proposition}[{\cite[Proposition III.3]{elbrachter2021deep}}]
\label{prop:mult_network}
There exists a constant $C>0$ such that for all $D\in \mathbb{R_{+}}$
and $\epsilon \in (0,1/2)$ there are networks 
$\Phi_{D,\epsilon} \in\mathcal{N\!N}_{L,5,2,1}$ with
\begin{equation}
	L
	\leq 
	C
	\left(
	\log
	\left(
		\lceil D \rceil
	\right)
	+
	\log
	\left(
		\frac{1}{\epsilon}
	\right)
	\right)
\end{equation}
such that
$$
	\norm{\Phi_{D,\epsilon}(x,y)-xy}_{L^\infty((-D,D)^2)}
	\leq 
	\epsilon.
$$
The weights of $\Phi_{D,\epsilon}$ 
bounded according to $\mathcal{B}(\Phi_{D,\epsilon})\leq1$.
\end{proposition}

Observe that $p_{0}(x)=1$ and $p_1(x)=x$
can be exactly represented using ReLU-NNs as follows
\begin{subequations}
\begin{align}
	p_0(x) 
	&
	=
	\begin{pmatrix}
	1 & -1 & -1 & 1
	\end{pmatrix}
	\varrho
	\left(
	\begin{pmatrix}
		1 \\
		1 \\
		1 \\
		1
	\end{pmatrix}
	x
	+
	\begin{pmatrix}
		2  \\
		1  \\
		-1 \\
		-2
	\end{pmatrix}
	\right)
	\in \mathcal{N\!N}_{2,4,1,1},\label{eq:relu_nn_po_1}
	\quad
	\text{and} \\
	p_1(x)
	&
	= 
	\begin{pmatrix}
		1 & -1 \\
	\end{pmatrix}
	\varrho
	\begin{pmatrix}
		\begin{pmatrix}
		1 \\
		-1
		\end{pmatrix}
		x
	\end{pmatrix}
	\in \mathcal{N\!N}_{2,2,1,1}.
	\label{eq:relu_nn_po_2}
\end{align}
\end{subequations}

The next result addresses the ReLU-NN emulation 
of polynomials of higher order.


\begin{proposition}[{\cite[Proposition III.5]{elbrachter2021deep}}]
\label{prop:approximation_polynomials}
There exist a constant $C>0$
such that for $m\in \mathbb{N}$, $A\geq1$, $D\in \mathbb{R}_+$, $\epsilon \in (0,1/2)$
and
\begin{equation}
	p_m(x)
	=
	\sum_{i=0}^m a_i x^i
	\quad
	\text{with }
	\displaystyle\max_{i=0,\dots,m}\snorm{a_i} \leq A,
\end{equation}
there are ReLU-NNs $\Phi_{p_m,D,\epsilon} \in \mathcal{N\!N}_{L,9,1,1}$ with 
\begin{equation}\label{eq:bound_ReLU_NN_pol}
	L
	\leq 
	C
	\left(
		m\log
		\left(
			\frac{1}{\epsilon}
		\right)
		+
		m^2 
		\log(\lceil D \rceil)+m\log(m)+m\log(A)
	\right)
\end{equation}
such that 
\begin{align}
	\norm{\Phi_{p_m,D,\epsilon}(x)-p_m(x)}_{L^\infty((-D,D))}
	\leq 
	\epsilon.
\end{align}
The weights of $\Phi_{p_m,D,\epsilon}$ are bounded according to
$\mathcal{B}(\Phi_{p_m,D,\epsilon})\leq1$.
\end{proposition}

\begin{remark}[ReLU-NN Emulation of Polynomials with Complex Coefficients]
\label{eq:remark_complex_coeff}
An inspection of the proof of Proposition \ref{prop:approximation_polynomials}
reveals that in order to accommodate for polynomials with complex coefficients 
the architecture of the ReLU-NN should be modified as follows:
two outputs should be considered instead of one in order to account for the real and imaginary
parts separately, and the width of the ReLU-NN should be increased to $11$.
The depth of the ReLU-NN remains unchanged, 
i.e.~it behaves exactly as in \eqref{eq:bound_ReLU_NN_pol}.
\end{remark}


\begin{proposition}[{\cite[Theorem III.8 and Corollary III.9]{elbrachter2021deep}}]
\label{prop:aprox_sc}
There exists a constant $C>0$ such for each $a,D \in \mathbb{R}_+$,
$\epsilon \in (0,1/2)$ there are ReLU-NNs
$\Phi^{\cos}_{a,D,\epsilon} \in \mathcal{N\!N}_{L,9,1,1}$ 
and $\Phi^{\sin}_{a,D,\epsilon} \in \mathcal{N\!N}_{L,9,1,1}$
with 
\begin{equation}
	L
	\leq 
	C
	\left(
	\left(
		\log
		\left(
			\frac{1}{\epsilon}
		\right)
	\right)^2
	+
	\log
	\left(
		\lceil aD\rceil
	\right)
	\right)
\end{equation}
such that
\begin{align}
	\norm{\cos(a x)-\Phi^{\cos}_{a,\epsilon}(x)}_{L^\infty((-D,D))} 
	\leq 
	\epsilon
	\quad
	\text{and}
	\quad
	\norm{\sin(a x)-\Phi^{\sin}_{a,\epsilon}(x)}_{L^\infty((-D,D))} 
	\leq 
	\epsilon.
\end{align}
The weights of $\Phi^{\cos}_{a,D,\epsilon}$ and $\Phi^{\sin}_{a,D,\epsilon}$
are bounded according to
$\mathcal{B}(\Phi^{\cos}_{a,\epsilon})\leq 1$ and $\mathcal{B}(\Phi^{\sin}_{a,\epsilon}) \leq 1$.
\end{proposition}

\subsection{ReLU-NN Emulation of Decaying Exponentials}
\label{ssec:approx_decay_exp}
We approximate decaying exponentials, 
i.e.~maps of the form $x\mapsto \exp(-a x)$ with $a\in\mathbb{R}_+$,
using ReLU-NNs within the interval $[0,D]$, where $D\in \mathbb{R}_+$.

\begin{proposition}\label{prop:approx_exp_NN}
There exist constants $C_1$, $C_2>0$  
such that for each $a$, $D \in \mathbb{R}_+$
there exist ReLU-NNs 
$\Phi^{\exp}_{a,D,\epsilon} \in \mathcal{N\!N}_{L,9,1,1}$ 
with 
\begin{align}
	L
	\leq
	C_1
	\left(
		\lceil\log(aD)\rceil^2
		+
		\left(\log\left(\frac{1}{\epsilon}\right)\right)^2
	\right),
	\quad
	\text{as }
	\epsilon \rightarrow 0,
\end{align}
such that $\norm{ \exp(-a x)- \Phi^{\exp}_{a,D,\epsilon}(x)}_{L^\infty((0,D))}\leq \epsilon$.
The weights of $\Phi^{\exp}_{a,D,\epsilon}$ are bounded in absolute value by a constant 
independent of $a$ and $D$.
\end{proposition}

\begin{proof}
We divide the proof in several steps.
\\

{\sf \encircle{1} ReLU-NN Emulation for $a\in(0,1]$ and $D\in (0,1]$.}
Firstly, let us consider the case $a\in (0,1]$ and $D\in(0,1]$.
The MacLaurin series expansion of $\exp(-ax)$ reads
\begin{align}
	\exp(-ax)
	=
	\sum_{n=0}^{\infty} \frac{(-ax)^n}{n!},
	\quad 
	x\in \mathbb{R}.
\end{align}
Taylor's theorem with Lagrange remainder yields
\begin{align}
	\exp(-ax)
	=
	\sum_{n=0}^{N} \frac{(-ax)^n}{n!}
	+
	\frac{(-a)^{N+1}}{(N+1)!} \exp(-a \xi_x)x^{N+1},
	\quad
	x\in \mathbb{R},
\end{align}
for some $\xi_x \in (0,x)$ and $N\in \mathbb{N}$. 
For $N = \lceil (\log(2/\epsilon))\rceil$ it holds
\begin{equation}
\begin{aligned}
	\norm{
		\exp(-ax) 
		- 
		\sum_{n=0}^{N} \frac{(-ax)^n}{n!}
	}_{L^\infty((0,D))}
	&
	\leq
	\norm{\frac{(-a)^{N+1}}{(N+1)!} \exp(-a \xi_x)x^{N+1}}_{L^\infty((0,D))} 
	\\
	&
	\leq 
	\frac{1}{(N+1)!}
	\leq
	\frac{\epsilon}{2}.
\end{aligned}
\end{equation}
Thanks to Proposition \ref{prop:approximation_polynomials}, 
there exists a constant $C>0$ such that for each $a\in (0,1]$
and $\epsilon \in (0,1/2)$ there are
ReLU-NNs $\Phi_{a,\epsilon} \in \mathcal{N\!N}_{L,9,1,1}$
with 
$$
	L
	\leq 
	C
	\left(
		N\log\left(\frac{2}{\epsilon}\right)+N\log(N)+N
	\right),
$$
satisfying 
\begin{align}
	\norm{
		\Phi_{a,\epsilon}(x)
		-
		\sum_{n=0}^{N} \frac{(-ax)^n}{n!}
	}_{L^\infty((0,D))}
	\leq 
	\frac{\epsilon}{2}.
\end{align}
Therefore, there exists a constant $C>0$ such that for each $a\in(0,1]$ 
there are ReLU-NNs $\Phi_{a,\epsilon} \in \mathcal{N\!N}_{L,9,1,1}$ with
$L \leq C \left(\log(1/\epsilon)\right)^2$, as $\epsilon \rightarrow 0$, satisfying
$\norm{\exp(-a x)-\Phi_{a,\epsilon}(x)}_{L^\infty((0,1))}\leq \epsilon$
with weights bounded according to $\mathcal{B}(\Phi_{a,\epsilon}) \leq 1$.
\\

{\sf \encircle{2} ReLU-NN Emulation of the Squaring Operation.}
Let us shortly recall Yarotsky's result concerning 
the approximation of the squaring operation using ReLU-NNs \cite{Yar17}. 

We denote by $\Phi^\text{S}_\epsilon$ the linear piecewise
interpolation that approximates uniformly $x \mapsto x^2$ up to
accuracy $\epsilon$ in the interval $[0,1]$.
As stated in \cite[Proposition III.2]{elbrachter2021deep},
we have that $\Phi^\text{S}_\epsilon\in \mathcal{N\!N}_{L,3,1,1}$,
with $L \leq C \log\left(1/ \epsilon\right)$ for any $\epsilon \in (0,1/2)$
satisfying
\begin{align}
\label{eq:error_squaring_network}
	\norm{
		\Phi^\text{S}_\epsilon(x)
		-
		x^2
	}_{L^\infty((0,1))}
	\leq
	\epsilon.
\end{align}
An inspection of Yarotsky's construction reveals the following Lipschitz estimate
\begin{align}
\label{eq:lipschitz_cont_nn}
	\snorm{\Phi^\text{S}_\epsilon(x)- \Phi^\text{S}_\epsilon(y)}
	\leq
	2 \snorm{x-y},
	\quad
	x,y\in \mathbb{R}.
\end{align}
The weights are bounded according to $\mathcal{B}(\Phi^\text{S}_\epsilon) \leq 1$. 
\\

{\sf \encircle{3} ReLU-NN Emulation for $a>1$ and $D \in(0,1]$.}
For each $a\in \mathbb{R}_+$, there exists a $C_a\in (1/2,1]$ such that $a = C_a 2^{\lceil\log(a) \rceil}$.
Let us define $f_{n,a}(x)\coloneqq\exp(-2^{n-1}(C_a x))$.
Recalling {\sf \encircle{1}}, there exists a constant
$C>0$ and ReLU-NNs $\Phi_{a,\zeta} \in \mathcal{N\!N}_{L,9,1,1}$ 
with $L \leq C \left(\log(1/\zeta)\right)^2$ for any $\zeta \in (0,1/2)$ 
satisfying
\begin{align}
	\norm{
		f_{1,a}
		-
		\Phi_{a,\zeta}
	}_{L^\infty((0,D))}
	\leq 
	\zeta.
\end{align}
We define recursively $\Phi_{n+1,a,\zeta} \coloneqq \Phi^\text{S}_\zeta\circ \Phi_{n,a,\zeta}$, 
where $\Phi^\text{S}_\zeta\in \mathcal{N\!N}_{L,3,1,1}$, with $L \leq C \log\left(1/ \zeta\right)$
for any $\zeta \in (0,1/2)$ is the ReLU-NN discussed in {\sf \encircle{2}}. 

Next, we prove inductively that for each $n\in \mathbb{N}$ it holds
\begin{align}
	\label{eq:approx_exp_ind}
	\norm{f_{n,a}-\Phi_{n,a,\zeta}}_{L^\infty((0,D))}
	\leq
	2^{2(n-1)} \zeta.
\end{align} 
The statement for $n=1$ has been proved in {\sf \encircle{1}}.
Let us estimate
\begin{equation}
\begin{aligned}
	\snorm{\Phi_{n+1,a,\zeta}(x) - f_{n+1,a}(x)}
	\leq
	&
	\snorm{
		\left(
			\Phi^\text{S}_\zeta\circ \Phi_{n,a,\zeta}
		\right)(x)
		- 
		\left(
			\Phi^\text{S}_\zeta \circ f_{n,a}
		\right)(x)
	}\\
	&
	+
	\snorm{
		\left(
			\Phi^\text{S}_\zeta \circ f_{n,a}
		\right)(x)
		-
		f_{n+1,a}(x)
	},
	\quad
	x\in [0,D].
\end{aligned}
\end{equation}
Recalling \eqref{eq:lipschitz_cont_nn}, we have that
\begin{align}
	\snorm{
		\left(
			\Phi^\text{S}_\zeta\circ\Phi_{n,a,\zeta}
		\right)(x)
		- 
		\left(
			\Phi^\text{S}_\zeta \circ f_{n,a}
		\right)(x)
	}
	\leq
	2\snorm{\Phi_{n,a,\zeta}(x)- f_{n,a}(x)},
	\quad
	x\in[0,D],
\end{align}
therefore, assuming that \eqref{eq:approx_exp_ind} holds (induction hypothesis),
we have
\begin{align}
	\norm{
			\Phi^\text{S}_\zeta \circ\Phi_{n,a,\zeta}
		- 
			\Phi^\text{S}_\zeta\circ f_{n,a}
	}_{L^\infty((0,D))}
	\leq
	2^{2n-1} \zeta.
\end{align}
Observe that $f_{n+1,a} = \left(f_{n,a}\right)^2$ and that $f_{n,a}(x)\in [0,1]$, for $x\in [0,1]$ and for all $n\in \mathbb{N}$.
Hence, we obtain
\begin{align}
	\norm{
			\Phi^\text{S}_\zeta\circ f_{n,a}
		-
		f_{n+1,a}
	}_{L^\infty((0,D))}
	\leq
	\zeta
	\quad
	\text{and}
	\quad
	\norm{\Phi_{n+1,a,\zeta}- f_{n+1,a}}_{L^\infty((0,D))}
	\leq
	2^{2n}\zeta,
\end{align}
thus proving the desired result.
Let us set $\epsilon=2^{2n}\zeta$ and $n= {\lceil\log(a) \rceil}-1$. Furthermore, we define
\begin{align}
	\omega_{a,\epsilon}
	\coloneqq
	\underbrace{
		\Phi^\text{S}_\zeta
		\circ 
		\Phi^\text{S}_\zeta
		\circ
		\Phi^\text{S}_\zeta
		\circ
		\cdots
		\circ
		\Phi^\text{S}_\zeta
	}_{{\lceil\log(a) \rceil} \text{ times}}
	\circ
	\Phi_{1,a,\zeta}.
\end{align}
Observe that $\omega_{a,\epsilon} \coloneqq \Phi_{n+1,a,\zeta}$ and
$\norm{\Phi_{n+1,a,\epsilon}}_{L^\infty((0,D))} \leq 1+ 2^{2n}\zeta$. 
The depth of $\omega_{a,\epsilon}$ is
\begin{equation}
\begin{aligned}
	\mathcal{L}\left( \omega_{a,\epsilon}\right)
	\leq
	&
	C
	\left(
		\left(\log\left(\frac{1}{\zeta}\right)\right)^2
		+
		{\lceil\log(a) \rceil} \log(1+ 2^{2n}\zeta)
		+
		{\lceil\log(a) \rceil}\log\left(\frac{1}{\zeta}\right)
	\right) \\
	\leq
	&
	C
	\left(
		\left( 2(\lceil \log(a) \rceil -1)+ \log\left(\frac{1}{\epsilon}\right)\right)^2
		+
		{\lceil\log(a) \rceil} \log(1+ \epsilon)
	\right.
	\\
	&
	+
	\left.
		\left(
			2(\lceil\log(a) \rceil-1) 
			+
			\log\left(\frac{1}{\epsilon}\right)
		\right)
		{\lceil\log(a) \rceil} 
	\right)\\
	\leq
	&
	C
	\left(
		\lceil\log(a)\rceil^2
		+
		\left(\log\left(\frac{1}{\epsilon}\right)\right)^2
	\right)
	\quad
	\text{as }
	\epsilon\rightarrow 0.
\end{aligned}
\end{equation}
The weight of $\Phi^\text{S}_\epsilon $ and $\Phi_{a,\zeta}$ 
are bounded in absolute value by a constant (recall that $\epsilon=2^{2n}\zeta$ and $n= {\lceil\log(a) \rceil}-1$).
The weights of $ \omega_{a,\epsilon}$ are those of $\Phi^\text{S}_\epsilon $ and $\Phi_{a,\zeta}$, 
hence the weights are bounded according to $\mathcal{B}(\omega_{a,\epsilon}) \leq 1$ as well.
\\

{\sf \encircle{4} ReLU-NN Emulation for $a>1$ and $D>1$.}
Let $a$, $D>1$ and let $\omega_{a,D,\epsilon} \in  \mathcal{N\!N}_{L,9,1,1}$, 
with 
\begin{align}
	L
	\leq
	C
	\left(
		\lceil\log(aD)\rceil^2
		+
		\log^2\left(\frac{1}{\epsilon}\right)
	\right)
	\quad
	\text{as }
	\epsilon \rightarrow 0
\end{align}
be the ReLU-NN from {\sf \encircle{3}} approximating $\exp(-aDx)$ with
target accuracy $\epsilon>0$
over the interval $[0,1]$, i.e. 
$\norm{\omega_{a,D,\epsilon}(x)-\exp(-aDx)}_{L^\infty((0,1))}\leq \epsilon$.
Define 
$\Phi_{a,D,\epsilon}(x)\coloneqq\omega_{a,D,\epsilon}\left(\frac{x}{D}\right)$, 
for $x\in [0,D]$.
We have that
\begin{align}
	\norm{\Phi_{a,D,\epsilon}(x)-\exp(-ax)}_{L^\infty((0,D))}
	=
	\norm{\omega_{a,D,\epsilon}(x)-\exp(-aDx)}_{L^\infty((0,1))}
	\leq
	\epsilon.
\end{align}
Since $D>1$, the weights are bounded according to
$\mathcal{B}(\Phi_{a,D,\epsilon}) \leq 1$.
\end{proof}

\subsection{ReLU-NN Emulation of $x\mapsto \frac{1}{x}$.}
\label{sec:ReLUEm1/x}
We proceed to approximate the function $x\mapsto \frac{1}{x}$ using ReLU neural networks 
within the interval $[1,D]$ with $D>1$. The construction of these neural networks
is performed in two steps. Firstly, we approximate $x\mapsto\frac{1}{x}$, for $x\geq 1$, by an
appropriate combination of decaying exponentials. The key ingredient of this 
construction corresponds to the so-called ``sinc''-quadrature rule developed by 
F.~Stenger \cite{Stenger12}. 
The second step consists in approximating each of these exponentials using the 
results from subsection \ref{ssec:approx_decay_exp}. The ReLU-NN that 
approximates $x\mapsto\frac{1}{x}$ will be a suitable linear combination of the ReLU-NNs approximating 
each of the decaying exponentials.

\begin{lemma}[{{\cite[Lemma 4.2]{GHK03}}}]
\label{lmm:stenger_quad}
Let $z\in \mathbb{C}$ with $\Re\{z\}\geq 1$. Then for each $k\in \mathbb{N}$
the (positive) quadrature points and weights
\begin{align}
	t_j
	&
	\coloneqq
	\log\left(\exp\left(jk^{-\frac{1}{2}}\right) +\sqrt{1+\exp\left(2jk^{-\frac{1}{2}}\right)}\right), \\
	\omega_j
	&
	\coloneqq
	\left(k+k\exp\left(-2j k^{-\frac{1}{2}}\right)\right)^{-\frac{1}{2}},
	\quad
	j=-k,\dots,k,
\end{align}
fullfil 
\begin{align}
	\snorm{\int\limits_{0}^{\infty}\exp(-zt)\normalfont\text{d}t 
	- 
	\sum_{j=-k}^{k}\omega_j \exp(-t_j z)}
	\leq
	C_{S}
	\exp
	\left(
		\frac{\snorm{\Im\{z\}}}{\pi}
	\right)\exp(-\sqrt{k}),
\end{align}
where $C_S>0$ does not depend upon $k$ or $z$.
\end{lemma}

Equipped with Lemma \ref{lmm:stenger_quad} we prove the following result.

\begin{lemma}
\label{lmm:approx_NN_inverse}
There exists a constant $C>0$ 
such that for each $D >1$ 
there exist ReLU-NNs $\upphi_{D,\epsilon}\in \mathcal{N\!N}_{L,13,1,1}$ 
with 
\begin{equation}\label{eq:bound_L_NN_inverse}
\begin{aligned}
	L
	\leq
	C
	\left(\log\left(\frac{1}{\epsilon}\right)\right)^2
	\left(
		\left\lceil\log\left(D\right)\right\rceil^2
		+
		\left(
			\log
			\left(
				\frac{1}{\epsilon}
			\right)
		\right)^2
	\right)
	\quad
	\text{as }
	\epsilon \rightarrow 0
\end{aligned}
\end{equation}
such that
\begin{align}
	\norm{
		\frac{1}{x} 
		- 
		\upphi_{D,\epsilon}(x)
	}_{L^\infty((1,D))}
	\leq 
	\epsilon,
\end{align}
and with weights bounded in absolute value by 
a constant independent of $D$ and $\epsilon$.
\end{lemma}

\begin{proof}
Observe that 
\begin{align}
	\frac{1}{x}
	=
	\int\limits_{0}^{\infty} \exp(-xt) \text{d}t
	\quad
	\text{for any }
	x>0.
\end{align}
Let us set $k_\epsilon \coloneqq \left\lceil\left (\log(2C_S/ \epsilon)\right)^2\right\rceil$
with $C_S$ as in Lemma \ref{lmm:stenger_quad}.
Lemma \ref{lmm:stenger_quad} yields
\begin{align}
	\snorm{\frac{1}{x}
	- 
	\sum_{j=-k_\epsilon}^{k_\epsilon}\omega_j \exp(-t_j x)} 
	\leq 
	\frac{\epsilon}{2}, 
	\quad
	x\geq 1.
\end{align}
According to Proposition \ref{prop:approx_exp_NN}
there exist $C>0$ and ReLU-NNs $\Phi_{j,D,\epsilon} \in \mathcal{N\!N}_{L_j,9,1,1}$, 
$j=-k_\epsilon,\dots,k_\epsilon$, 
with depth
\begin{equation}
\begin{aligned}
	L_j
	\leq
	C
	\left(
		\left\lceil \log\left(t_j D\right)\right\rceil^2
	+
	\left(
		\log\left(\frac{2}{\epsilon}
		\displaystyle
		\sum_{j=-k_\epsilon}^{k_\epsilon}
		\omega_j\right)
	\right)^2
	\right),
\end{aligned}
\end{equation}
such that
\begin{align}\label{eq:error_individual_NNs}
	\norm{
		\Phi_{j,D,\epsilon}(x) 
		- 
		\exp
		\left(
			-t_j x
		\right)
	}_{L^\infty((1,D))}
	\leq
	\frac{\epsilon}{2\displaystyle\sum_{j=-k_\epsilon}^{k_\epsilon}\omega_j},
	\quad
	\text{for }
	j=-k_\epsilon,\dots,k_\epsilon.
\end{align}
We observe that for $j=- k_\epsilon,\dots,k_\epsilon$
\begin{equation}
	t_j \leq \log(3) + k^{\frac{1}{2}}_{\epsilon}\leq 2 k^{\frac{1}{2}}_\epsilon
	\quad
	\text{and}
	\quad
	\omega_j \leq k^{-\frac{1}{2}}_\epsilon
\end{equation}
thus $\displaystyle\sum_{j=-k_\epsilon}^{k_\epsilon} \omega_j \leq 2k^{\frac{1}{2}}_\epsilon$.

Define the ReLU-NNs
\begin{align}\label{eq:theta_construction}
	\Theta_{-k_\epsilon,\epsilon}(x)\coloneqq
	\begin{pmatrix}
		x \\
		\omega_{-{k_\epsilon}} \Phi_{-{k_\epsilon},D,\epsilon}(x)  \\
		0
	\end{pmatrix}
	,
	\quad
	\Theta_{j,\epsilon}(x_1,x_2,x_3)
	\coloneqq
	A
	\begin{pmatrix}
		x_1 \\ 
		\omega_j \Phi_{j,D,\epsilon}(x_2) \\
		x_3
	\end{pmatrix},
\end{align}
for $j=-k_\epsilon+1,\dots,k_\epsilon$,
where $A\in \mathbb{R}^{3\times 3}$ 
is the matrix performing the operation
$A (y_1,y_2,y_3)^\top = (y_1,y_1,y_2+y_3)^\top$.
Define
\begin{align}
	\upphi_{D,\epsilon}(x)
	&
	\coloneqq
	\begin{pmatrix}
		0 & 0 & 1 \\
	\end{pmatrix}
	(
	\Theta_{k_{\epsilon},\epsilon} 
	\circ
	\Theta_{k_{\epsilon}-1,\epsilon} 
	\circ 
	\cdots
	\circ
	\Theta_{-k_{\epsilon}+1,\epsilon}
	\circ
	\Theta_{-k_{\epsilon},\epsilon})(x) 
	,
	\quad
	x\in [1,D].
\end{align}
and observe that
\begin{equation}
	\upphi_{D,\epsilon}(x)
	=
	\sum_{j=-k_\epsilon}^{k_\epsilon}
	\omega_j
	\Phi_{j,D,\epsilon}(x),
	\quad
	x\in [1,D].
\end{equation}
It follows from \eqref{eq:error_individual_NNs} that
\begin{align}
	\norm{
		\sum_{j=-k_\epsilon}^{k_\epsilon}\omega_j \exp(-t_j x) 
		-	 
		\upphi_{D,\epsilon}(x)
	}_{L^\infty((1,D))} \leq \frac{\epsilon}{2}.
\end{align}
The triangle inequality yields
\begin{equation}
\begin{aligned}
	\norm{
		\frac{1}{x} 
		-
		\upphi_{D,\epsilon}(x)
	}_{L^\infty((1,D))}
	\leq
	&
	\norm{
		\frac{1}{x} 
		- 
		\sum_{j=-k_\epsilon}^{k_\epsilon}
		\omega_j \exp(-t_j x)
	}_{L^\infty((1,D))}
	\\
	&
	+
	\norm{
		\sum_{j=-k_\epsilon}^{k_\epsilon}
		\omega_j \exp(-t_j x) 
		- 
		\upphi_{D,\epsilon}(x)
	}_{L^\infty((1,D))}\\
	\leq
	& \epsilon.
\end{aligned}
\end{equation}
The depth of $\upphi_{D,\epsilon}$ is 
\begin{equation}
\begin{aligned}
	\mathcal{L}
	\left(
		\upphi_{D,\epsilon}
	\right)
	=
	\sum_{j=-k_\epsilon}^{k_\epsilon}
	L_j
	\leq
	&
	C
	\sum_{j=-k_\epsilon}^{k_\epsilon}
	\left(
		\left\lceil\log\left(t_j D\right)\right\rceil^2
		+
		\left(
			\log\left(\frac{2}{\epsilon}\displaystyle\sum_{j=-k}^k\omega_j\right)
		\right)^2
	\right)\\
	\leq
	&
	2C
	k_\epsilon
	\left(
		\left\lceil\log\left(2 k^{\frac{1}{2}}_\epsilon D\right)\right\rceil^2
		+
		\left(
			\log
			\left(
				\frac{4}{\epsilon}k^{\frac{1}{2}}_\epsilon
			\right)
		\right)^2
	\right)
	\\
	\leq
	&
	2C
	k_\epsilon
	\left(
		\left\lceil\log\left(D\right)\right\rceil^2
		+
		2\left(\log\left(2 k^{\frac{1}{2}}_\epsilon \right)\right)^2
		+
		\left(
			\log
			\left(
				\frac{2}{\epsilon}
			\right)
		\right)^2
	\right).
\end{aligned}
\end{equation}
For $\epsilon \in \left(0,\frac{1}{2C_S}\right)$ we have that
$\left\lceil\left (\log\left(\frac{2C_S}{\epsilon}\right)\right)^2\right\rceil \leq 8\left(\log\left(\frac{1}{\epsilon}\right)\right)^2$. 
We obtain
\begin{align}
	\mathcal{L}
	\left(
		\upphi_{D,\epsilon}
	\right)
	\leq
	16
	C
	\left(\log\left(\frac{1}{\epsilon}\right)\right)^2
	\left(
		\left(\log\left(D\right)\right)^2
		+
		2
		\left(
		\log
		\left(
			8
			\log
			\left(
				\frac{1}{\epsilon}
			\right)
		\right)
		\right)^2
		+
		\left(
			\log
			\left(
				\frac{2}{\epsilon}
			\right)
		\right)^2
	\right)
	\quad
	\text{as } \epsilon\rightarrow0.
\end{align}
In view of \eqref{eq:relu_nn_po_2} and \eqref{eq:theta_construction}
we have that $\mathcal{M}(\upphi_{D,\epsilon}) = 13$.
The weights of $\upphi_{D,\epsilon}$ are the ones of $\Phi_{j,D,\epsilon}$,
for $j=-k_\epsilon,\dots,k_\epsilon$, 
which are bounded in absolute
value by a constant $C>0$ (which is independent of $D$),
and $\omega_j$, $j=-k_\epsilon,\dots,k_\epsilon$,
which in turn are bounded by $1$. 
Thus,
the weights are bounded according to
$\mathcal{B}(\upphi_{D,\epsilon}) \leq 1$.

\end{proof}

\begin{lemma}
\label{lmm:inverse_pol_NN}
There exists a constant $C>0$ 
such that for each 
$n\in \mathbb{N}_{\geq 2}$ and $D>1$, 
there exist 
ReLU-NNs $\uppsi_{n,D,\epsilon}\in \mathcal{N\!N}_{L,13,1,1}$ 
with 
\begin{equation}
\begin{aligned}
	L
	\leq
	&
	C
	n
	\left(
	\log
	\left(
		\frac{1}{\epsilon}
	\right)
	+
	\log(2n+1)
	\right)
	\\
	&
	+
	C
	\left(
	\log
	\left(
		\frac{1}{\epsilon}
	\right)
	+
	\log(2n+1)
	\right)^2
	\left(
		\left(\log\left(D\right)\right)^2
		+
		\left(
			\log
			\left(
				\frac{2}{\epsilon}
			\right)
			+
			\log(2n+1)
		\right)^2
	\right)
	\quad
	\text{as }
	\epsilon 
	\rightarrow 0
\end{aligned}
\end{equation}
such that
\begin{align}
	\norm{
		\frac{1}{x^n} 
		- 
		\uppsi_{n,D,\epsilon}(x)
	}_{L^\infty((1,D))}
	\leq \epsilon.
\end{align}
The weights of $\uppsi_{n,D,\epsilon}$ are bounded according to
$\mathcal{B}(\uppsi_{n,D,\epsilon}) \leq 1$.
\end{lemma}
\begin{proof}
Let $\upphi_{D,\zeta}\in \mathcal{N\!N}_{L,13,1,1}$ with 
\begin{equation}
\begin{aligned}
	L
	\leq
	C
	\left(\log\left(\frac{1}{\zeta}\right)\right)^2
	\left(
		\left(\log\left(D\right)\right)^2
		+
		\left(
			\log
			\left(
				\frac{1}{\zeta}
			\right)
		\right)^2
	\right)
	\quad
	\text{as }
	\zeta \rightarrow 0
\end{aligned}
\end{equation}
and such that
\begin{align}
	\norm{
		\frac{1}{x} 
		- 
		\upphi_{D,\zeta}(x)
	}_{L^\infty((1,D))}
	\leq 
	\zeta
\end{align}
be as in Proposition \ref{lmm:approx_NN_inverse}
with $C>0$ independent of $\zeta>0$.

Let $\mu_{D_{n,\zeta},\zeta} \in \mathcal{N\!N}_{L,5,2,1}$ with 
$L\leq C\left(\log\lceil D_{n,\zeta}\rceil)+ \log\left(\frac{1}{\zeta}\right)\right)$
as $\epsilon\rightarrow 0$ be the multiplication network from
Proposition \ref{prop:mult_network} satisfying 
\begin{equation}
	\norm{
		\mu_{D_{n,\zeta} ,\zeta}(x,y)-xy
	}_{L^\infty
	\left(\left(-D_{n,\zeta},D_{n,\zeta} \right)^2\right)} 
	\leq 
	\zeta,
\end{equation}
where
\begin{equation}
	D_{n,\zeta} 
	=
	1
	+
	\norm{
		\frac{1}{x^{n}} 
		-
		\uppsi_{n,D,\zeta}(x)
	}_{L^\infty((1,D))}.
\end{equation}
Define
\begin{equation}
	\uppsi_{n+1,D,\zeta}(x)
	\coloneqq
	\mu_{D_{n,\zeta},\zeta}
	\left(
		\uppsi_{n,D,\zeta}(x)
		,
		\upphi_{D,\zeta}(x)
	\right),
	\quad
	\uppsi_{1,D,\zeta}(x)
	\coloneqq
	\upphi_{D,\zeta}(x),
	\quad
	x\in [1,D].
\end{equation}
Next, we calculate
\begin{equation}
\begin{aligned}
	\norm{
		\frac{1}{x^{n+1}} 
		-
		\uppsi_{n+1,D,\zeta}(x)
	}_{L^\infty((1,D))}
	\\
	&
	\hspace{-3.5cm}
	\leq
	\norm{
		\frac{1}{x} 
		\left(
			\frac{1}{x^{n}} 
			-
			\uppsi_{n,D,\zeta}(x)
		\right)
	}_{L^\infty((1,D))}
	+
	\norm{
		\uppsi_{n,D,\zeta}(x)
		\left(
			\frac{1}{x}
			-
			\upphi_{D,\zeta}(x)
		\right)
	}_{L^\infty((1,D))}
	\\
	&
	\hspace{-3.2cm}
	+
	\norm{
		\uppsi_{n,D,\zeta}(x)\upphi_{D,\zeta}(x)
		-
		\mu_{D_{n,\zeta},\zeta}
		\left(
			\uppsi_{n,D,\zeta}(x)
			,
			\upphi_{D,\zeta}(x)
		\right)
	}_{L^\infty((1,D))}
	\\
	&
	\hspace{-3.5cm}
	\leq
	\norm{
		\frac{1}{x^{n}} 
		-
		\uppsi_{n,D,\zeta}(x)
	}_{L^\infty((1,D))}
	\\
	&
	\hspace{-3.2cm}
	+
	\zeta
	\left(
		1
		+
		\norm{
			\frac{1}{x^{n}} 
			-
			\uppsi_{n,D,\zeta}(x)
		}_{L^\infty((1,D))}
	\right)
	+
	\zeta.
\end{aligned}
\end{equation}
Hence, for $\zeta \in (0,1)$ we obtain
\begin{equation}
\begin{aligned}
	\norm{
		\frac{1}{x^{n+1}} 
		-
		\uppsi_{n+1,D,\zeta}(x)
	}_{L^\infty((1,D))}
	\leq
	&
	\zeta
	\left(
		1
		+
		\zeta
	\right)^n
	+
	2
	\zeta
	\sum_{j=0}^{n-1}
	\left(
		1
		+
		\zeta
	\right)^j
	\\
	\leq
	&
	\zeta
	\left(
		1
		+
		\zeta
	\right)^n
	+
	2
	\zeta
	n
	(1+\zeta)^n
	\\
	\leq
	&
	(2n+1)
	(1+\zeta)^n
	\zeta.
\end{aligned}
\end{equation}
Thus, for $n \in \mathbb{N}_{\geq 2}$
\begin{equation}
	\uppsi_{n,D,\zeta}(x)
	=
	\begin{pmatrix}
	0 & 1
	\end{pmatrix}
	\left(
		\Theta_{n,\zeta}
		\circ
		\cdots
		\circ
		\Theta_{2,\zeta}
		\circ
		\Theta_{1,\zeta}
	\right)
	(x),
	\quad
	x\in[1,D],
\end{equation}
where
\begin{equation}
	\Theta_{j,\zeta}
	(x_1,x_2)
	\coloneqq
	\begin{pmatrix}
		x_1 \\
		\mu_{D_{n,\zeta},\zeta}(x_1,x_2)
	\end{pmatrix},
	\quad
	j\geq 3,
\end{equation}
and
\begin{equation}
	\Theta_{2,\zeta}
	(x)
	\coloneqq
	\begin{pmatrix}
		x \\
		\mu_{D_{n,\zeta},\zeta}(x,x)
	\end{pmatrix}
	\quad
	\text{and}
	\quad
	\Theta_{1,\zeta}
	(x)
	\coloneqq
	\uppsi_{1,D,\zeta}(x).
\end{equation}
We have that
\begin{equation}
	\mathcal{M}
	\left(
		\uppsi_{n,D,\zeta}
	\right)
	=
	\max_{j=1,\dots,n}
	\mathcal{M}
	\left(
		\Theta_{j,\zeta}
	\right)
	=
	\mathcal{M}
	\left(
		\uppsi_{1,D,\zeta}(x)
	\right)
	=
	13
\end{equation}
and 
\begin{equation}
\begin{aligned}
	\mathcal{L}
	\left(
		\uppsi_{n,D,\zeta}
	\right)
	=
	&
	\sum_{j=1}^{n}
	\mathcal{L}
	\left(
		\Theta_{j,\zeta}
	\right)
	\\
	\leq
	&
	C
	n
	\left(
		\log
		\left(
			1
			+
			(2n-1)
			(1+\zeta)^{n-1}
			\zeta
		\right)
		+ 
		\log\left(\frac{1}{\zeta}\right)
	\right)
	\\
	&
	+
	C
	\left(\log\left(\frac{1}{\zeta}\right)\right)^2
	\left(
		\left\lceil\log\left(D\right)\right\rceil^2
		+
		\left(
			\log
			\left(
				\frac{2}{\zeta}
			\right)
		\right)^2
	\right).
\end{aligned}
\end{equation}
By setting $\epsilon = (2n-1)(1+\zeta)^{n-1} \zeta$ and observing 
that $\zeta \leq \frac{\epsilon}{2n-1}$ we get
\begin{equation}
\begin{aligned}
	\log
	\left(
		\frac{1}{\zeta}
	\right)
	&
	=
	\log
	\left(
		\frac{1}{\epsilon}
	\right)
	+
	\log(2n-1)
	+
	n
	\log
	\left(
		1
		+
		\zeta
	\right)
	\\
	&
	\leq
	\log
	\left(
		\frac{1}{\epsilon}
	\right)
	+
	\log(2n-1)
	+
	(n-1)
	\log
	\left(
		1
		+
		\frac{\epsilon}{2n-1}
	\right)
	\\
	&
	\leq
	\log
	\left(
		\frac{1}{\epsilon}
	\right)
	+
	\log(2n-1)
	+
	\epsilon
	\frac{n-1}{2n-1},
\end{aligned}
\end{equation}
thus yielding the claimed result.
\end{proof}

\section{Wavenumber-Robust Deep ReLU-NN Emulation}
\label{sec:k_robust_ReLU_NN_emulation}
In this section, we establish the existence of ReLU-NNs
emulating $\varphi_\kappa \in L^2(\Gamma)$ with 
a robust dependence of the depth, width and weights bounds 
on the wavenumber, and 
spectral convergence with respect to the target accuracy.

A result in this direction is established in the following proposition,
which is in turn based on Corollary \ref{eq:smoothness_neuman_trace}
and Proposition \ref{prop:approximation_P}.

\begin{proposition}[$\kappa$-Explicit ReLU-NN Emulation of $V_\kappa$]
\label{prop:approximation_V_bad}
Let Assumption \ref{assumption:smooth_curve} be satisfied.
For each $n\in \IN_{\geq 2}$ there exist $C_n>0$ independent of $\kappa$
and ReLU-NNs $V_{\kappa,n,\epsilon} \in \mathcal{N\!N}_{L,15,1,2}$, 
i.e.
one input and two outputs accounting for the real and imaginary parts of $V_\kappa$
separately, and
\begin{equation}
	L
	\leq
	C_n
	(1+\kappa)^{\frac{2}{3}}
	\left(
		(1+\kappa)^{\frac{2}{3}}
		\epsilon^{-\frac{2}{n-1}}
		+ 
		\epsilon^{-\frac{1}{n-1}}
		\log
		\left(\frac{1}{\epsilon}\right)
		+
		\epsilon^{-\frac{1}{n-1}}
		\log
		\left(
			1
			+
			\kappa
		\right)
	\right) 
	\quad
	\text{as }
	\epsilon \rightarrow 0
\end{equation}
satisfying
\begin{equation}
	\norm{
		V_\kappa
		-
		\left(V_{\kappa,n,\epsilon}\right)_1
		-
		\imath
		\left(V_{\kappa,n,\epsilon}\right)_2
	}_{
		L^\infty
		\left(
			(0,2\pi)
		\right)
	}
	\leq
	\epsilon,
\end{equation}
where $(V_{\kappa,n,\epsilon})_1$ and $(V_{\kappa,n,\epsilon})_2$ represent 
the first and second outputs of $V_{\kappa,n,\epsilon}$, respectively, with
weights bounded in absolute value by $C_n(1+\kappa)^{\frac{(n-1)}{3}}$
for some constant $C_n>0$ independent of $\kappa$ and dependent on
$n \in \IN_{\geq 2}$.
\end{proposition}

\begin{proof}
According to Lemma \ref{eq:smoothness_neuman_trace}
the function $V_\kappa$ in \eqref{eq:decomp_neumann_trace}
belongs to $\mathcal{P}_{\boldsymbol{\lambda_\kappa},2\pi,\mathbb{C}}$
(cf.~Definition \ref{def:smooth_funct}) with $D = 2\pi$ and 
$\boldsymbol{\lambda}_\kappa = \{\lambda_{n,\kappa}\}_{n \in \IN_0}$ with
$\lambda_{n,\kappa} = C_n(1+\kappa)^{\frac{(n-1)}{3}}$ for each $n \in \IN_0$,
where each $C_n>0$ does not depend on $\kappa$.

It follows from Proposition \ref{prop:approximation_P} that
for each $n \in \IN_{\geq 2}$ there exist $\widetilde{C}_{n,\kappa}>0$,
which in this case does depend on $\kappa$, and
ReLU-NNs $V_{\kappa,n,\epsilon} \in \mathcal{N\!N}_{L,25,1,2}$
with
\begin{equation}
	L
	\leq
	\widetilde{C}_{n,\kappa}
	\left(
		\widetilde{C}_{n,\kappa}
		\epsilon^{-\frac{2}{n-1}}
		+ 
		\epsilon^{-\frac{1}{n-1}}
		\log
		\left(\frac{1}{\epsilon}\right)
		+
		\epsilon^{-\frac{1}{n-1}}
		\log
		\left(
			2
			\pi
		\right)
		+
		\epsilon^{-\frac{1}{n-1}}
		\log
		\left(
			1
			+
			\kappa
		\right)
	\right)
	\text{ as }\epsilon \rightarrow 0,
\end{equation}
and satisfying
\begin{equation}
	\norm{
		V_\kappa
		-
		\left(
			V_{\kappa,n,\epsilon}
		\right)_1
		-
		\imath
		\left(
			V_{\kappa,n,\epsilon}
		\right)_2
	}_{
		L^\infty
		\left(
			(0,2\pi)
		\right)
	}
	\leq
	\epsilon,
\end{equation}
where, according to the bound in \eqref{eq:explicit_constant_smooth_f},
$\widetilde{C}_{n,\kappa}\leq{{C}}_n \left(1+\kappa\right)^{\frac{2}{3}}$.
This concludes the proof.
\end{proof}

The result presented in Proposition \ref{prop:approximation_V_bad}
provides a ReLU-NN emulation of the Neumann datum with spectral accuracy,
fixed width, and increasing depth with an explicit dependence on the problem's
wavenumber. However, this result is not entirely satisfactory as the required
number of layers increases algebraically with the wavenumber, as opposed to a
wavenumber-robust dependence that would be obtained by having a number of layers
that increases poli-logarithmically with the wavenumber.

The reason behind this behaviour is that through 
Lemma \ref{eq:smoothness_neuman_trace},
we bound uniformly the derivatives of $V_\kappa$ over the entire interval $[0,2\pi]$.
It is well-understood that close to the tangency points $t_1$ and $t_2$
(cf. Figure \ref{fig:geometry}) the appearance of transition layers of size
$\mathcal{O}(\kappa^{-\frac{1}{3}})$ thwarts the convergence of
traditional BEM discretizations. 

In order to obtain wavenumber-robust ReLU-NN emulation of the solution
to the BIE \eqref{eq:cfie_eq} one needs to resort to Theorem \ref{thm:expansion_V}.
In particular, the asymptotic expansion stated in \eqref{eq:decomp_V_frequency}
uses the the composition operation to describe the behaviour of the normal derivative.
This operation is a key building block in the construction of ReLU-NNs, regardless of the
chosen activation function. 

The above observation naturally hints the use of ReLU-NNs for the approximation
of $V_\kappa$, the normal derivative, and, in turn, the approximation of the scattered 
field. 
\subsection{Wavenumber Robust ReLU-NN Emulation of Fock's Integral and its Derivatives}
\label{ssec:approx_fock_integral}
We consider the approximation of Fock's integral as in \eqref{eq:fock_integral}
within the interval 
\begin{align}\label{eq:interval_I_k}
	\mathcal{I}_\kappa
	\coloneqq
	\left(-\kappa^{\frac{1}{3}} A, \kappa^{\frac{1}{3}} A\right),
\end{align}
for $A>0$ 
that does not depend on the wavenumber $\kappa$.

The next result addresses the wavenumber-robust 
ReLU-NN emulation rate bounds of Fock's integral
with a fixed width and a depth depending poly-logarithmically on the wavenumber.
The proof of this result may be found in Appendix \ref{sec:proof_FockInt}.

\begin{proposition}[Wavenumber-Robust ReLU-NN Emulation of Fock's Integral $\Psi$]
\label{prop:approx_FockInt_NN}
For each $n\in \IN_{\geq 2}$, 
there exist a constant $C_n>0$ and ReLU-NNs 
$\Psi_{\kappa,n,\epsilon} \in \mathcal{N\!N}_{L,30,1,2}$, 
i.e.~
one input and two outputs accounting for the real and imaginary parts of $\Psi$, and
\begin{equation}\label{eq:depth_Psi}
\begin{aligned}
	L
	\leq
	&
	C_n
	\left(
	\log
		\left(
			\frac{1}{\epsilon}
		\right)
	\right)^2
	\left(
		\log
		\left(
			\kappa^{\frac{1}{3}}A
		\right)
		+
		\log
		\left(
			\kappa^{\frac{1}{3}}A
		\right)^2
		+
			\log
		\left(
			\frac{1}{\epsilon}
		\right)
		+
		\left(
		\log
		\left(
			\frac{1}{\epsilon}
		\right)
		\right)^2
	\right)
	\\
	&
	+
	B_{n,\epsilon}
	\left(
		B_{n,\epsilon}
		\epsilon^{-\frac{2}{n-1}}
		+
		\epsilon^{-\frac{1}{n-1}}
		\log\left(\frac{1}{\epsilon}\right) 
		+
		\epsilon^{-\frac{2}{n-1}}
		\log
		\left(
			\frac{1}{\epsilon}
		\right)
		+
		\epsilon^{-\frac{1}{n-1}}
	\right)
	\quad
	\text{as }\epsilon \rightarrow 0,
\end{aligned}
\end{equation}
with $B_{n,\epsilon}\leq C_n \epsilon^{-\frac{n+1}{(n-1)(3n+2)}}$ 
and for
\begin{equation}\label{eq:condition_kappa}
	\kappa
	\geq
	\left(\frac{2}{A} \right)^3
	\max
	\left\{
	 \left(
	 	2
	 	\frac{C_n}{\epsilon}
	\right)^{\frac{3}{3(n+1)-1}}
	,
	1
	\right\},
\end{equation}
satisfying
\begin{equation}
	\norm{
		\Psi
		-
		\left(\Psi_{\kappa,n,\epsilon}\right)_1
		-
		\imath
		\left(\Psi_{\kappa,n,\epsilon}\right)_2
	}_{
		L^\infty
		\left(
			\mathcal{I}_\kappa
		\right)
	}
	\leq
	\epsilon.
\end{equation}
The weights of $\Psi_{\kappa,n,\epsilon}$ are bounded in absolute value 
but a constant independent of $\kappa$.
\end{proposition}

The next result addresses the ReLU-NN emulation of the higher-order
derivatives of Fock's integral.
The proof of this result follows similar steps to that of Proposition 
\ref{prop:approx_FockInt_NN}, thus here we only highlight the main 
differences.
The proof of this result may also be found in Appendix \ref{sec:proof_FockInt}.

\begin{proposition}[Wavenumber-Robust ReLU-NN Emulation of $\Psi^{(\ell)}$]
\label{prop:approx_FockInt_NN_der}
For each $\ell\in \IN$ and $n \in \IN_{\geq 2}$
there exist a constant $C_{n,\ell}>0$ and ReLU-NNs 
$\Psi_{\kappa,n,\ell,\epsilon} \in \mathcal{N\!N}_{L,30,1,2}$,
i.e.~one input and two outputs accounting for the real and
imaginary parts of $\Psi^{(\ell)}$, 
and
\begin{equation}\label{eq:depth_Psi_L_der}
\begin{aligned}
	L
	\leq
	&
	C_{n,\ell}
	\left(
		\log
		\left(
			\frac{1}{\epsilon}
		\right)
	\right)^2
	\left(
		\log(\kappa)
		+
		\log
		\left(
			\kappa^{\frac{1}{3}}A
		\right)
		+
		\log
		\left(
			\kappa^{\frac{1}{3}}A
		\right)^2
		+
			\log
		\left(
			\frac{1}{\epsilon}
		\right)
		+
		\left(
		\log
		\left(
			\frac{1}{\epsilon}
		\right)
		\right)^2
	\right)
	\\
	&
	+
	B_{n,\ell,\epsilon}
	\left(
		B_{n,\ell,\epsilon}
		\epsilon^{-\frac{2}{n-1}}
		+
		\epsilon^{-\frac{1}{n-1}}
		\log\left(\frac{2}{\epsilon}\right) 
		+
		\epsilon^{-\frac{2}{n-1}}
		\log
		\left(
			\frac{1}{\epsilon}
		\right)
		+
		\epsilon^{-\frac{1}{n-1}}
	\right)
\end{aligned}
\end{equation}
with 
\begin{equation}
	B_{n,\ell,\kappa,\epsilon}
	\leq
	C_{n,\ell}
	\left(
		\log\left(\kappa^\frac{1}{3}A\right)
		+
		\log \left (\frac{1}{\epsilon}\right)  
	\right)^{\frac{n+1}{n-1}}
	\epsilon^{-\frac{n+1}{(n-1)(3n+\ell+2)}},
\end{equation}
as $\epsilon \rightarrow 0$, and for
\begin{equation}\label{eq:condition_kappa_l}
	\kappa
	\geq
	\left(\frac{2}{A} \right)^3
	\max
	\left\{
	 \left(
	 	2
	 	\frac{C_n}{\epsilon}
	\right)^{\frac{3}{3(n+1)+\ell-1}}
	,
	1
	\right\},
\end{equation}
satisfying
\begin{equation}
	\norm{
		\Psi^{(\ell)}
		-
		\left(\Psi_{\kappa,n,\ell,\epsilon}\right)_1
		-
		\imath
		\left(\Psi_{\kappa,n,\ell,\epsilon}\right)_2
	}_{
		L^\infty
		\left(
			\mathcal{I}_\kappa
		\right)
	}
	\leq
	\epsilon.
\end{equation}
The weights of $\Psi_{\kappa,n,\ell,\epsilon}$ are bounded in absolute value 
but a constant independent of $\kappa$.
\end{proposition}

\subsection{Wavenumber-Robust Approximation of the Neumann Trace}
\label{sec:wave_number_robust_approximation}
We present three different results that aim at 
approximating the Neumann trace of the sound-soft acoustic scattering
problem in two dimensions. Our approach leverages on the decomposition
stated in \eqref{eq:decomp_neumann_trace}. We proceed to describe these
three different statements:
\begin{itemize}
	\item[(i)] 
	{\bf Wavenumber-Robust ReLU-NN Emulation of $V_\kappa$.}
	In Theorem \ref{thm:approx_V_s_k}, we provide a constructive wavenumber-robust
	ReLU-NN emulation of the function $V_\kappa$ appearing in \eqref{eq:decomp_neumann_trace}.
	The construction of such ReLU-NN makes use of the following ingredients. 
	Proposition \ref{prop:asymptotic_V_remainder}
	provides a wavenumber-explicit asymptotic expansion of $V_\kappa$.
	Propositions \ref{prop:approx_FockInt_NN} and \ref{prop:approx_FockInt_NN_der}
	provide wavenumber robust ReLU-NN emulation of Fock's integral and its derivatives.
	By combining these two results, we obtain a wavenumber robust ReLU-NN emulation of $V_\kappa$.
	\item[(ii)] 
	{\bf Partial Wavenumber-Robust Approximation of the Neumann Trace.}
	The second result, presented in Corollary \ref{cor:approx_k_explicit}, aims at 
	providing a first, preliminary wavenumber-robust ReLU-NN emulation result
	of the Neumann trace. Therein, however, we still incorporate explicitly the
	oscillatory component of the solution. 
	\item[(iii)]
	{\bf Full Wavenumber-Robust Approximation of the Neumann Trace.}
	Finally, in Theorem \ref{thm:approximation_phi_k}, we provide a full wavenumber robust
	approximation of the Neumann trace $\varphi_{\kappa}$. 
	As opposed to Corollary \ref{cor:approx_k_explicit},
	we provide a complete ReLU-NN emulation, which includes emulation of the oscillatory component of the Neumann trace.
\end{itemize}

\begin{theorem}[Wavenumber-Robust ReLU-NN Emulation of $V_\kappa$]
\label{thm:approx_V_s_k}
Let Assumption \ref{assumption:smooth_curve} be satisfied.
For each $n\in \IN$ there exist $C_n>0$ and ReLU-NNs
$V_{\kappa,n,\epsilon} \in \mathcal{N\!N}_{L,38,1,2}$, 
i.e. 
one input and two outputs accounting 
for the real and imaginary part of $V_\kappa$ separately, 
such that for every $0<\epsilon \leq 1$ and for every 
$\kappa \geq C_n \epsilon^{-\frac{3}{4n+2}}$,
with 
\begin{equation}\label{eq:depth_V_k_n_eps}
\begin{aligned}
	L
	\leq
	&
	C_{n}
	\left(
		\log
		\left(
			\frac{1}{\epsilon}
		\right)
	\right)^2
	\left(
		\log
		\left(
			\kappa
		\right)^2
		+
			\log
		\left(
			\frac{1}{\epsilon}
		\right)
		+
		\left(
		\log
		\left(
			\frac{1}{\epsilon}
		\right)
		\right)^2
	\right)
	\\
	&
	+
	B_{n,\kappa,\epsilon}
	\left(
		B_{n,\kappa,\epsilon}
		\epsilon^{-\frac{2}{n-1}}
		+
		\epsilon^{-\frac{1}{n-1}}
		\log\left(\frac{2}{\epsilon}\right) 
		+
		\epsilon^{-\frac{2}{n-1}}
		\log
		\left(
			\frac{1}{\epsilon}
		\right)
		+
		\epsilon^{-\frac{1}{n-1}}
	\right)
\end{aligned}
\end{equation}
where
\begin{equation}
	B_{n,\kappa,\epsilon}
	\leq
	C_{n}
	\left(
		\log\left(\kappa\right)
		+
		\log \left (\frac{1}{\epsilon}\right)  
	\right)^{\frac{n+1}{n-1}}
	\epsilon^{-\frac{n+1}{(n-1)(4n+2)}}
	\quad
	\epsilon \rightarrow 0,
\end{equation}
there holds
\begin{equation}
	\norm{
		V_{\kappa}
		-
		\left(V_{\kappa,n,\epsilon}\right)_1
		-
		\imath
		\left(V_{\kappa,n,\epsilon}\right)_2
	}_{
		L^\infty
		\left(
			(0,2\pi)
		\right)
	}
	\leq
	\epsilon.
\end{equation}
The weights of $V_{\kappa,n,\epsilon}$ are bounded in absolute value 
by $C_n$.
\end{theorem}

\begin{proof}
Assumption \ref{assumption:smooth_curve} allows to 
use the asymptotic expansions in Section \ref{sec:AsTotField}.

Let us consider Proposition \ref{prop:asymptotic_V_remainder}, 
in particular the decomposition stated in \eqref{eq:decom_V_k}, 
which reads as follows:
for any $L$, $M\in \mathbb{N}_0$ it holds
\begin{equation}\label{eq:decom_V_k_updated}
	V_\kappa(s)
	=
	\sum_{\ell,m=0}^{L,M} 
	\kappa^{-\frac{1}{3}-\frac{2}{3}\ell-m}
	b_{\ell,m}(s) 
	\Psi^{(\ell)} \left(\kappa^\frac{1}{3} Z(s)\right)
	+
	R_{L,M,\kappa}(s),
	\quad
	s\in [0,2\pi],
\end{equation}
with remainder term $R_{L,M,\kappa}$ satisfying \eqref{eq:residual_R}.

In view of \eqref{eq:residual_R_mu} for each $n \in \IN$ 
we select $L = M = n $,
thus yielding for each $j \in \mathbb{N}_0$
\begin{align}\label{eq:upper_bound_res}
	\snorm{D^j _sR_{L,M,\kappa}(s)}
	\leq
	C_{L,M,j}
	\left(
		1
		+
		\kappa
	\right)^{-\frac{2}{3}(n+1)+\frac{j}{3}},
	\quad
	s\in [0,2\pi],
%
\end{align}
where $C_{L,M,j}$ does not depend on $\kappa$. 

For the sake of simplicity, we write the double sum in 
\eqref{eq:decom_V_k} as a single one. To this end,
we uniquely enumerate the pair of indices in the double sum
through a bijection $\{1,\dots,(n+1)^2\}\ni \ell \mapsto (k_{1,\ell},k_{2,\ell}) \in \{0,\dots,n\} \times \{0,\dots,n\}$
and obtain
\begin{equation}
\begin{aligned}
	V_\kappa(s)
	=
	\sum_{\ell=1}^{(n+1)^2} 
	\kappa^{-\frac{1}{3}-\frac{2}{3}k_{1,\ell}-k_{2,\ell}}
	b_{k_{1,\ell},k_{2,\ell}}(s) 
	\Psi^{\left(k_{1,\ell}\right)} \left(\kappa^\frac{1}{3} Z(s)\right)
	+
	R_{L,M,\kappa}(s),
	\quad
	s\in [0,2\pi].
\end{aligned}
\end{equation}
The following claims follow from Proposition \ref{prop:approximation_P}
in Appendix \ref{appendix:proofs}:
\begin{itemize}
\item[(i)]
The bound stated in \eqref{eq:upper_bound_res} implies that for each $n \in \mathbb{N}$
one has $R_{L,M,\kappa} \in \mathcal{P}_{\boldsymbol{\lambda},2\pi,\mathbb{C}}$ (cf. 
with 
\begin{equation}
	\boldsymbol{\lambda}_n\coloneqq\{\lambda_{j,n}\}_{j \in \mathbb{N}_0}
	\quad
	\text{and}
	\quad
	\lambda_{j,n}
	\coloneqq
	C_{L,M,j}
	\left(
		1
		+
		\kappa
	\right)^{-\frac{2}{3}(n+1)+\frac{j}{3}}
	,
	\quad
	j\in \IN_0.
\end{equation}
For $j= 0, \dots, 2(n+1)$ it holds $\snorm{D^j _sR_{L,M,\kappa}(s)} \leq C_{L,M,j}$,
where $C_{L,M,j}$ is as in \eqref{eq:upper_bound_res} and does not depend on $\kappa$.
According to Proposition \ref{prop:approximation_P},
for each $n\in \IN_{\geq 2}$ there exist ReLU-NNs 
$\upphi_{n,\epsilon}\in \mathcal{N\!N}_{L_R(\epsilon),15,1,2}$
with 
\begin{equation}
	L_R
	\left(
		\epsilon
	\right)
	\leq 
	C_n
	\left(
		\epsilon^{-\frac{2}{n-1}}
		+
		\epsilon^{-\frac{1}{n-1}}
		\log\left(\frac{1}{\epsilon}\right) 
		+
		\epsilon^{-\frac{1}{n-1}}
	\right)
	\quad
	\text{as }
	\epsilon \rightarrow 0
\end{equation}
such that 
$$
	\norm{
		R_{L,M,\kappa}(s)
		-
		\left(
			\left(\upphi_{n,\epsilon}(s) \right)_1
			+
			\imath 
			\left(\upphi_{n,\epsilon}(s) \right)_2
		\right)
		}_{L^\infty((0,2\pi))} 
		\leq 
		\epsilon,
$$
with $C_n>0$ independent of $\kappa$.
The weights of $\upphi_{n,\epsilon}$ are bounded
in absolute value by a constant depending on $n$ but not on $\kappa$ or $\epsilon$. 

\item[(ii)]
Item (ii) in Theorem \ref{thm:expansion_V} 
and Proposition \ref {prop:asymptotic_V_remainder} imply that for each
$\ell,m \in \IN_0\cap [0,\dots,n]$ one has 
$b_{\ell,m} \in \mathcal{P}_{\boldsymbol{\lambda},2\pi,\mathbb{C}}$
with $\boldsymbol{\lambda}\coloneqq\{\lambda_n\}_{n\in \mathbb{N}_0}$ 
independent of $\kappa$.
According to Proposition \ref{prop:approximation_P}
for each $n\in \IN_{\geq 2}$ there exist 
$C_n>0$ independent of $\kappa$ and ReLU-NNs 
$\beta_{\ell,m,n,\epsilon} \in \mathcal{N\!N}_{L,15,1,2}$
with 
\begin{equation}
	L_{b_{\ell,m}}
	\left(
		\epsilon
	\right)
	\leq 
	B_n
	\left(
		\epsilon^{-\frac{2}{n-1}}
		+
		\epsilon^{-\frac{1}{n-1}}
		\log\left(\frac{1}{\epsilon}\right) 
		+
		\epsilon^{-\frac{1}{n-1}}
	\right)
	\quad
	\text{as } 
	\epsilon \rightarrow 0
\end{equation}
such that 
$$
	\norm{
		b_{\ell,m}
		-
		\left(
			\left(
				\beta_{\ell,m,n,\epsilon}
			\right)_1
			+
			\imath
			\left(
				\beta_{\ell,m,n,\epsilon}
			\right)_2
		\right)
	}_{L^\infty((0,2\pi))} 
	\leq 
	\epsilon.
$$ 
Again, the weights of $\beta_{\ell,m,n,\epsilon}$ are bounded
in absolute value by a constant depending on $n$ but not on $\kappa$. 
\item[(iii)] 
Item (ii) in Remark \ref{rmk:periodicity_Z_symbol} and 
item (ii) in Theorem \ref{thm:expansion_V} imply that
$Z \in \mathcal{P}_{\boldsymbol{\lambda},2\pi,\mathbb{R}}$
with a sequence $\boldsymbol{\lambda}_{\kappa}\coloneqq\{\lambda_n\}_{n\in \mathbb{N}_0}$
independent of $\kappa$.
Therefore, according to Proposition \ref{prop:approximation_P}
for each $n\in \IN_{\geq 2}$ there exist $C_n>0$ independent
of $\kappa$ and ReLU-NNs 
$\Phi_{Z,n,\epsilon} \in \mathcal{N\!N}_{L_Z(\epsilon),13,1,2}$
with 
\begin{equation}
	L_Z(\epsilon)
	\leq 
	C_n
	\left(
		\epsilon^{-\frac{2}{n-1}}
		+
		\epsilon^{-\frac{1}{n-1}}
		\log\left(\frac{1}{\epsilon}\right) 
		+
		\epsilon^{-\frac{1}{n-1}}
	\right)
	\quad
	\text{as } 
	\epsilon \rightarrow 0
\end{equation}
such that 
\begin{equation}\label{eq:error_Z}
	\norm{
		Z
		-
		\Phi_{Z,n,\epsilon}
	}_{L^\infty((0,2\pi))} 
	\leq 
	\epsilon.
\end{equation}
Define $T_\kappa(x) = \kappa^{\frac{1}{3}} x $, which according
to \eqref{eq:relu_nn_po_1} belongs to $\mathcal{N\!N}_{2,2,1,1}$.
Furthermore, according to Proposition \ref{prop:trading_weights}
there exists a ReLU-NN (which we still refer to as $T_\kappa$) 
emulating exactly this map.
Indeed, $T_\kappa \in \mathcal{N\!N}_{L,3,1,2}$
with $L\leq 5+\frac{1}{3}\log(\kappa)$. 
Define the ReLU-NN $\Lambda_{\kappa,n,\epsilon}  = T_\kappa \circ \Phi_{Z,n,\epsilon}$.
For each $n \in \mathbb{N}$ one has
$\Lambda_{\kappa,n,\epsilon} 
\in 
\mathcal{N\!N}_{L_{\Lambda_{\kappa,n,\epsilon}}(\epsilon),13,1,1}$
with
\begin{equation}\label{eq:Lambda_depth}
\begin{aligned}
	L_{\Lambda_{\kappa,n,\epsilon}}(\epsilon)
	\leq
	&
	5+\frac{1}{3}\log(\kappa)
	\\
	&
	+
	C_n
	\left(
		\left(
			\frac{\kappa^{\frac{1}{3}}}{\epsilon}
		\right)^{\frac{2}{n-1}}
		+
		\left(
			\frac{\kappa^{\frac{1}{3}}}{\epsilon}
		\right)^{\frac{1}{n-1}}
		\log\left(\frac{\kappa^{\frac{1}{3}}}{\epsilon}\right) 
		+
		\left(
			\frac{\kappa^{\frac{1}{3}}}{\epsilon}
		\right)^{\frac{1}{n-1}}
	\right),
\end{aligned}
\end{equation}
as $\epsilon \rightarrow 0$, satisfying
\begin{equation}
	\norm{
		\kappa^{\frac{1}{3}}
		Z
		-
		\Lambda_{\kappa,n,\epsilon}
	}_{L^\infty((0,2\pi))} 
	\leq
	\epsilon,
\end{equation}
and with weights bounded in absolute value by a constant independent of $\kappa$, yet 
still depending on $n \in \IN_{\geq 2}$.
The presence of the wavenumber in \eqref{eq:Lambda_depth} at this point does not
directly lead to wavenumber-robust bounds. However, ahead we will see how 
the presence of the wavenumber can be mitigated. 
\item[(iv)]
For each $n\in \IN_{\geq 2}$, Propositions \ref{prop:approx_FockInt_NN} and
\ref{prop:approx_FockInt_NN_der} assert the existence of ReLU-NN 
$\Psi_{\kappa,n,\epsilon}$ and $\Psi_{\kappa,\ell,n,\epsilon}$ emulating
Fock's integral $\Psi$ and its derivatives, respectively. 
We set $A\coloneqq  2 \displaystyle\max_{s\in [0,2\pi]} \snorm{Z(s)}>0$ in the definition of $\mathcal{I}_\kappa$
and denote by $L_{\Psi^{(\ell)}}(\epsilon)$ the depth of the ReLU-NN approximating $\Psi_{\kappa,n,\epsilon}$ 
and $\Psi_{\kappa,\ell,n,\epsilon}$,
as described in Propositions \ref{prop:approx_FockInt_NN} and
\ref{prop:approx_FockInt_NN_der}, with accuracy $\epsilon>0$.
\end{itemize}

We are interested in the wavenumber-robust ReLU-NN emulation of
terms of the form $b_{\ell,m} \Psi^{(\ell)} \circ Z_\kappa $, for $\ell, m \in \IN_0$,
which appear in \eqref{eq:decom_V_k_updated}.
We recall Lemma \ref{lmm:approx_comp_mult_complex}
in Appendix \ref{sec:aux_res}. 
One can readily
verify that items (ii) through (iv) satisfy the assumptions of said result. 
Therefore, given $\ell,m \in \IN_0$, we conclude that for each $n\in \IN_{\geq 2}$
there exist ReLU-NNs $\Sigma_{n,\ell,m,\epsilon} \in \mathcal{N\!N}_{L,32,1,2}$ 
with 
\begin{equation}\label{eq:bound_L_Sigma}
\begin{aligned}
	L
	\leq 
	&
	C
	\left(
		\log
		\left(
			\lceil U_{\ell,m,\kappa} \rceil
		\right)
		+
		\log
		\left(
			\frac{12(n+1)^2}{\epsilon  \kappa^{\frac{1}{3}}}
		\right)
	\right)
	\\
	&
	+
	2
	\underbrace{
		L_{b_{\ell,m}}
		\left(
		\frac{
			\epsilon \kappa^{\frac{1}{3}}
		}{
			24
			(n+1)^2
			\norm{
				b_{\ell,m}
			}_{L^\infty(\mathcal{I}_{2\pi})}
			\norm{
				\Psi^{(\ell+1)}
			}_{L^\infty(\mathcal{I}_\kappa)}
		}
		\right)
		}_{\text{(I)}}
	\\
	&
	+
	2
	\underbrace{
		L_{\Psi^{(\ell)}}
		\left(
		\frac{
			\epsilon \kappa^{\frac{1}{3}}
		}{
			(n+1)^2
			\norm{
				b_{\ell,m}
			}_{L^\infty(\mathcal{I}_{2\pi})}
		}
		\right)
	}_{\text{(II)}}
		\\
		&
		+
		2
		\underbrace{
		L_{\Lambda_{\kappa,n,\epsilon}}
		\left(
		\frac{
			\epsilon \kappa^{\frac{1}{3}}
		}{
			24
			(n+1)^2
			\norm{
				b_{\ell,m}
			}_{L^\infty\left(\mathcal{I}_{2\pi}\right)}
			\norm{
				\Psi^{(\ell+1)}
			}_{L^\infty(\mathcal{I}_\kappa)}
		}
		\right),
		}_{\text{(III)}}
\end{aligned}
\end{equation}
where
\begin{equation}
	U_{\ell,m,\kappa} 
	=
	2 \max\left\{\norm{b_{\ell,m}}_{L^\infty(\mathcal{I}_{2\pi})},\norm{\Psi^{(\ell+1)}}_{L^\infty(\mathcal{I}_\kappa)}\right\}
\end{equation}
satisfying
\begin{equation}\label{eq:accuracy_blm_kappa}
	\norm{
		b_{\ell,m} 
		\Psi^{(\ell)}
		\circ
		Z_\kappa 
		-
		\Sigma_{n,\ell,m,\epsilon}
	}_{L^\infty(\mathcal{I}_D)}
	\leq
	\frac{\epsilon \kappa^{\frac{1}{3}}}{(n+1)^2}.
\end{equation}

We proceed to expand the terms on the right-hand side of \eqref{eq:bound_L_Sigma}
term-by-term. Firstly, by recalling that $b_{\ell,m}$ does not depend on $\kappa$
and that, according to Remark \ref{rmk:bound_psi}, for any $\ell \in \IN$ one has
$\norm{\Psi^{(\ell)}}_{L^\infty(\mathcal{I}_\kappa)} \leq C_\ell$, with $C_\ell$ independent
of $\kappa$ for each $\ell \in \IN$. Therefore, for each $n, \ell \in \IN$ there exists $B_{n,\ell}>0$ independent of $\kappa$
such that
\begin{equation}
\begin{aligned}
	\text{(I)}
	\leq
	B_{n,\ell}
	\left(
		\epsilon^{-\frac{2}{n-1}}
		+
		\epsilon^{-\frac{1}{n-1}}
		\log\left(\frac{1}{\epsilon}\right) 
		+
		\epsilon^{-\frac{1}{n-1}}
	\right)
	\quad
	\text{as } 
	\epsilon \rightarrow 0.
\end{aligned}
\end{equation}

The second term, labelled (II) in \eqref{eq:bound_L_Sigma},
is bounded as in \eqref{eq:depth_Psi} and \eqref{eq:depth_Psi_L_der},
for $\ell = 0$ and $\ell \in \IN$, respectively, and effectively remains unchanged
with possible different constants. 
The third term in \eqref{eq:bound_L_Sigma}, i.e. (III), is bounded
according to
\begin{equation}
	\text{(III)}
	\leq
	5+\frac{1}{3}\log(\kappa)
	+
	C_m
	\left(
		\epsilon^{-\frac{2}{n-1}}
		+
		\epsilon^{-\frac{1}{n-1}}
		\log\left(\frac{1}{\epsilon}\right) 
		+
		\epsilon^{-\frac{1}{n-1}}
		\log\left(\kappa \right) 
	\right),
\end{equation}
with $C_n$ independent of $\kappa$. Observe that the presence of the wavenumber 
in the terms of the form $\epsilon^{-\frac{1}{n-1}}$ and $\epsilon^{-\frac{2}{n-1}}$ is cancelled
due to the required accuracy stated in \eqref{eq:accuracy_blm_kappa}.
Next, set
\begin{equation}
	V_{\kappa,n,\epsilon}(s)
	\coloneqq
	\sum_{\ell=1}^{(n+1)^2} 
	\kappa^{-\frac{1}{3}-\frac{2}{3}k_{1,\ell}-k_{2,\ell}}
	\Sigma_{n,k_{1,\ell},k_{2,\ell},\epsilon}
	(s)
	+
	\upphi_{n,\epsilon}(s)
	,
	\quad
	s\in [0,2\pi].
\end{equation}
Therefore, we have
\begin{equation}
\begin{aligned}
	V_\kappa(s)
	-
	\left(
		V_{\kappa,n,\epsilon}(s)
	\right)_1
	-
	\imath
	\left(
		V_{\kappa,n,\epsilon}(s)
	\right)_2
	\\
	&
	\hspace{-3.5cm}
	=
	\sum_{\ell=1}^{(n+1)^2} 
	\kappa^{-\frac{1}{3}-\frac{2}{3}k_{1,\ell}-k_{2,\ell}}
	\left(
		b_{k_{1,\ell},k_{2,\ell}}(s) 
		\Psi^{\left(k_{1,\ell}\right)} \left(\kappa^\frac{1}{3} Z(s)\right)
		-
		\Sigma_{n,k_{1,\ell},k_{2,\ell},\epsilon}
	\right)
	\\
	&
	\hspace{-3.2cm}
	+
	R_{L,M}(s,\kappa)
	-
	\left(
		\left(\upphi_{n,\epsilon}(s) \right)_1
		+
		\imath 
		\left(\upphi_{n,\epsilon}(s) \right)_2
	\right)
\end{aligned}
\end{equation}
and
\begin{equation}
\begin{aligned}
	\norm{
		V_\kappa(s)
		-
		\left(
			V_{\kappa,n,\epsilon}(s)
		\right)_1
		-
		\imath
		\left(
			V_{\kappa,n,\epsilon}(s)
		\right)_2
	}_{L^\infty((0,2\pi))}
	\\
	&
	\hspace{-5.2cm}
	\leq
	\sum_{\ell=1}^{(n+1)^2} 
	\kappa^{-\frac{1}{3}-\frac{2}{3}k_{1,\ell}-k_{2,\ell}}
	\norm{
			\Sigma_{n,k_{1,\ell},k_{2,\ell},\epsilon}
		-
		b_{k_{1,\ell},k_{2,\ell}}
		\Psi^{\left(k_{1,\ell}\right)} \left(\kappa^\frac{1}{3} Z\right)
	}_{L^\infty((0,2\pi))}
	\\
	&
	\hspace{-4.9cm}
	+
	\norm{
		\upphi_{n,\epsilon}(s)
		-
		R_{L,M}(s,\kappa)
	}_{L^\infty((0,2\pi))}
	\\
	&
	\hspace{-5.2cm}
	\leq
	\epsilon.
\end{aligned}
\end{equation}
We express $V_{\kappa,n,\epsilon}$ as a ReLU-NN of fixed width. 
To this end, we define
\begin{equation}
	\Theta_{0,n,\epsilon}
	(x)
	\coloneqq
	\begin{pmatrix}
		\left( \upphi_{n,\epsilon}(x) \right)_1 \\
		\left( \upphi_{n,\epsilon}(x) \right)_2
	\end{pmatrix}
\end{equation}
and
\begin{equation}
	\Theta_{n,\ell,\epsilon}
	(x_1,x_2,x_3)
	\coloneqq
	A_\ell
	\begin{pmatrix}
		\left( \Sigma_{n,k_{1,\ell},k_{2,\ell},\epsilon}(x_1) \right)_1 \\
		\left(\Sigma_{n,k_{1,\ell},k_{2,\ell},\epsilon} (x_1)\right)_2 \\
		x_1 \\
		x_2 \\
		x_3 \\
	\end{pmatrix}
	,
	\quad
	\ell = 1,\dots, (n+1)^2,
\end{equation}
with $A_\ell \in \mathbb{R}^{3 \times 5}$ given by
\begin{equation}
	A_\ell
	=
	\begin{pmatrix}
		0 & 0 & 1 & 0 & 0 \\
		\kappa^{-\frac{1}{3}-\frac{2}{3}k_{1,\ell}-k_{2,\ell}} & 0 & 0 & 1 & 0 \\
		0 & \kappa^{-\frac{1}{3}-\frac{2}{3}k_{1,\ell}-k_{2,\ell}} & 0 & 0 & 1
	\end{pmatrix}
	.
\end{equation}
One can readily see that
\begin{equation}
	V_{\kappa,n,\epsilon}(s)
	=
	\begin{pmatrix}
	0 & 0 & 0 & 1 & 0 \\
	0 & 0 & 0 & 0 & 1
	\end{pmatrix}
	\left(
		\Theta_{n,(n+1)^2,\epsilon}
		\circ
		\cdots
		\circ
		\Theta_{n,1,\epsilon}
		\circ
		\Theta_{n,0,\epsilon}
	\right)
	(s).
\end{equation}
Therefore
\begin{equation}
\begin{aligned}
	\mathcal{M}
	\left(
		V_{\kappa,n,\epsilon}
	\right)
	&	
	=
	\max_{\ell\in \{0,\dots,(n+1)^2\}}
	\mathcal{M}
	\left(
		\Theta_{n,\ell,\epsilon}
	\right)
	\\
	&
	=
	6
	+
	\max_{\ell\in \{0,\dots,(n+1)^2\}}
	\mathcal{M}
	\left(
		\Sigma_{n,k_{1,\ell},k_{2,\ell},\epsilon}
	\right)
	=
	38
\end{aligned}
\end{equation}
and
\begin{equation}
	\mathcal{L}
	\left(
		V_{\kappa,n,\epsilon}
	\right)
	=
	\sum_{\ell=0}^{(n+1)^2}
	\mathcal{L}
	\left(
		\Theta_{\ell,\epsilon}
	\right)
	=
	\mathcal{L}
	\left(
		\upphi_{n,\epsilon}
	\right)
	+
	\sum_{\ell=1}^{(n+1)^2}
	\mathcal{L}
	\left(
		\Sigma_{n,k_{1,\ell},k_{2,\ell},\epsilon}
	\right),
\end{equation}
which yields \eqref{eq:depth_V_k_n_eps}.
\end{proof}

\begin{corollary}[Partial Wavenumber-Robust ReLU-NN Emulation of the Neumann Trace]
\label{cor:approx_k_explicit}
Let Assumption \ref{assumption:smooth_curve} be satisfied.
For each $n\in \IN$ there exist $C_n>0$ and ReLU-NNs
$V_{\kappa,n,\epsilon} \in \mathcal{N\!N}_{L,38,1,2}$ with $L$
as in \eqref{eq:depth_V_k_n_eps}, and $\Phi_{\gamma,\hat{{\bf d}},n,\epsilon} \in \mathcal{N\!N}_{L,13,1,1}$
with 
$$L\leq
	C_n
	\left(
		\epsilon^{-\frac{2}{n-1}}
		+
		\epsilon^{-\frac{1}{n-1}}
		\log\left(\frac{1}{\epsilon}\right) 
		+
		\epsilon^{-\frac{1}{n-1}}
	\right)
$$
satisfying
\begin{equation}
	\norm{
		\hat{\varphi}_\kappa
		-
		\kappa
		\left(
			\left(
				V_{\kappa,n,\epsilon}
			\right)_1
			+
			\imath
			\left(
				V_{\kappa,n,\epsilon}
			\right)_2
		\right)
		\exp
		\left(
			\imath \kappa \Phi_{\gamma,\hat{{\bf d}},n,\epsilon}
		\right)
	}_{
		L^\infty
		\left(
			(0,2\pi)
		\right)
	}
	\leq
	\epsilon,
\end{equation}
where $(V_{\kappa,n,\epsilon})_1$ and $(V_{\kappa,n,\epsilon})_2$ represent 
the first and second outputs of $V_{\kappa,n,\epsilon}$, respectively, 
$C_n$ is independent of $\kappa$, and
$\hat{\varphi}_\kappa$ is as in \eqref{eq:decomp_neumann_trace}.
The weights of $\Phi_{\boldsymbol{\gamma},\hat{{\bf d}},n,\epsilon}$ are
bounded in absolute value by a constant independent of $\kappa$.
\end{corollary}

\begin{proof}
It follows from Proposition \ref{prop:approximation_P}
	that for each $n \in \mathbb{N}_{\geq 2}$ there exist ReLU-NNs
	$\Phi_{\boldsymbol{\gamma},\hat{{\bf d}},n,\epsilon}\in  \mathcal{N\!N}_{L,13,1,1}$ 
        with
	\begin{align}\label{eq:bound_depth_lc_cheb_f}
	L
	\leq
	C_n
	\left(
		\epsilon^{-\frac{2}{n-1}}
		+
		\epsilon^{-\frac{1}{n-1}}
		\log\left(\frac{1}{\epsilon}\right) 
		+
		\epsilon^{-\frac{1}{n-1}}
	\right)
	\quad
	\text{as } \epsilon \rightarrow 0
	\end{align}
	such that
	\begin{equation}
		\norm{
			\boldsymbol{\gamma}\cdot \hat{{\bf d}}
			-
			\Phi_{\boldsymbol{\gamma},\hat{{\bf d}},n,\epsilon}
		}_{L^\infty\left(\mathcal{I}_{2\pi}\right)}
		\leq 
		\epsilon,
	\end{equation}
	with weights bounded in absolute value by a constant depending on $n \in \IN$, but
	not on the wavenumber $\kappa$.
\end{proof}

\begin{theorem}[Full Wavenumber-Robust ReLU-NN Emulation of the Neumann Trace]
\label{thm:approximation_phi_k}
There exists a constant $C>0$ such that 
for each $n\in \IN_{n\geq C\log(\kappa)}$ 
there exist $C_n>0$ and ReLU-NNs
$\hat{\varphi}^{\mathcal{N\!N}}_{\kappa,n,\epsilon} \in \mathcal{N\!N}_{L,42,1,2}$,
i.e.~one input and two outputs 
for real and imaginary parts of $\hat{\varphi}_{\kappa}$,
of depth
\begin{equation}\label{eq:length_phi_kappa}
\begin{aligned}
	L
	\leq
	&
	C_{n}
	\left(
		\log
		\left(
			\kappa
		\right)
		+
		\log
		\left(
			\frac{1}{\epsilon}
		\right)
	\right)^2
	\left(
		\log
		\left(
			\kappa
		\right)^2
		+
			\log
		\left(
			\frac{1}{\epsilon}
		\right)
		+
		\left(
		\log
		\left(
			\frac{1}{\epsilon}
		\right)
		\right)^2
	\right)
	\\
	&
	+
	B_{n,\kappa,\epsilon}
       	\left(1+B_{n,\kappa,\epsilon}+2\log\left(\frac{2}{\epsilon}\right)\right) \epsilon^{-\frac{2}{n-1}}
	\text{as } \epsilon \rightarrow 0
\end{aligned}
\end{equation}
and 
\begin{equation}
	B_{n,\kappa,\epsilon}
	\leq
	C_{n}
	\left(
		\log\left(\kappa\right)
		+
		\log \left (\frac{1}{\epsilon}\right)  
	\right)^{\frac{n+1}{n-1}}
	\epsilon^{-\frac{n+1}{(n-1)(4n+2)}}
	\quad
	\text{as } \epsilon \rightarrow 0
\end{equation}
satisfying for $\kappa \geq C_n\epsilon^{-\frac{3}{4n+2}}$
\begin{equation}\label{eq:error_bound_varphi_kappa}
	\norm{
		\hat{\varphi}_{\kappa} 
		-
		\left(
			\hat{\varphi}^{\mathcal{N\!N}}_{\kappa,n,\epsilon}
		\right)_1
		-
		\imath
		\left(
			\hat{\varphi}^{\mathcal{N\!N}}_{\kappa,n,\epsilon}
		\right)_2
	}_{
		L^\infty
		\left(
			(0,2\pi)
		\right)
	}
	\leq
	\epsilon,
\end{equation}
where $\left(\hat{\varphi}^{\mathcal{N\!N}}_{\kappa,n,\epsilon}\right)_1$ and 
$\left(\hat{\varphi}^{\mathcal{N\!N}}_{\kappa,n,\epsilon}\right)_2$ represent 
the first and second outputs of $\hat{\varphi}^{\mathcal{N\!N}}_{\kappa,n,\epsilon}$, respectively.
The weights of $\hat{\varphi}^{\mathcal{N\!N}}_{\kappa,n,\epsilon}$ 
are bounded in absolute value by a constant independent of $\kappa$, however still depending 
on $n\in \mathbb{N}$.
\end{theorem}

\begin{proof}
Firstly, it follows from the decomposition stated in \eqref{eq:decomp_neumann_trace} that
\begin{equation}\label{eq:decom_neumann_trace_2}
	\hat{\varphi}_{\kappa}(s)
	=
	V_\kappa(s) 
	g_\kappa 
	\left(
		\boldsymbol{\gamma}\cdot \hat{{\bf d}}(s)
	\right),
	\quad
	s\in [0,2\pi],
\end{equation}
with $g_\kappa(x) = \kappa \exp(\imath \kappa x)$.

The proof of this theorem relies on the application of Lemma \ref{lmm:approx_comp_mult_complex}
to \eqref{eq:decom_neumann_trace_2}.
We proceed to verify the hypothesis of this result item-by-item.
\begin{itemize}
	\item[(i)]
	According to Theorem \ref{thm:approx_V_s_k},
	for each $n\in \IN_{\geq 2}$ there exist $C_n>0$ and ReLU-NNs 
	$V_{\kappa,n,\epsilon} \in \mathcal{N\!N}_{L_{V_{\kappa}}(\epsilon),38,1,2}$,
	i.e.~with one input and two outputs (real and imaginary part, respectively),
	such that 
	\begin{equation}
	\norm{
		V_\kappa
		-
		\left(
			V_{\kappa,n,\epsilon}
		\right)_1
		-
		\imath
		\left(
			V_{\kappa,n,\epsilon}
		\right)_2
	}_{
		L^\infty
		\left(
			(0,2\pi)
		\right)
	}
	\leq
	\epsilon,
	\end{equation}
	with depth $L_{V_{\kappa}}(\epsilon)$ depending on $\epsilon>0$
	as in \eqref{eq:depth_V_k_n_eps}.
	\item[(ii)]
	Set $\zeta_k = \exp(\imath \kappa x)$. 
	It follows from Proposition \ref{prop:aprox_sc} that there exist ReLU-NNs
	$\Phi_{\zeta_k,\epsilon} \in \mathcal{N\!N}_{L_{\zeta_\kappa}(\epsilon),18,1,2}$ with
	\begin{equation}
		L_{\zeta_k}(\epsilon,B)
		\leq 
		C
		\left(
		\left(
		\log
		\left(
			\frac{1}{\epsilon}
		\right)
		\right)^2
		+
		\log
		\left(
			\kappa
			B
		\right)
		\right)
		\quad
		\text{as }
		\epsilon \rightarrow 0,
	\end{equation}
	with $C>0$ independent of $\kappa$, such that
	\begin{equation}
		\norm{
				\zeta_\kappa 
				-
				\left(
					\Phi_{\zeta_\kappa ,\epsilon}
				\right)_1
				-
				\imath
				\left(
					\Phi_{\zeta_\kappa ,\epsilon}
				\right)_2
		}_{L^\infty((-B,B))}
		\leq
		\epsilon,
	\end{equation}
	and with weights bounded in absolute value by a constant independent of $\kappa$.
	According to Proposition \ref{prop:trading_weights}, the map $T_\kappa(x) = \kappa x$
	can be exactly expressed by a ReLU-NN belonging to $\mathcal{N\!N}_{L,3,1,1}$ with
	$L\leq 5 + \log(\kappa)$, and with weights bounded in absolute value by one.
	Consequently, the ReLU-NN
	\begin{equation}
		\Phi_{g_\kappa,\epsilon}(x) 
		= 
		\begin{pmatrix}
			T_\kappa \circ (\Phi_{\zeta_k,\frac{\epsilon}{\kappa}}(x))_1 \\
			T_\kappa \circ (\Phi_{\zeta_k,\frac{\epsilon}{\kappa}}(x))_2
		\end{pmatrix}
		\in 
		\mathcal{N\!N}_{L_{g_k}(\epsilon,B),18,1,2}
	\end{equation}
	emulates $g_\kappa$ according to 
	\begin{equation}
		\norm{
				g_\kappa
				-
				\left(
					\Phi_{g_\kappa,\epsilon}
				\right)_1
				-
				\imath
				\left(
					\Phi_{g_\kappa,\epsilon}
				\right)_2
		}_{L^\infty((-B,B))}
		\leq
		\epsilon
		\quad
		\text{as } \epsilon\rightarrow 0,
	\end{equation}
	with 
	\begin{equation}
		L_{g_k}(\epsilon,B)
		\leq
		5 + \log(\kappa)
		+
		C
		\left(
		\left(
		\log
		\left(
			\frac{\kappa}{\epsilon}
		\right)
		\right)^2
		+
		\log
		\left(
			\kappa
			B
		\right)
		\right)
		\quad
		\text{as } \epsilon\rightarrow 0,
	\end{equation}
	where $C>0$ does not depend on $\kappa$, $B$ or $\epsilon>0$.
	\item[(iii)]
	As in the proof of Corollary \ref{cor:approx_k_explicit},
	it follows from Proposition \ref{prop:approximation_P}
	that for each $n \in \mathbb{N}_{\geq2}$ there exist ReLU-NNs
	$\Phi_{\gamma,\hat{{\bf d}},n,\epsilon}\in 
	\mathcal{N\!N}_{L_{\gamma,\hat{{\bf d}}}(\epsilon),13,1,1}$ with
	\begin{align}\label{eq:bound_depth_lc_cheb_f}
	L_{\gamma,\hat{{\bf d}}}(\epsilon)
	\leq
	C_n
	\left(
		\epsilon^{-\frac{2}{n-1}}
		+
		\epsilon^{-\frac{1}{n-1}}
		\log\left(\frac{1}{\epsilon}\right) 
		+
		\epsilon^{-\frac{1}{n-1}}
	\right),
	\quad
	\text{as }\epsilon \rightarrow 0,
	\end{align}
	such that
	$
		\norm{
			\boldsymbol{\gamma}\cdot \hat{{\bf d}}
			-
			\Phi_{\gamma,\hat{{\bf d}},n,\epsilon}
		}_{L^\infty\left(\mathcal{I}_{2\pi}\right)}
		\leq 
		\epsilon,
	$
	with weights bounded in absolute value by a constant depending on $n \in \IN$
	but not on $\kappa$.
\end{itemize}
Set $D_{\boldsymbol{\gamma}} = \norm{\boldsymbol{\gamma}}_{L^\infty((0,2\pi))}$.
Therefore, according to Lemma \ref{lmm:approx_comp_mult_complex} there exist
ReLU-NNs $\hat{\varphi}^{\mathcal{N\!N}}_{\kappa,n,\epsilon} \in  \mathcal{N\!N}_{L,42,1,2}$
\begin{equation}\label{eq:bound_L_varphi_NN}
\begin{aligned}
	L
	\leq 
	&
	C
	\left(
	\log
	\left(
		\lceil U \rceil
	\right)
	+
	\log
	\left(
		\frac{12}{\epsilon}
	\right)
	\right)
	+
	2
	L_{V_{\kappa}}
	\left(
		\frac{
			\epsilon
		}{
			24
			\norm{V_\kappa}_{L^\infty(\mathcal{I}_{2\pi})}\norm{g'_\kappa}_{L^\infty(\mathcal{I}_{2\pi})}
		}
	\right)
	\\
	&
	+
	2
	L_{g_\kappa}
	\left(
		\frac{
			\epsilon
		}{
			24
			\norm{V_\kappa}_{L^\infty(\mathcal{I}_{2\pi})}
		}
		,
		2
		D_{\boldsymbol{\gamma}} 
	\right)
	+
	2
	L_{\gamma,\hat{{\bf d}}}
	\left(
		\frac{
			\epsilon
		}{
			24
			\norm{V_\kappa}_{L^\infty(\mathcal{I}_{2\pi})}\norm{g'_\kappa}_{L^\infty(\mathcal{I}_{2\pi})}
		}
	\right)
\end{aligned}
\end{equation}
satisfying
\begin{equation}
	\norm{
		\hat{\varphi}_{\kappa}
		-
		\left(
			\hat{\varphi}^{\mathcal{N\!N}}_{\kappa,n,\epsilon}
		\right)_1
		-
		\imath
		\left(
			\hat{\varphi}^{\mathcal{N\!N}}_{\kappa,n,\epsilon}
		\right)_2
	}_{
		L^\infty
		\left(
			(0,2\pi)
		\right)
	}
	\leq
	\epsilon,
\end{equation}
where 
\begin{equation}
	U 
	=
	2
	\max
	\left\{
		\norm{V_\kappa}_{L^\infty(\mathcal{I}_{2\pi})}
		,
		\norm{g'_\kappa}_{L^\infty(\mathcal{I}_{2\pi})}
	\right\}.
\end{equation}
Clearly, $\norm{g'_\kappa}_{L^\infty(\mathcal{I}_{2\pi})} =\kappa^2$, and
it follows from Lemma \ref{eq:smoothness_neuman_trace} and Remark
\ref{rmk:correction_lemma_V_k} that
$\norm{V_\kappa}_{L^\infty(\mathcal{I}_{2\pi})}$ is bounded by a constant
independent of $\kappa$.

Observe that according to
\begin{equation}\label{eq:L_V_kappa}
\begin{aligned}
	L_{V_{\kappa}}
	&
	\left(
		\frac{
			\epsilon
		}{
			24
			\norm{V_\kappa}_{L^\infty(\mathcal{I}_{2\pi})}\norm{g'_\kappa}_{L^\infty(\mathcal{I}_{2\pi})}
		}
	\right)
	\\
	\leq
	&
	C'_{n}
	\left(
		\log(\kappa)
		+
		\log
		\left(
			\frac{1}{\epsilon}
		\right)
	\right)^2
	\left(
		\log
		\left(
			\kappa
		\right)^2
		+
			\log
		\left(
			\frac{\kappa}{\epsilon}
		\right)
		+
		\left(
		\log
		\left(
			\frac{\kappa}{\epsilon}
		\right)
		\right)^2
	\right)
	\\
	&
	+
	B_{n,\kappa,\epsilon}
	\left(
		B_{n,\kappa,\epsilon}
		\left(\frac{\kappa^2}{\epsilon}\right)^{\frac{2}{n-1}}
		+
		\left(\frac{\kappa^2}{\epsilon}\right)^{\frac{1}{n-1}}
		\log\left(\frac{\kappa}{\epsilon}\right) 
		+
		\left(\frac{\kappa^2}{\epsilon}\right)^{\frac{2}{n-1}}
		\log
		\left(
			\frac{\kappa}{\epsilon}
		\right)
		+
		\left(\frac{\kappa^2}{\epsilon}\right)^{\frac{1}{n-1}}
	\right)
\end{aligned}
\end{equation}
with 
\begin{equation}
	B_{n,\kappa,\epsilon}
	\leq
	C'_{n}
	\left(
		\log\left(\kappa\right)
		+
		\log \left (\frac{1}{\epsilon}\right)  
	\right)^{\frac{n+1}{n-1}}
	\left(\frac{\kappa^2}{\epsilon}\right)^{\frac{n+1}{(n-1)(4n+2)}}.
\end{equation}
This bound is not entirely wavenumber-robust. However, 
there exist $C,C'>0$ independent of $n$ and $\kappa$
such that for $n\geq C \log(\kappa)$ we have
$
	\kappa^{\frac{4}{n-1}}
	+
	\kappa^{\frac{2}{n-1}}
	+
	\kappa^{2\frac{n+1}{(n-1)(4n+2)}}
	\leq
	C',
$
thus allowing us to remove these terms in \eqref{eq:L_V_kappa}. 
Similar considerations are valid for the last two terms in \eqref{eq:bound_L_varphi_NN},
thus yielding the bound stated in \eqref{eq:length_phi_kappa}.

Finally, the weights are bounded
in absolute value by a constant independent of $\kappa$ and $\epsilon>0$
(still depending on $n\in \mathbb{N})$.
\end{proof}

%
\subsection{Wavenumber-Robust Approximation of the Far-Field}
\label{sec:wave_number_robust_approximation}
We are interested in the effect of the ReLU-NN approximation
of the Neumann trace on the far-field pattern of the scattered field.
An asymptotic expansion of \eqref{eq:green_rep_formula} yields
the following expression for the far-field behaviour of $u^{\text{s}}$
\begin{align}\label{eq:rep_formula_Dc_approx}
	u^{\text{s}}({\bf x})
	\sim
	\frac{\exp(\imath \pi/4)}{2\sqrt{2\pi}}
	\frac{\exp(\imath \kappa r)}{\sqrt{\kappa r}}
	F(\hat{\bf x})
	\quad
	\text{as }
	r
	\coloneqq
	\norm{{\bf x}}
	\rightarrow
	\infty,
\end{align}
where we have assumed that the origin of the cartesian system of coordinates
is contained in $\text{D}$, $\hat{\bf x} = {\bf x}/\norm{{\bf x}} \in \mathbb{S}_1$,
i.e. the unit circle, and
\begin{equation}
\begin{aligned}
	F(\hat{\bf x})
	&
	\coloneqq
	-
	\int\limits_{\Gamma}
	\exp(-\imath\kappa \hat{\bf x} \cdot {\bf y} )
	{\varphi}_{\kappa}({\bf x})
	\text{ds}_{{\bf y}}
	\\
	&
	=
	-
	\int\limits_{0}^{2\pi}
	\exp(-\imath\kappa \hat{\bf x} \cdot \boldsymbol{\gamma}(s))
	\hat{\varphi}_{\kappa}(s)
	\norm{\boldsymbol{\gamma}'(s)}
	\text{d}s,
	\quad
	\hat{\bf x}  \in \mathbb{S}_1.
\end{aligned}
\end{equation}
Let $\hat{\varphi}^{\mathcal{N\!N}}_{\kappa,n,\epsilon}$ be
as in Theorem \ref{thm:approximation_phi_k} and define
\begin{equation}\label{eq:far_field_pattern}
	F_{\kappa,n,\epsilon}(\hat{\bf x})
	=
	-
	\int\limits_{0}^{2\pi}
	\exp(-\imath\kappa \hat{\bf x} \cdot \boldsymbol{\gamma}(s))
	\left(
		\left(
			\hat{\varphi}^{\mathcal{N\!N}}_{\kappa,n,\epsilon}
		\right)_1
		+
		\imath
		\left(
			\hat{\varphi}^{\mathcal{N\!N}}_{\kappa,n,\epsilon}
		\right)_2
	\right)
	\norm{\boldsymbol{\gamma}'(s)}
	\text{d}s,
	\quad
	\hat{\bf x}  \in \mathbb{S}_1.
\end{equation}
Thus, it holds
\begin{equation}
	\norm{
		F
		-
		F_{\kappa,n,\epsilon}
	}_{L^\infty(\mathbb{S}^1)}
	\leq
	2\pi
	\norm{
		\boldsymbol{\gamma}'
	}_{L^\infty((0,2\pi))}
	\norm{
		\hat{\varphi}_{\kappa} 
		-
		\left(
			\hat{\varphi}^{\mathcal{N\!N}}_{\kappa,n,\epsilon}
		\right)_1
		-
		\imath
		\left(
			\hat{\varphi}^{\mathcal{N\!N}}_{\kappa,n,\epsilon}
		\right)_2
	}_{
		L^\infty
		\left(
			(0,2\pi)
		\right)
	}.
\end{equation}
Therefore, the constructed ReLU-NN $\hat{\varphi}^{\mathcal{N\!N}}_{\kappa,n,\epsilon}$ once is
inserted in \eqref{eq:far_field_pattern} produces an equally accurate approximation for the 
far-field pattern, of course, up to a constant which notably is independent of the wavenumber. 
\section{Concluding Remarks}
\label{sec:Concl}
We proved wavenumber-robust ReLU-NN emulation rate bounds
for the sound-soft acoustic scattering by a model 
smooth, strictly convex obstacle in two dimensions.
Due to these assumptions on the scatterers' shapes 
issues like conical diffraction, 
trapping modes and 
characteristic boundary behaviours were avoided.  

We present a constructive approximation rate bound
that relied on two main tools. 
Firstly, we adopt
a boundary reduction of the scattering problem using boundary potentials.
This yields an \emph{equivalent BIE equation of the second kind}
for unknown Neumann datum on the scatterers' surface.
Second, 
using the asymptotic expansion by Melrose and Taylor \cite{MelTaylor85}, 
we provide a detailed construction of
ReLU-NNs approximating the Neumann datum with NNs of 
moderate, fixed width,
and of 
depth that increases spectrally with the target accuracy, 
whereas a bound for the NN depth 
that depends poly-logarithmically on the wavenumber.

The present results pertain only to the \emph{DNN approximation error}.
As it is well-known (see, e.g., \cite{mishra2023estimates}), 
in a DNN-based solution algorithm,
also the training and the generalization errors need to be bounded.
In the presently considered, exterior acoustic scattering problem, 
\eqref{eq:sound_soft_problem}-\eqref{eq:Sommerf}, 
the boundary reduction to the BIE \eqref{eq:cfie_eq} on the scatterer's
profile $\Gamma$, similar to \cite{AHS23_2956},
reduces controlling all these errors in $L^2(\Gamma)$, 
rather than in the unbounded, exterior domain $\D^\cc$.
In addition, 
all quantities being periodic w.r. to the arclength parametrization of $\Gamma$,
for which efficient quadratures, such as the trapezoidal rule, are available.

We conclude this work by indicating extensions of the present results, and
comment on aspects of possible DL algorithms based on the BIE \eqref{eq:cfie_eq}.
For $L,M \in \IN$, 
let us consider the map 
\begin{align}\label{eq:loss_nn}
	L^2
	\left(
	\left(
		0,2\pi
	\right)
	\right)
	\supset
	\mathcal{N\!N}_{L,M,1,2} 
	\ni
	\Phi_\kappa
	\mapsto 
	\text{Loss}(\Phi_\kappa)
	\coloneqq
	\norm{
		\hat{\varphi}_\kappa
		-
			(\Phi_\kappa)_1
			-
			\imath
			(\Phi_\kappa)_2
	}^2_{L^2((0,2\pi))}.
\end{align}
Given a target accuracy $\epsilon>0$, we aim to
find a ReLU-NN $\Phi^\star_{\kappa,\epsilon}\in\mathcal{N\!N}_{L,M,1,2} 
\subset L^2((0,2\pi))$ such that
\begin{align}
	\text{Loss}
	\left(
		\Phi^\star_{\kappa}
	\right) 
	\leq 
	\epsilon,
\end{align}
where the depth $L$ and width $M$ depend on $\kappa$
and $\epsilon$. 
Of course, the computation of the loss function defined in
\eqref{eq:loss_nn} cannot be performed in practice.

Recalling Proposition \ref{thm:coercivity_cfie}, there exists a constant $\eta_0>0$ 
such that, given $\delta\in(0,1/2)$, there exists $\kappa_0$ (depending on $\delta$)
such that for $\kappa \geq \kappa_0$ and $\eta\geq \eta_0 \kappa$.
Choosing $\eta$ so that 
$\eta_0 \kappa \leq \eta  \lesssim \kappa$, we have that 
\begin{equation}
\begin{aligned}
	\norm{
		\hat{\varphi}_\kappa
		-
		(\Phi_\kappa)_1
		-
		\imath
		(\Phi_\kappa)_2	
	}_{L^2((0,2\pi))}
	&
	\leq 
	\frac{
	\displaystyle\sup_{s\in [0,2\pi]}
	\norm{
		\boldsymbol{\gamma}'(s)
	}^{-1}
	}{\frac{1}{2}-\delta}
	\norm{
		\mathsf{A}'_{\kappa,\eta}
		\left(
			\varphi_\kappa
			-
			\tau^{-1}_{\boldsymbol{\gamma}}
			\left(
			(\Phi_\kappa)_1
			+
			\imath
			(\Phi_\kappa)_2
			\right)
		\right)
	}_{L^2(\Gamma)}
	\\
	&
	\leq 
	\frac{
	\displaystyle\sup_{s\in [0,2\pi]}
	\norm{
		\boldsymbol{\gamma}'(s)
	}^{-1}
	}{\frac{1}{2}-\delta}
	\norm{
		f_{\kappa,\eta} 
		- 
		\mathsf{A}'_{\kappa,\eta}
		\left(
		\tau^{-1}_{\boldsymbol{\gamma}}
		(\Phi_\kappa)_1
		+
		\imath
		(\Phi_\kappa)_2
		\right)
	}_{L^2(\Gamma)}
	\\
	&
	\leq 
	C(\delta,\boldsymbol{\gamma})
	\norm{
		\tau_{\boldsymbol{\gamma}} f_{\kappa,\eta} 
		- 
		\tau_{\boldsymbol{\gamma}}
		\mathsf{A}'_{\kappa,\eta}
		\tau^{-1}_{\boldsymbol{\gamma}}
		\left(
		(\Phi_\kappa)_1
		+
		\imath
		(\Phi_\kappa)_2
		\right)
	}_{L^2((0,2\pi)}
\end{aligned}
\end{equation}
with $\tau_{\boldsymbol{\gamma}}:L^2(\Gamma) \rightarrow L^2((0,2\pi))$
as in Section \ref{sec:AsTotField}, and where the constant
\begin{equation}
	C(\delta,\boldsymbol{\gamma})
	\coloneqq
	\frac{
	\displaystyle\sup_{s\in [0,2\pi]}
	\norm{
		\boldsymbol{\gamma}'(s)
	}^{-1}
	\norm{
		\boldsymbol{\gamma}'
	}_{L^\infty((0,2\pi))}
	}{\frac{1}{2}-\delta}
	>0
\end{equation}
does not depend on the wavenumber $\kappa$.
Set 
\begin{equation}
	\widetilde{\text{Loss}}
	\left(
		\Phi_{\kappa}
	\right)
	\coloneqq
	\norm{
		\tau_{\boldsymbol{\gamma}} f_{\kappa,\eta} 
		- 
		\tau_{\boldsymbol{\gamma}}
		\mathsf{A}'_{\kappa,\eta}
		\tau^{-1}_{\boldsymbol{\gamma}}
		\left(
		(\Phi_\kappa)_1
		+
		\imath
		(\Phi_\kappa)_2
		\right)
	}^2_{L^2((0,2\pi)}
\end{equation}

We conclude that a NN $\Phi^\star_\kappa \in\mathcal{N\!N}_{L,M,1,2}$ 
computed by minimizing the loss function $\widetilde{\text{Loss}}$ within accuracy
$\epsilon>0$ approximates $\hat{\varphi}_\kappa$ with accuracy
$\sqrt{C(\delta,\boldsymbol{\gamma}) \epsilon}$ in the $L^2((0,2\pi)$-norm
for any $\delta\in(0,1/2)$ as long as $\eta>0$ is chosen as explained previously.
 
The presently employed NN constructions
were of so-called ``strict ReLU'' type.
This entailed various technicalities in our proofs due to the inability of ReLU-NNs
to perform exact multiplication of reals. 
Similar results with simpler proofs 
could be expected with more general activation functions.
The present ReLU-based approximations do imply, however, 
corresponding results for neuromorphic approximations by so-called
spiking NNs \cite{Maass1997b}. 
This follows from the recent ReLU-to-spiking conversion 
algorithm in \cite{SWBCPG2022}.

\appendix
\section{ReLU-NN Emulation of Smooth Functions}
\label{appendix:proofs}
In this section, we establish ReLU-NN \emph{spectral} emulation
rates of smooth functions using Chebyshev polynomials.

Let $T_n(x)\coloneqq\cos(n\arccos(x))\in \mathbb{P}_n$,
for $x\in[-1,1]$ and $n\in \mathbb{N}_0$, denote the
$n$-th Chebyshev polynomials of the first kind.
We have that $T_0(x)=1$, $T_1(x)=x$ and 
for each $ n\in \mathbb{N}$ it holds
\begin{align}
\label{eq:Chev_iter}
	T_{n+1}(x)
	=
	2xT_n(x)-T_{n-1}(x),
	\quad
	x\in [-1,1].
\end{align}
Chebyshev polynomials satisfy the following orthogonality
property in the interval $[-1,1]$
\begin{align}
	\int\limits_{-1}^{1} \frac{T_n(x) T_m(x)}{\sqrt{1-x^2}} \text{d}x
	=
	\left\{
	\begin{array}{cl}
	0, & n\neq m, \\
	\pi, & n= m =0, \\
	\frac{\pi}{2}, & n = m \neq 0. 
	\end{array}
	\right.
\end{align}
For $f\in \mathscr{C}^1([-1,1])$, the series 
\begin{align}
	f(x)
	=
	{\sum_{k=0}^{\infty}}^{\prime} a_k(f) T_k(x),
	\quad
	x\in [-1,1],
\end{align}
(the prime in the previous sum
means that the first term should be halved)
converges uniformly and absolutely in $[-1,1]$ and in $L^\infty((-1,1))$, 
respectively, where for $k\in \mathbb{N}_0$
\begin{align}\label{eq:a_f_T}
	a_k
	\left(f\right)
	\coloneqq
	\frac{2}{\pi}
	\int\limits_{-1}^{1} 
	\frac{T_k(x) f(x) }{\sqrt{1-x^2}} 
	\text{d}x.
\end{align}
According to \cite[Theorem 4.2]{Trefethen08} for $f\in \mathscr{C}^\infty([-1,1])$ 
and for all $\ell\in \mathbb{N}$ it holds that
\begin{align}\label{eq:bounds_a_f}
	\snorm{a_k(f)}
	\leq
	\frac{\norm{f^{(\ell+1)}}_{L^\infty((-1,1))}}{k(k-1)\cdots(k-\ell)}
	\leq
	\norm{f^{(\ell+1)}}_{L^\infty((-1,1))}
	,
	\quad
	k\geq \ell+1.
\end{align}
The following result quantifies how finite
linear combinations of Chebyshev polynomials 
can be emulated by means of ReLU-NNs.
\begin{proposition}\label{prop:approx_cheb_pol}
Let $n\in \mathbb{N}$ and let the sequence $\boldsymbol{a} \coloneqq \{a_k\}_{k=0}^{n} \subset \mathbb{K}$
be given, where $\mathbb{K} \in \{\mathbb{R},\mathbb{C}\}$. Set
\begin{align}
	f_{\boldsymbol{a},n}(x)
	=
	\sum_{k=0}^{n}
	a_k
	T_k(x),
	\quad
	x\in [-1,1].
\end{align}
There exists a constant $C>0$ such that the following statements hold true.
\begin{itemize}
	\item[(i)]
	Let $\mathbb{K} = \mathbb{R}$. Then there exist ReLU-NNs
	$\Phi_{\boldsymbol{a},n,\epsilon} \in  \mathcal{N\!N}_{L,13,1,1}$ with
	\begin{align}\label{eq:bound_depth_lc_cheb}
	L
	\leq
	C
	\left(
		n^2
		+
		n
		\log
		\left(\frac{1}{\epsilon}\right)
		+
		n
		\log
		\left(
			\max_{k=0,\dots,n}\snorm{a_k}
		\right)
	\right),
	\quad
	\text{as }
	\epsilon \rightarrow 0,
	\end{align}
	such that
	\begin{equation}
		\norm{
			f_{\boldsymbol{a},n}
			-
			\Phi_{\boldsymbol{a},n,\epsilon}
		}_{L^\infty((-1,1))}\leq \epsilon.
	\end{equation}
	\item[(ii)]
	Let $\mathbb{K} = \mathbb{C}$. Then there exist ReLU-NNs
	$\Phi_{\boldsymbol{a},n,\epsilon} \in  \mathcal{N\!N}_{L,15,1,2}$ with
	$L$ as in \eqref{eq:bound_depth_lc_cheb} such that
	\begin{equation}
		\norm{
			f_{\boldsymbol{a},n}
			-
			(\Phi_{\boldsymbol{a},n,\epsilon})_1
			-\imath
			(\Phi_{\boldsymbol{a},n,\epsilon})_2
		}_{L^\infty((-1,1))}
		\leq 
		\epsilon,
	\end{equation}
	where $(\Phi_{\boldsymbol{a},n,\epsilon})_1$ and $\imath(\Phi_{\boldsymbol{a},n,\epsilon})_2$
	denote the first and second outputs of $\Phi_{\boldsymbol{a},n,\epsilon}$, respectively.
\end{itemize}
In either case, the weights are bounded according to
\begin{align}
	\mathcal{B}
	\left(
		\Phi_{\boldsymbol{a},n,\epsilon}
	\right)
	\leq
	\max
	\left\{ 
		\displaystyle\max_{k=0,\dots,n}\snorm{a_k}
		,
		2
	\right\}.
\end{align}
\end{proposition}

\begin{proof}
We divide the proof in several steps.
\\

{\sf \encircle{1} ReLU-NN Emulation of Chebyshev Polynomials.} 
Let $\zeta \in (0,1)$.
We claim that for each $n\in \mathbb{N}$ there exist 
ReLU-NNs $\Phi^{\text{Cheb}}_{n,\zeta}$ such that
\begin{align}
	\label{eq:error_NN_Cheb}
	\norm{
		T_{n}
		-
		\Phi^{\text{Cheb}}_{n,\zeta}
	}_{L^\infty((-1,1))}
	\leq 
	\zeta \, 2^{2n}
	\eqqcolon
	D_{n,\zeta}.
\end{align}
Firstly, observe that $T_{0}(x)=1$ and $T_1(x)=x$
can be exactly emulated using ReLU-NNs.
Indeed, we have
\begin{equation}
	T_0 
	\in \mathcal{N\!N}_{2,4,1,1}
	\quad
	\text{and}
	\quad
	T_1
	\in \mathcal{N\!N}_{2,2,1,1},
\end{equation}
as in \eqref{eq:relu_nn_po_1} and \eqref{eq:relu_nn_po_2}, 
respectively.
Hence, \eqref{eq:error_NN_Cheb} holds for $n=0,1$. 

Assume that there exist ReLU-NNs $\Phi^{\text{Cheb}}_{n,\zeta}$ 
satisfying \eqref{eq:error_NN_Cheb} (induction hypothesis).
Then
\begin{align}
	\norm{\Phi^{\text{Cheb}}_{n,\zeta}}_{L^\infty((-1,1))}
	\leq
	\norm{\Phi^{\text{Cheb}}_{n,\zeta}-T_n}_{L^\infty((-1,1))}
	+
	\norm{T_n}_{L^\infty((-1,1))}
	\leq 
	D_{n,\zeta} +1.
\end{align}
Let $\mu_{D_{n,\zeta}+1,\zeta} \in\mathcal{N\!N}_{L,5,2,1}$ be the multiplication ReLU-NN
from Proposition \ref{prop:mult_network} in the interval $[-D_{n,\zeta}-1,D_{n,\zeta}+1]$.
There exists $C>0$ such that
\begin{align}\label{eq:mult_net_1}
	\norm{\mu_{D_{n,\zeta}+1,\zeta}(x,y)-xy}_{L^\infty((-D_{n,\zeta}-1,D_{n,\zeta}+1)^2)} 
	\leq 
	\zeta,
\end{align}
with 
\begin{align}\label{eq:mult_net_2}
	\mathcal{L}\left(\mu_{D_{n,\zeta}+1,\zeta}\right) 
	\leq 
	C
	\left(
	\log
	\left(
		\lceil D_{n,\zeta} +1 \rceil^2
	\right)
	+
	\log
	\left(
		\frac{1}{\zeta}
	\right)
	\right),
\end{align}
for a constant $C>0$ independent of $n$ and $\zeta$,
and with $\mathcal{B}(\mu_{D_{n,\zeta}+1,\zeta}) \leq 1$.

For each $n\in \mathbb{N}$ we set
\begin{align}\label{eq:def_cheb_net}
	\Phi^{\text{Cheb}}_{n+1,\zeta}(x)
	\coloneqq
	2\mu_{D_{n,\zeta}+1,\zeta}
	\left(
		x,\Phi^{\text{Cheb}}_{n,\zeta}(x)
	\right)
	-
	\Phi^{\text{Cheb}}_{n-1,\zeta}(x),
	\quad
	x\in[-1,1].
\end{align}

The network $\Phi^{\text{Cheb}}_{n+1,\zeta}$ defined recursively in \eqref{eq:def_cheb_net}
can be expressed as the composition of ReLU-NNs, 
hence it is a ReLU-NN itself. 
Assume that \eqref{eq:error_NN_Cheb} holds 
for the Chebyshev polynomials $T_n\in \mathbb{P}_n$ and $T_{n-1}\in\mathbb{P}_{n-1}$. 
For $x\in [-1,1]$, we calculate
\begin{equation}
\begin{aligned}
	\snorm{\Phi^{\text{Cheb}}_{n+1,\zeta}(x)-T_{n+1}(x)}
	\leq
	&
	2\snorm{x(T_n(x) - \Phi^{\text{Cheb}}_{n,\zeta})}
	+
	2\snorm{x\Phi^{\text{Cheb}}_{n,\zeta}(x)- \mu_{D_{n,\zeta}+1}
	\left(
		x,\Phi^{\text{Cheb}}_{n,\zeta}(x)
	\right)} \\
	&+
	\snorm{\Phi^{\text{Cheb}}_{n-1,\zeta}(x)-T_{n-1}(x)} \\
	\leq
	& 
	2D_{n,\zeta} + 2\zeta +D_{n-1,\zeta}
	\\
	\leq
	&
	\zeta 2^{2n}\left(2+1+\frac{1}{4}\right)
	\leq  
	\zeta  2^{2(n+1)}=D_{n+1,\zeta}.
\end{aligned}
\end{equation}
It follows that $\norm{\Phi^{\text{Cheb}}_{n+1,\zeta}- T_{n+1}}_{L^\infty((-1,1))}\leq D_{n+1,\zeta}$.
Hence, \eqref{eq:error_NN_Cheb} holds for all $n\in \mathbb{N}_0$.
For each $k\in\mathbb{N}$ we set
\begin{align}\label{eq:dnn_thetat_cheb}
	\Theta_0(x)
	\coloneqq
	\begin{pmatrix}
	x \\
	\Phi^{\text{Cheb}}_{1,\zeta}(x) \\
	\Phi^{\text{Cheb}}_{0,\zeta}(x) \\
	\frac{1}{2}a_0 \, \Phi^{\text{Cheb}}_{0,\zeta}(x)
	\end{pmatrix}
	\quad
	\text{and}
	\quad
	\Theta_k(x_1,x_2,x_3,x_4)
	\coloneqq
	M_k
	\begin{pmatrix}
	x_1 \\
	x_2 \\
	x_3\\
	x_4 \\
	\mu_{D_{k,\zeta}+1,\zeta}(x_1,x_2)
	\end{pmatrix},
\end{align}
where
\begin{equation}
	M_k
	=
	\begin{pmatrix}
	1 & 0 & 0 & 0 & 0\\
	0 & 0 & -1& 0  & 2\\
	0 & 1 & 0 & 0  & 0\\
	0 & a_k & 0 & 0 & 1
	\end{pmatrix},
\end{equation}
thus
\begin{equation}
	\Theta_k(x_1,x_2,x_3,x_4,x_5)
	=
	\begin{pmatrix}
	x_1 \\
	2\mu_{D_{k,\zeta}+1,\zeta}(x_1,x_2)-x_3 \\
	x_2 \\
	a_k x_2+x_4 \\
	\end{pmatrix}.
\end{equation}
Let us set
\begin{align}\label{eq:dnn_thetat_cheb_final_net}
	\Phi_{\boldsymbol{a},n,\epsilon}(x)
	\coloneqq
	\begin{pmatrix}
	0 & 0 & 0 & 1
	\end{pmatrix}
	\left(
		\Theta_n
		\circ
		\Theta_{n-1}
		\circ
		\Theta_{n-2}
		\circ
		\cdots
		\circ
		\Theta_1
		\circ
		\Theta_0
	\right)(x).
\end{align}
Observe that
\begin{align}
	\Phi_{\boldsymbol{a},n,\epsilon}(x)
	=
	{\sum_{k=0}^{n}}^{\prime} 
	a_k 
	\Phi^{\text{Cheb}}_{k,\zeta}(x).
\end{align}
The ReLU-NN $\Phi_{\boldsymbol{a},n,\epsilon}$ approximates $f_{\boldsymbol{a},n}$ as follows
\begin{equation}
\begin{aligned}
	\norm{
		\Phi_{\boldsymbol{a},n,\epsilon}
		-
		f_{\boldsymbol{a},n}
	}_{L^\infty((-1,1))}
	\leq
	&
	{\sum_{k=0}^{n}}^{\prime} 
	\snorm{a_k}
	\norm{\Phi^{\text{Cheb}}_{k,\zeta}-T_k}_{L^\infty([-1,1])} 
	\\
	\leq
	&
	\zeta
	\max_{k=0,\dots,n}\snorm{a_k} {\sum_{k=0}^{n}} 2^{k}
	\leq
	\max_{k=0,\dots,n}\snorm{a_k}
	\underbrace{
	\zeta \,
	2^{2(n+1)}}_{=D_{n+1,\zeta}}.
\end{aligned}
\end{equation}
Thus, by setting $\epsilon = \displaystyle\max_{k=0,\dots,n}\snorm{a_k}\zeta 2^{2(n+1)}$
we obtain $\norm{\Phi_{\boldsymbol{a},n,\epsilon} - f_{\boldsymbol{a},n} }_{L^\infty((-1,1))} \leq \epsilon$.

{\sf \encircle{2} Bounds for the Width and Depth.} 
To approximate $f_{\boldsymbol{a},n}$ by the ReLU-NN $\Phi_{f,n,\epsilon}$
with accuracy $\epsilon>0$
one needs 
\begin{equation}
\begin{aligned}
	\mathcal{L}
	\left(
		\Phi_{\boldsymbol{a},n,\epsilon}
	\right)
	&
	=
	\sum_{k=0}^n 
	\mathcal{L}\left(\mu_{D_{k,\zeta}+1,\zeta}\right)
	\\
	&
	\leq
	C
	\left(
		n
		+
		\sum_{k=0}^n \log(D_{k,\zeta}+1) 
		+ 
		\sum_{k=0}^n \log\left(\frac{1}{\zeta}\right)
	\right) \\
	&
	\leq
	C'
	\left(
		n^2
		+ 
		n 
		\log
		\left(\frac{1}{\epsilon}\right)
		+
		n 
		\log
		\left(
			\max_{k=0,\dots,n}\snorm{a_k}
		\right)
	\right),
	\quad
	\text{as }
	\epsilon \rightarrow 0,
\end{aligned}
\end{equation}
with $C,C'>0$ independent of $n \in \IN_0$.

An inspection of \eqref{eq:dnn_thetat_cheb} reveals that
\begin{equation}
	\mathcal{M}
	\left(
		\Phi_{\boldsymbol{a},n,\epsilon}
	\right)
	=
	\max_{j\in\{0,\dots,n\}}
	\mathcal{M}
	\left(
		\Theta_j
	\right)
	=
	8
	+
	\max_{j\in\{0,\dots,n\}}
	\mathcal{M}
	\left(
		\mu_{D_{j,\zeta}+1,\zeta}
	\right)
	=
	13.
\end{equation}
The weights of $\Phi_{\boldsymbol{a},n,\epsilon}$ are bounded in absolute value by 
\begin{align}
\label{eq:bound_weights_Cheb_NN}
	\max
	\left\{ 
		2,
		\displaystyle\max_{k=0,\dots,n}\snorm{a_k}
	\right\}.
\end{align}

{\sf \encircle{3} ReLU-NN Emulation of Complex-Valued Functions.} 
Finally, we discuss the case $\{a_k\}_{k=0}^{n} \subset \mathbb{C}$,
i.e.~we allow complex values. 

In this case, we approximate the real and imaginary part separately.
To this end, we modify the ReLU-NNs in \eqref{eq:dnn_thetat_cheb}
as follows
{\small
\begin{align}\label{eq:dnn_thetat_cheb_complex}
	\Theta_0(x)
	\coloneqq
	\begin{pmatrix}
	x \\
	\Phi^{\text{Cheb}}_{1,\zeta}(x) \\
	\Phi^{\text{Cheb}}_{0,\zeta}(x) \\
	\frac{1}{2}\Re\{a_0\} \, \Phi^{\text{Cheb}}_{0,\zeta}(x) \\
	\frac{1}{2}\Im\{a_0\} \, \Phi^{\text{Cheb}}_{0,\zeta}(x)
	\end{pmatrix}
	\;\;
	\text{and}
	\;\;
	\Theta_k(x_1,x_2,x_3,x_4,x_5)
	\coloneqq
	M_k
	\begin{pmatrix}
	x_1 \\
	x_2 \\
	x_3\\
	x_4 \\
	x_5 \\
	\mu_{D_{k,\zeta}+1,\zeta}(x_1,x_2)
	\end{pmatrix}
\end{align}
}%
with
\begin{equation}
	M_k
	=
	\begin{pmatrix}
	1 & 0 & 0 & 0 & 0 & 0\\
	0 & 0 & -1& 0 & 0 & 2\\
	0 & 1 & 0 & 0 & 0 & 0\\
	0 & \Re\{a_k\} & 0 & 0 & 1 & 0\\
	0 & \Im\{a_k\} & 0 & 0 & 0 & 1\\
	\end{pmatrix},
\end{equation}
thus
\begin{equation}
	\Theta_k(x_1,x_2,x_3,x_4,x_5)
	=
	\begin{pmatrix}
	x_1 \\
	2\mu_{D_{k,\zeta}+1,\zeta}(x_1,x_2)-x_3 \\
	x_2 \\
	\Re\{a_k\} x_2+x_4 \\
	\Im\{a_k\} x_2+x_5
	\end{pmatrix},
\end{equation}
and the one in \eqref{eq:dnn_thetat_cheb_final_net} as
\begin{align}
	\Phi_{\boldsymbol{a},n,\epsilon}(x)
	\coloneqq
	\begin{pmatrix}
	0 & 0 & 0 & 1 & 0 \\
	0 & 0 & 0 & 0 & 1 \\
	\end{pmatrix}
	\left(
		\Theta_n
		\circ
		\Theta_{n-1}
		\circ
		\Theta_{n-2}
		\circ
		\cdots
		\circ
		\Theta_1
		\circ
		\Theta_0
	\right)(x).
\end{align}
Observe that
\begin{align}
	\left(
		\Phi_{\boldsymbol{a},n,\epsilon}(x)
	\right)_1
	+
	\imath
	\left(
		\Phi_{\boldsymbol{a},n,\epsilon}(x)
	\right)_2
	=
	{\sum_{k=0}^{n}}  
	a_k 
	\Phi^{\text{Cheb}}_{k,\zeta}(x),
\end{align}
where $(\Phi_{\boldsymbol{a},n,\epsilon}(x))_1$ and $(\Phi_{\boldsymbol{a},n,\epsilon}(x))_2$
denote the first and second outputs of $\Phi_{\boldsymbol{a},n,\epsilon}$.
The bounds on the depth of exactly as in {\sf Step \encircle{2}}.
An inspection of \eqref{eq:dnn_thetat_cheb_complex} reveals that
\begin{equation}
	\mathcal{M}
	\left(
		\Phi_{\boldsymbol{a},n,\epsilon}
	\right)
	=
	\max_{j\in\{0,\dots,n\}}
	\mathcal{M}
	\left(
		\Theta_j
	\right)
	=
	10
	+
	\max_{j\in\{0,\dots,n\}}
	\mathcal{M}
	\left(
		\mu_{D_{j,\zeta}+1,\zeta}
	\right)
	=
	15.
\end{equation}
\end{proof}

%

Equipped with Propositions \ref{prop:approx_cheb_pol} we may 
state the ReLU-NN emulation of smooth functions.
To this end, we introduce the the following class
of functions.

\begin{definition}
\label{def:smooth_funct}
Let $\boldsymbol{\lambda}\coloneqq\{\lambda_n\}_{n\in \mathbb{N}_0}$
be a monotonically increasing sequence of positive numbers and let
$D\in \mathbb{R}_+$. We set
\begin{align}
	\mathcal{P}_{\boldsymbol{\lambda},D,\mathbb{K}}
	\coloneqq
	\left\{
		f\in\mathscr{C}^\infty([-D,D],\mathbb{K}): \,
		\norm{f^{(n)}}_{L^\infty((-D,D))} \leq \lambda_{n}, 
		\;\;
		\text{for each }n\in\mathbb{N}_0
	\right\},
\end{align}
where $\mathbb{K} \in \{\mathbb{R},\mathbb{C}\}$.
\end{definition}

Observe that in Definition \ref{def:smooth_funct} we allow the
functions to be either real- or complex-valued.
In view of Proposition \ref{prop:approx_cheb_pol}, this has implications
in the network's architecture. Indeed, one needs to accomodate the real
and imaginary outputs separately, thus yielding a slightly wider networks. 
In particular, when the sequence $\boldsymbol{\lambda}\coloneqq\{\lambda_n\}_{n\in \mathbb{N}_0}$
is given by $\lambda_n = n!$ for $n \in \mathbb{N}_0$, i.e.~when $f$ is an analytic 
function in the interval $[-D,D]$ for some $D>0$,
in \cite{elbrachter2021deep} the existence of
ReLU-NNs approximating this particular class of functions with exponential accuracy 
was proved.

The next result addresses the emulation of functions in 
$\mathcal{P}_{\boldsymbol{\lambda},D,\mathbb{K}}$ 
by ReLU-NNs with spectral accuracy in terms of the NN size.

\begin{proposition}\label{prop:approximation_P}
Let $\boldsymbol{\lambda}\coloneqq\{\lambda_n\}_{n\in \mathbb{N}_0}$ 
be a monotonically increasing sequence of positive numbers and 
let $D>0$.
For each $n \in \mathbb{N}_{\geq 2}$ 
there exist a constant $B_n>0$ such that for each
$f \in \mathcal{P}_{\boldsymbol{\lambda},D,\mathbb{K}}$ it holds:
\begin{itemize}
	\item[(i)]
	Let $\mathbb{K} = \mathbb{R}$. 
        There exist ReLU-NNs
	$\Phi_{f,n,D,\epsilon} \in  \mathcal{N\!N}_{L,13,1,1}$ 
        with
	\begin{align}\label{eq:bound_depth_lc_cheb_f}
	L
	\leq
	B_n
	\left(
		B_n
		\epsilon^{-\frac{2}{n-1}}
		+
		\epsilon^{-\frac{1}{n-1}}
		\log\left(\frac{2}{\epsilon}\right) 
		+
		\epsilon^{-\frac{1}{n-1}}
		\log
		\left(
			D
		\right)
		+
		\epsilon^{-\frac{1}{n-1}}
		\log
		\left(
			\lambda_n
		\right)
	\right)
	\end{align}
	as $\epsilon \rightarrow 0$ such that
	\begin{equation}
		\norm{
			f
			-
			\Phi_{f,n,D,\epsilon}
		}_{L^\infty\left(\left(-D,D\right)\right)}
		\leq 
		\epsilon.
	\end{equation}
	\item[(ii)]
	If $\mathbb{K} = \mathbb{C}$, there exist ReLU-NNs
	$\Phi_{f,n,D,\epsilon} \in  \mathcal{N\!N}_{L,15,1,2}$ with
	$L$ as in \eqref{eq:bound_depth_lc_cheb_f} such that
	\begin{equation}
		\norm{
			f
			-
			\left(
				(\Phi_{f,n,D,\epsilon})_1
				+
				\imath
				(\Phi_{f,n,D,\epsilon})_2
			\right)
		}_{L^\infty\left(\left(-D,D\right)\right)}
		\leq 
		\epsilon,
	\end{equation}
	where $(\Phi_{f,n,D,\epsilon})_1$ and $(\Phi_{f,n,D,\epsilon})_2$
	denote the first and second outputs of $\Phi_{f,n,D,\epsilon}$, respectively.
\end{itemize}
In either case
\begin{equation}\label{eq:explicit_constant_smooth_f}
	B_n
	\leq
	C
	\max\left\{
		n+2,
		4
		\left(
			\widetilde{C}_n
			D^{n+1}
			\lambda_{n+1}
		\right)^{\frac{1}{n-1}}
	\right\},
\end{equation}
and the weights are bounded according to
\begin{equation}
	\mathcal{B}
	\left(
		\Phi_{f,n,D,\epsilon}
	\right)
	\leq
	\max\left\{ 2,\lambda_n\right\},
\end{equation}
where $\widetilde{C}_n>0$ is independent of $\boldsymbol{\lambda}$.
\end{proposition}

\begin{proof}
We divide the proof of this result in three steps.

{\sf \encircle{1} 
Chebyshev Approximation for $D=1$.}
Firstly, we consider the $n$-term truncation
of the Chebyshev expansion for $f\in \mathcal{P}_{\boldsymbol{\lambda},1, \mathbb{R}}$
for $D =1$ and $\mathbb{K} = \mathbb{R}$.
For $f\in \mathcal{P}_{\boldsymbol{\lambda},1, \mathbb{R}}$ and  $n\in \mathbb{N}_0$ we set
\begin{align}
\label{eq:Cheb_exp}
	f_n(x)
	\coloneqq
	{\sum_{k=0}^{n}}^{\prime} a_k(f) \, T_k(x),
	\quad
	x\in [-1,1],
\end{align}
with $a_k(f) $ as in \eqref{eq:a_f_T}.
For each $n \in \IN_0$ it holds (see e.g.~\cite{rivlin2020chebyshev})
\begin{align}
\label{eq:error_trunc_Cheb}
	\norm{f-f_n}_{L^\infty((-1,1))}
	\leq
	\left(\frac{4}{\pi^2}\log(n+1) +4\right)
	\inf_{\pi_n \in \mathbb{P}_n}
	\norm{f-\pi_n}_{L^\infty((-1,1))}.
\end{align}
According to Jackson's theorem \cite[p.~26]{deBoor78}
for each $n \in \IN_0$ and $k\in\mathbb{N}$ such that $n>k+1$
it holds
\begin{align}
	\label{eq:best_approx_error}
	\inf_{\pi_n \in \mathbb{P}_n}
	\norm{f-\pi_n}_{L^\infty((-1,1))}
	\leq
	C_k \frac{2^k}{(n-1)^k} \omega\left(f^{(k)};\frac{1}{n-1-k} \right),
\end{align} 
where $C_k\coloneqq6 (3\text{e})^k/(1+k)$ and $\omega(f;h)$ denotes the
modulus of continuity of the function $f:[-1,1]\rightarrow \mathbb{R}$ at $h>0$, i.e.
\begin{align}
	\omega(f;h)
	\coloneqq
	\sup_{\substack{x,y\in [-1,1] \\ \snorm{x-y}\leq h}} \snorm{f(x)-f(y)}.
\end{align}
For any $g\in \mathscr{C}^{k+1}([-1,1])$ we have that
$\omega\left(g^{(k)}; h\right) \leq h \norm{g^{(k+1)}}_{L^\infty((-1,1))}$
for all $h>0$. By combining \eqref{eq:error_trunc_Cheb} and \eqref{eq:best_approx_error}
one readily observes that for all $n>k+1$ it holds
\begin{align}\label{eq:error_estimate_Cheb_approx}
	\norm{f-f_n}_{L^\infty((-1,1))}
	\leq
	\widetilde{C}_k
	n^{-k+1}
	\norm{f^{(k+1)}}_{L^\infty((-1,1))}
	\leq
	\widetilde{C}_k \lambda_{k+1}
	n^{-k+1},
\end{align} 
with $\widetilde{C}_k= 4^{k+1} C_k$,
which does not depend on
$\boldsymbol{\lambda}$.

For $k\in \mathbb{N}_{\geq2}$ and $\epsilon>0$  we set
\begin{align}\label{eq:index_n_k}
	n_{k}
	\coloneqq
	\max\left\{
		k+2,
		\left \lceil
			\left(2\frac{\widetilde{C}_k \lambda_{k+1}}{\epsilon} \right)^{\frac{1}{k-1}}
		\right \rceil
	\right\}
	\in
	\mathbb{N}.
\end{align}
Recalling \eqref{eq:error_estimate_Cheb_approx}, for any $\epsilon>0$
and for each $k\in \mathbb{N}$ one has
$$\norm{f-f_{n_{k}}}_{L^\infty((-1,1))}\leq\frac{\epsilon}{2}.$$

{\sf \encircle{2} 
ReLU-NN Emulation for $D=1$.}
Next, we consider the ReLU-NN emulation of function in 
$\mathcal{P}_{\boldsymbol{\lambda},D,\mathbb{R}}$
as the complex valued case follows from the exact same 
arguments.  

According to Proposition \ref{prop:approx_cheb_pol}, item (i),
there exists $C>0$ such that for each $k\in \mathbb{N}$ 
there exist networks
$\Phi_{f,k,\epsilon} \in \mathcal{N\!N}_{L,13,1,1}$ 
with
\begin{align}
\label{eq:depth_network_n}
	L
	\leq
	C  
	\left(
		n^2_{k} 
		+ 
		n_{k}  \log\left(\frac{2}{\epsilon}\right) 
		+ 
		n_{k} \log\left(\lambda_k\right)
	\right)
	\quad
	\text{as }
	\epsilon \rightarrow 0,
\end{align}
such that $\norm{\Phi_{f,k,\epsilon}-f_{n_{k}}}_{L^\infty((-1,1))} \leq \frac{\epsilon}{2}$,
with $f_{n_k}:[-1,1]\rightarrow \mathbb{R}$ as in \eqref{eq:Cheb_exp}.
For $\epsilon\in(0,1)$ we have that $1\leq \epsilon^{-\frac{1}{\ell}}$ for all $\ell \in \IN $
therefore for $k\in \mathbb{N}_{\geq2}$ 
\begin{align}
\label{eq:bound_nk}
	n_{k}
	\leq
	\max\left\{
		k+2,
		4
		\left(
			\widetilde{C}_k \lambda_{k+1}
		\right)^{\frac{1}{k-1}}
	\right\}
	\epsilon^{-\frac{1}{k-1}}.
\end{align}

Combining \eqref{eq:bound_nk} with \eqref{eq:depth_network_n}
and using the triangle inequality, we that conclude for each 
$k \in \mathbb{N}$ there exists a constant $B_k>0$ depending uniquely
on $\boldsymbol{\lambda}$ such that for each $f\in \mathcal{P}_{\boldsymbol\lambda,1,\mathbb{R}}$
there exist ReLU-NNs $\Phi_{f,k,\epsilon} \in  \mathcal{N\!N}_{L,13,1,1}$ 
with $L$ as in \eqref{eq:depth_network_n} and $n_k \in \mathbb{N}$
as in \eqref{eq:index_n_k}
such that $\norm{\Phi_{f,k,\epsilon}-f}_{L^\infty((-1,1))}\leq \epsilon$,
and with weights bounded in absolute value by
$
	\max
	\left\{
		2,
		\lambda_k
	\right\}.
$

{\sf \encircle{3} 
ReLU-NN Emulation for $D>1$.}
Next, we consider the case $D>1$. 
Let $f \in \mathcal{P}_{\boldsymbol{\lambda},D,\mathbb{R}}$
and define $\widehat{f}_D(x) = f\left(x D\right)$, for $x \in [-1,1]$.
If $f \in \mathcal{P}_{\boldsymbol{\lambda},D,\mathbb{R}}$
with $\boldsymbol{\lambda}\coloneqq\{\lambda_n\}_{n\in \mathbb{N}_0}$
then $\widehat{f}_D \in \mathcal{P}_{\widehat{\boldsymbol{\lambda}}_D,1,\mathbb{R}}$
with $\widehat{\boldsymbol{\lambda}}_D\coloneqq\{D^{k}\lambda_k\}_{k\in \mathbb{N}_0}$.

Therefore, for each $k\in \mathbb{N}_{\geq2}$ there exist ReLU-NNs
$\Phi_{\hat{f}_D,k,\epsilon} \in  \mathcal{N\!N}_{L,13,1,1}$ 
such that
\begin{equation}
	\norm{\Phi_{\hat{f}_D,k,\epsilon} -\widehat{f}_D}_{L^\infty((-1,1))}
	\leq 
	\epsilon.
\end{equation}
For each $k\in \IN_{\geq2}$ define
$\Phi_{f,k,\epsilon}(x) \coloneqq \Phi_{\widehat{f}_D,k,\epsilon}(x/D)$, thus yielding
\begin{equation}
	\norm{
		\Phi_{f,k,\epsilon} 
		-
		f
	}_{L^\infty((-D,D))}
	=
	\norm{
		\Phi_{\hat{f}_D,k,\epsilon} 
		-
		\widehat{f}_D
	}_{L^\infty((-1,1))}\leq \epsilon.
\end{equation}
The width of $\Phi_{f,k,\epsilon}$ is $13$ 
and the depth is bounded according to
\begin{equation}
\begin{aligned}
	\mathcal{L}
	\left(
		\Phi_{f,k,\epsilon}
	\right)
	\leq
	&
	\mathcal{L}
	\left(
		 \Phi_{\widehat{f}_D,k,\epsilon}
	\right)
	\\
	\leq
	&
	B_k
	\left(
		B_k
		\epsilon^{-\frac{2}{k-1}}
		+
		\epsilon^{-\frac{1}{k-1}}
		\log\left(\frac{2}{\epsilon}\right) 
		+
		\epsilon^{-\frac{1}{k-1}}
		\log
		\left(
			D
		\right)
		+
		\epsilon^{-\frac{1}{k-1}}
		\log
		\left(
			\lambda_k
		\right)
	\right)
\end{aligned}
\end{equation}
with 
\begin{equation}
	B_k
	\leq
	C
	\max\left\{
		k+2,
		4
		\left(
			\widetilde{C}_k 
			D^{k+1}
			\lambda_{k+1}
		\right)^{\frac{1}{k-1}}
	\right\},
\end{equation}
and $C>0$ independent of $k$ and $D$.
\end{proof}

\section{Auxiliary Results}
\label{sec:aux_res}
We introduce and prove technical results used in this work. 

\subsection{Product and Composition of ReLU-NNs}
\label{sec:PrdCmpReLU}

\begin{lemma}\label{lmm:product_of_real_NNs}
Let $\mathcal{I}_{\mathcal{D}} = (-D,D)$ for $D>0$.
Let $\alpha,\beta: {\mathcal{I}}_D\rightarrow \mathbb{R}$ be real-valued,
continuous functions in ${\mathcal{I}}_D$.
Assume there exist ReLU-NNs 
$\Phi_{\alpha,\epsilon} \in \mathcal{N\!N}_{L_\alpha(\epsilon),M_\alpha,1,1}$
and 
$\Phi_{\beta,\epsilon} \in \mathcal{N\!N}_{L_\beta(\epsilon),M_\beta,1,2}$
(i.e.~with depths $L_\alpha(\epsilon)$ and $L_\beta(\epsilon)$ 
depending on $\epsilon>0$ and fixed widths $M_\alpha$ and $M_\beta$) 
such that
\begin{equation}
	\norm{
		\alpha
		-
		\Phi_{\alpha,\epsilon}
	}_{L^\infty(\mathcal{I}_D)} 
	\leq 
	\epsilon
	\quad
	\text{and}
	\quad
	\norm{
		\beta
		-
		\Phi_{\beta,\epsilon}
	}_{L^\infty(\mathcal{I}_D)} 
	\leq 
	\epsilon.
\end{equation}
Then, for $\epsilon>0$ small enough there exist ReLU-NNs
$\Phi_{\times,\epsilon} \in \mathcal{N\!N}_{L,M,1,1}$ 
such that
\begin{equation}
	\norm{
		\alpha
		\beta
		-
		\Phi_{\times,\epsilon}
	}_{L^\infty(\mathcal{I}_D)} 
	\leq
	\epsilon,
\end{equation}
with $M$ given by
\begin{equation}
	M
	=
	\max\{
	2
	+
	M_\beta
	,
	\max
	\{
	M_\alpha,
	5
	\}
	\},
\end{equation}
and
\begin{equation}
\begin{aligned}
	L
	\leq
	&
	C
	\left(
	\log
	\left(
		\lceil U_{\alpha,\beta} \rceil
	\right)
	+
	\log
	\left(
		\frac{3}{\epsilon}
	\right)
	\right)
	\\
	&
	+
	\max
	\left\{
	L_\alpha
	\left(
	\frac{
		\epsilon
	}{
		3
		\norm{
			\beta
		}_{L^\infty(\mathcal{I}_\beta)}
	}
	\right)
	,
	\right\}
	+
	\max
	\left\{
	L_\beta
	\left(
		\frac{
		\epsilon
		}{
		6
		\norm{
			\alpha
		}_{L^\infty(\mathcal{I}_D)}
		}
	\right)
	,2
	\right\}.
\end{aligned}
\end{equation}
where 
$U_{\alpha,\beta} =2\max\left\{\norm{\alpha}_{L^\infty(\mathcal{I}_D)},\norm{\beta}_{L^\infty(\mathcal{I}_D)}\right\}$.
The weights of $\Phi_{\times,\epsilon}$ are bounded according to
\begin{equation}\label{eq:bound_mult_real_weights}
	\mathcal{B}
	\left(
		\Phi_{\times,\epsilon}
	\right)
	\leq
	\max
	\left\{
		\mathcal{B}
		\left(
			\Phi_{\alpha,\epsilon}
		\right),
		\mathcal{B}
		\left(
			\Phi_{\beta,\epsilon}
		\right),
		1	
	\right\}.
\end{equation}
\end{lemma}

\begin{proof}
Let $\mu_{U_{\alpha,\beta} ,\epsilon}$ 
be the multiplication ReLU-NN from Proposition \ref{prop:mult_network}
over the interval $(-U_{\alpha,\beta} ,U_{\alpha,\beta} )$ 
with accuracy $\frac{\epsilon}{3}>0$.
Define
$$
	\Phi_{\times,\epsilon}(x) 
	=
	\mu_{U_{\alpha,\beta},\epsilon}
	\left(
		\Phi_{\alpha,\epsilon}(x)
		,
		\Phi_{\beta,\epsilon}(x)
	\right),
	\quad
	x\in {\mathcal{I}_D}, 
$$
which is a ReLU-NN itself.
Firstly,
\begin{equation}\label{eq:error_approx_ot}
\begin{aligned}
	\norm{
		\alpha
		\beta
		-
		\Phi_{\times,\epsilon}
	}_{L^\infty(\mathcal{I}_D)} 
	\leq
	&
	\norm{
		\alpha
		\beta
		-
		\Phi_{\alpha,\epsilon}
		\Phi_{\beta,\epsilon}
	}_{L^\infty(\mathcal{I}_D)}
	+
	\norm{
		\Phi_{\alpha,\epsilon}
		\Phi_{\beta,\epsilon}
		-
		\Phi_{\times,\epsilon}
	}_{L^\infty(\mathcal{I}_D)}
	\\
	\leq
	&
	\norm{
		\beta
		\left(
			\alpha
			-
			\Phi_{\alpha,\epsilon}
		\right)
	}_{L^\infty(\mathcal{I}_D)}
	+
	\norm{
		\Phi_{\alpha,\epsilon}
		\left(
			\beta
			-
			\Phi_{\beta,\epsilon}
		\right)
	}_{L^\infty(\mathcal{I}_D)}
	\\
	&
	+
	\norm{
		\Phi_{\alpha,\epsilon}
		\Phi_{\beta,\epsilon}
		-
		\Phi_{\times,\epsilon}
	}_{L^\infty(\mathcal{I}_D)}
	\\
	\leq
	&
	\norm{
		\beta
	}_{L^\infty(\mathcal{I}_D)}
	\norm{
			\alpha
			-
			\Phi_{\alpha,\epsilon}
	}_{L^\infty(\mathcal{I}_D)}
	+
	\norm{
		\Phi_{\alpha,\epsilon}
	}_{L^\infty(\mathcal{I}_D)}
	\norm{
			\beta
			-
			\Phi_{\beta,\epsilon}
	}_{L^\infty(\mathcal{I}_D)}
	\\
	&
	+
	\norm{
		\Phi_{\alpha,\epsilon}
		\Phi_{\beta,\epsilon}
		-
		\Phi_{\times,\epsilon}
	}_{L^\infty(\mathcal{I}_D)}.
\end{aligned}
\end{equation}
We consider ReLU-NNs of accuracy
\begin{equation}\label{eq:error_bound_alpha_omega}
	\norm{
		\alpha
		-
		\Phi_{\alpha,\epsilon} 
	}_{L^\infty(\mathcal{I}_D)}
	\leq
	\frac{
		\epsilon
	}{
		3
		\norm{
			\beta
		}_{L^\infty(\mathcal{I}_D)}
	}
	,
	\quad
	\norm{
		\beta
		-
		\Phi_{\beta,\epsilon} 
	}_{L^\infty(\mathcal{I}_D)}
	\leq
	\frac{
		\epsilon
	}{
		6
		\norm{
			\alpha
		}_{L^\infty(\mathcal{I}_D)}
	},
\end{equation}
and observe that for
$\epsilon \in \left(0,3\norm{\alpha}_{L^\infty(\mathcal{I}_D)}\norm{\beta}_{L^\infty(\mathcal{I}_D)}\right)$ one has
$\norm{\Phi_{\alpha,\epsilon} }_{L^\infty(\mathcal{I}_D)}\leq 2\norm{\alpha}_{L^\infty(\mathcal{I}_D)}$.

Thus, we have that $\Phi_{\times,\epsilon} \in \mathcal{N\!N}_{L,M,1,1}$
\begin{equation}
	\norm{
		\alpha
		\beta
		-
		\Phi_{\times,\epsilon}
	}_{L^\infty(\mathcal{I}_D)} 
	\leq
	\epsilon
\end{equation}
We proceed to bound the width and depth of $\Phi_{f,\epsilon}$.
An inspection of the definition of $\Phi_{f,\epsilon}$ reveals that
\begin{equation}
	\Phi_{f,\epsilon}(x)
	=
	\left(
		\Lambda_{2,\epsilon}
		\circ
		\Lambda_{1,\epsilon}
	\right)(x),
	\quad
	x \in \mathcal{I}_D,
\end{equation}
where
\begin{equation}
	\Lambda_{1,\epsilon}(x)
	=
	\begin{pmatrix}
		x \\
		\Phi_{\beta,\epsilon}(x)
	\end{pmatrix}
	\quad
	\text{and}
	\quad
	\Lambda_{2,\epsilon}(x_1,x_2)
	=
	\mu_{U_{\alpha,\beta},\epsilon}
	\left(
		\Phi_{\alpha,\epsilon} (x_1)
		,
		x_2
	\right).
\end{equation}
Thus, we have
\begin{align}
	\mathcal{M}
	\left(
		\Phi_{\times,\epsilon}
	\right)
	=
	\max
	\left\{
		\mathcal{M}
		\left(
			\Lambda_{1,\epsilon}
		\right)
		,
		\mathcal{M}
		\left(
			\Lambda_{2,\epsilon}
		\right)
	\right\}
	=
	\max\{
	2
	+
	M_\beta
	,
	\max
	\{
	M_\alpha,
	5
	\}
	\},
\end{align}
and for the depth we have
\begin{equation}
\begin{aligned}
	\mathcal{L}
	\left(
		\Phi_{\times,\epsilon}
	\right)
	\leq
	&
	\mathcal{L}
	\left(
		\Lambda_{1,\epsilon}
	\right)
	+
	\mathcal{L}
	\left(
		\Lambda_{2,\epsilon}
	\right)
	\\
	\leq 
	&
	C
	\left(
	\log
	\left(
		\lceil U_{\alpha,\beta} \rceil
	\right)
	+
	\log
	\left(
		\frac{3}{\epsilon}
	\right)
	\right)
	\\
	&
	+
	\max
	\left\{
	L_\alpha
	\left(
	\frac{
		\epsilon
	}{
		3
		\norm{
			\beta
		}_{L^\infty(\mathcal{I}_D)}
	}
	\right)
	,
	2
	\right\}
	+
	\max
	\left\{
	L_\beta
	\left(
		\frac{
		\epsilon
		}{
		6
		\norm{
			\alpha
		}_{L^\infty(\mathcal{I}_D)}
		}
	\right)
	,2
	\right\}.
\end{aligned}
\end{equation}
Finally, the weights of $\Phi_{\times,\epsilon}$ 
are those of $\Phi_{\alpha,\epsilon}$,
$\Phi_{\beta,\epsilon} $, 
and 
of 
$\mu_{U_{\alpha,\beta},\epsilon}$,
thus yielding \eqref{eq:bound_mult_real_weights}.
\end{proof}

\begin{lemma}\label{lmm:general_result_oscillatory_texture}
Let $\mathcal{I}_D = (-D,D)$ for $D>0$.
Let $\omega: \R \rightarrow \R$ and  $\alpha,\beta: {\mathcal{I}_D}\rightarrow \R$
be continuously differentiable and continuous functions, respectively. 
Set $f(x) = \alpha(x)  (\omega \circ \beta)(x)$ for $x \in {\mathcal{I}_D}$.
Assume the following:
\begin{itemize}
	\item[(i)]
	There exist ReLU-NNs
	$\Phi_{\omega,\epsilon} \in \mathcal{N\!N}_{L_\omega(\epsilon,B),M_\omega,1,1}$
	(i.e.~with depth $L_\omega(\epsilon,B)$ depending
	on $\epsilon>0$ and $B>0$, and fixed width $M_\omega$)
	such that $\norm{\omega-\Phi_{\omega,\epsilon}}_{L^\infty((-B,B))} \leq \epsilon$
	for $\epsilon >0$ and any $B>0$. 
	\item[(ii)]
	In addition, assume that there exist ReLU-NNs
	$\Phi_{\alpha,\epsilon} \in \mathcal{N\!N}_{L_\alpha(\epsilon),M_\alpha,1,1}$ and
	$\Phi_{\beta,\epsilon} \in \mathcal{N\!N}_{L_\beta(\epsilon),M_\beta,1,1}$
	(i.e.~with depths $L_\alpha(\epsilon)$ and $L_\beta(\epsilon)$ depending
	on $\epsilon>0$ and fixed widths $M_\alpha$ and $M_\beta$)
	such that for $\epsilon >0$
	\begin{equation}
		\norm{\alpha-\Phi_{\alpha,\epsilon}}_{L^\infty(\mathcal{I}_D)} 
		\leq 
		\epsilon
		\quad
		\text{and}
		\quad
		\norm{\beta-\Phi_{\beta,\epsilon}}_{L^\infty(\mathcal{I}_D)} 
		\leq 
		\epsilon.
	\end{equation}
\end{itemize}
Then, there exist $C>0$ and ReLU-NNs
	$\Phi_{f,\epsilon} \in \mathcal{N\!N}_{L(\epsilon),M,1,1}$
	such that for $\epsilon>0$ small enough
	\begin{equation}
	\norm{
		f
		-
		\Phi_{f,\epsilon}
	}_{L^\infty(\mathcal{I}_D)}
	\leq
	\epsilon,
	\end{equation}
	with
	$
	M
	=
	\max\{
	2
	+
	\max\{M_\omega,M_\beta\}
	,
	\max
	\{
	M_\alpha,
	5
	\}
	\}
	$
	and
	\begin{equation}
	\begin{aligned}
	L(\epsilon)
	\leq 
	&
	C
	\left(
	\log
	\left(
		\lceil U_{\alpha,\omega} \rceil
	\right)
	+
	\log
	\left(
		\frac{3}{\epsilon}
	\right)
	\right)
	+
	L_\omega
	\left(
		\frac{
		\epsilon
		}{
		6
		\norm{
			\alpha
		}_{L^\infty(\mathcal{I}_D)}
		}
		,
		2\norm{\beta}_{L^\infty(\mathcal{I}_D)}
	\right)
	\\
	&
	+
	L_\alpha
	\left(
	\frac{
		\epsilon
	}{
		6
		\norm{
			\omega
		}_{L^\infty(\mathcal{I}_\beta)}
	}
	\right)
	+
	L_\beta
	\left(
		\frac{
		\epsilon
		}{
		6
		\norm{
			\alpha
		}_{L^\infty(\mathcal{I}_D)}
		\norm{
			\omega'
		}_{L^\infty(\mathcal{I}_\beta)}
		}
	\right),
\end{aligned}
\end{equation}
where $U_{\alpha,\omega} =2 \max\left\{\norm{\alpha}_{L^\infty(\mathcal{I}_{D})},\norm{\omega}_{L^\infty(\mathcal{I}_\beta)}\right\}$, and 
\begin{equation}\label{eq:interval_beta}
	\mathcal{I}_{\beta}
	\coloneqq
	\left(
		-2\norm{\beta}_{L^\infty(\mathcal{I}_D)}
		,
		2\norm{\beta}_{L^\infty(\mathcal{I}_D)}
	\right).
\end{equation}
The weights of $\Phi_{f,\epsilon}$ are bounded according to
\begin{equation}\label{eq:bound_mult_comp_weights}
	\mathcal{B}
	\left(
		\Phi_{f,\epsilon}
	\right)
	\leq
	\max
	\left\{
		\mathcal{B}
		\left(
			\Phi_{\alpha,\epsilon}
		\right),
		\mathcal{B}
		\left(
			\Phi_{\beta,\epsilon}
		\right),
		\mathcal{B}
		\left(
			\Phi_{\omega,\epsilon}
		\right),
		1	
	\right\}.
\end{equation}
\end{lemma}

\begin{proof}
Let $\mu_{U_{\alpha,\omega} ,\zeta}$ be the multiplication ReLU-NN from Proposition \ref{prop:mult_network}
over the interval $(-U_{\alpha,\omega} ,U_{\alpha,\omega} )$ with accuracy $\zeta>0$,
to be be specified ahead in terms of $\epsilon>0$.
Define
$$
	\Phi_{f,\epsilon}(x) 
	=
	\mu_{U_{\alpha,\omega} ,\zeta}(\Phi_{\alpha,\epsilon}(x),\left(\Phi_{\omega,\epsilon} 
	\circ \Phi_{\beta,\epsilon}\right)(x)),
	\quad
	x\in {\mathcal{I}_D}, 
$$
which is a ReLU-NN itself.

Firstly,
\begin{equation}\label{eq:error_approx_ot}
\begin{aligned}
	\norm{
		f
		-
		\Phi_{f,\epsilon}
	}_{L^\infty(\mathcal{I}_D)}
	\leq
	&
	\norm{
		\Phi_{f,\epsilon} 
		-
		\Phi_{\alpha,\epsilon} 
		\left(\Phi_{\omega,\epsilon}  \circ  \Phi_{\beta,\epsilon} \right)
	}_{L^\infty(\mathcal{I}_D)}
	+
	\norm{
		\Phi_{\alpha,\epsilon} 
		\left(\Phi_{\omega,\epsilon}  \circ \Phi_{\beta,\epsilon} \right) 
		-
		f
	}_{L^\infty(\mathcal{I}_D)}
	\\
	\leq
	&
	\norm{
		\Phi_{f,\epsilon} 
		-
		\Phi_{\alpha,\epsilon} 
		\left(
			\Phi_{\omega,\epsilon}  \circ  \Phi_{\beta,\epsilon} 
		\right) 
	}_{L^\infty(\mathcal{I}_D)} \\
	&
	+
	\norm{
		\Phi_{\alpha,\epsilon} 
		\left(
			\Phi_{\omega,\epsilon}  \circ  \Phi_{\beta,\epsilon} 
		\right)
		-
		\alpha
		\left( 
			\Phi_{\omega,\epsilon}  \circ \Phi_{\beta,\epsilon} 
		\right)
	}_{L^\infty(\mathcal{I}_D)}
	\\
	&
	+
	\norm{
		\alpha
		\left(
			\Phi_{\omega,\epsilon}  \circ \Phi_{\beta,\epsilon} 
		\right)
		-
		f
	}_{L^\infty(\mathcal{I}_D)}
	\\
	\leq
	&
	\underbrace{
	\norm{
		\Phi_{f,\epsilon} 
		 -
		 \Phi_{\alpha,\epsilon} 
		 \left(
			\Phi_{\omega,\epsilon}  \circ  \Phi_{\beta,\epsilon} 
		\right)
	}_{L^\infty(\mathcal{I}_D)}
	}_{(\heartsuit)}
	\\
	&
	+
	\underbrace{
	\norm{
		 \Phi_{\omega,\epsilon}  \circ  \Phi_{\beta,\epsilon} 
	}_{L^\infty(\mathcal{I}_D)}
	\norm{
		 \Phi_{\alpha,\epsilon} 
		 -
		 \alpha
	}_{L^\infty(\mathcal{I}_D)}
	}_{(\clubsuit)}
	\\
	&
	+
	\underbrace{
	\norm{
		\alpha
	}_{L^\infty(\mathcal{I}_D)}
	\norm{
		\Phi_{\omega,\epsilon}  \circ \Phi_{\beta,\epsilon}  
		-
		\omega \circ \beta 
	}_{L^\infty(\mathcal{I}_D)}
	}_{(\spadesuit)}.
\end{aligned}
\end{equation}
Let
\begin{equation}
	\epsilon
	\in
	\left(
		0
		,
		\min
		\left\{
		6
		\norm{
			\alpha
		}_{L^\infty(\mathcal{I}_D)}
		\norm{
			\omega'
		}_{L^\infty(\mathcal{I}_\beta)}
		\norm{
			\beta
		}_{L^\infty(\mathcal{I}_D)}
		,
		\norm{
			\beta
		}_{L^\infty(\mathcal{I}_D)}
		,
		3
		\norm{
			\alpha
		}_{L^\infty(\mathcal{I}_D)}
		\norm{
			\omega
		}_{L^\infty(\mathcal{I}_\beta)}
		\right\}
	\right).
\end{equation}
We consider ReLU-NNs of accuracy
\begin{equation}\label{eq:error_bound_alpha_omega}
	\norm{
		 \Phi_{\alpha,\epsilon} 
		 -
		 \alpha
	}_{L^\infty(\mathcal{I}_D)}
	\leq
	\frac{
		\epsilon
	}{
		6
		\norm{
			\omega
		}_{L^\infty(\mathcal{I}_\beta)}
	}
	,
	\quad
	\norm{
		 \Phi_{\beta,\epsilon} 
		 -
		 \beta
	}_{L^\infty(\mathcal{I}_D)}
	\leq
	\frac{
		\epsilon
	}{
		6
		\norm{
			\alpha
		}_{L^\infty(\mathcal{I}_D)}
		\norm{
			\omega'
		}_{L^\infty(\mathcal{I}_\beta)}
	},
\end{equation}
and
\begin{equation}
	\norm{
		 \Phi_{\omega,\epsilon} 
		 -
		 \omega
	}_{L^\infty(\mathcal{I}_\beta)}
	\leq
	\frac{
		\epsilon
	}{
		6
		\norm{
			\alpha
		}_{L^\infty(\mathcal{I}_D)}
	}.
\end{equation}
In addition, 
\begin{equation}\label{eq:bound_phi_beta}
\begin{aligned}
	\norm{\Phi_{\beta,\epsilon} }_{L^\infty(\mathcal{I}_D)}
	&
	\leq
	\norm{
		\Phi_{\beta,\epsilon} 
		-
		\beta
	}_{L^\infty(\mathcal{I}_D)}
	+
	\norm{
		\beta
	}_{L^\infty(\mathcal{I}_D)}
	\\
	&
	\leq
	\frac{
		\epsilon
	}{
		6
		\norm{
			\alpha
		}_{L^\infty(\mathcal{I}_D)}
		\norm{
			\omega'
		}_{L^\infty(\mathcal{I}_\beta)}
	}
	+
	\norm{
		\beta
	}_{L^\infty(\mathcal{I}_D)}
	\\
	&
	\leq
	2
	\norm{
		\beta
	}_{L^\infty(\mathcal{I}_D)}
\end{aligned}
\end{equation}
and
\begin{equation}\label{eq:bound_Phi_omega_Phi_beta}
\begin{aligned}
	\norm{
		\Phi_{\omega,\epsilon}  \circ \Phi_{\beta,\epsilon} 
	}_{L^\infty(\mathcal{I}_D)}
	\leq
	&
	\norm{
		\Phi_{\omega,\epsilon}  \circ \Phi_{\beta,\epsilon} 
		-
		\omega \circ  \Phi_{\beta,\epsilon} 
	}_{L^\infty(\mathcal{I}_D)}
	+
	\norm{
		\omega \circ \Phi_{\beta,\epsilon} 
		-
		\omega \circ \beta
	}_{L^\infty(\mathcal{I}_D)}
	\\
	&
	+
	\norm{
		\omega \circ \beta
	}_{L^\infty(\mathcal{I}_D)}
	\\
	\leq
	&
	\norm{
		\Phi_{\omega,\epsilon} 
		-
		\omega
	}_{L^\infty\left(\mathcal{I}_{\beta}\right)}
	+
	\norm{
		\omega'
	}_{L^\infty(\mathcal{I}_\beta)}
	\norm{
		\Phi_{\beta,\epsilon} 
		-
		\beta
	}_{L^\infty(\mathcal{I}_D)}
	\\
	&
	+
	\norm{
		\omega
	}_{L^\infty(\mathcal{I}_{\beta})}
	\\
	\leq
	&
	\frac{
		\epsilon
	}{
		6
		\norm{
			\alpha
		}_{L^\infty(\mathcal{I}_D)}
	}
	+
	\frac{
		\epsilon
	}{
		6
		\norm{
			\alpha
		}_{L^\infty(\mathcal{I}_D)}
	}
	+
	\norm{
		\omega
	}_{L^\infty(\mathcal{I}_{\beta})}
	\\
	=
	&
	\frac{
		\epsilon
	}{
		3
		\norm{
			\alpha
		}_{L^\infty(\mathcal{I}_D)}
	}
	+
	\norm{
		\omega
	}_{L^\infty(\mathcal{I}_{\beta})}
	\\
	\leq
	&
	2
	\norm{
		\omega
	}_{L^\infty(\mathcal{I}_{\beta})}.
\end{aligned}
\end{equation}
Next,
\begin{equation}\label{eq:bound_Phi_omega_Phi_beta_diff}
\begin{aligned}
	\norm{
		\Phi_{\omega,\epsilon}  \circ \Phi_{\beta,\epsilon}  
		-
		\omega \circ \beta 
	}_{L^\infty(\mathcal{I}_D)}
	\leq
	&
	\norm{
		\Phi_{\omega,\epsilon}  \circ \Phi_{\beta,\epsilon}  
		-
		\omega \circ \Phi_{\beta,\epsilon}  
	}_{L^\infty(\mathcal{I}_D)}
	\\
	&
	+
	\norm{
		\omega \circ \Phi_{\beta,\epsilon}  
		-
		\omega \circ \beta 
	}_{L^\infty(\mathcal{I}_D)}
	\\
	\leq
	&
	\norm{
		\Phi_{\omega,\epsilon} 
		-
		\omega
	}_{L^\infty(\mathcal{I}_{\beta})}
	+
	\norm{
		\omega'
	}_{L^\infty(\mathcal{I}_\beta)}
	\norm{
		\Phi_{\beta,\epsilon}  
		-
		\beta 
	}_{L^\infty(\mathcal{I}_D)}
	\\
	\leq
	&
	\frac{
		\epsilon
	}{
		6
		\norm{
			\alpha
		}_{L^\infty(\mathcal{I}_D)}
	}
	+
	\frac{
		\epsilon
	}{
		6
		\norm{
			\alpha
		}_{L^\infty(\mathcal{I}_D)}
	}
	=
	\frac{
		\epsilon
	}{
		3
		\norm{
			\alpha
		}_{L^\infty(\mathcal{I}_D)}
	}.
\end{aligned}
\end{equation}
From Proposition \ref{prop:mult_network} there exist $C>0$
and ReLU-NNs $\mu_{U_{\alpha,\omega},\zeta} \in\mathcal{N\!N}_{L,5,2,1}$ with 
$$
	L
	\leq 
	C
	\left(
	\log
	\left(
		\lceil U_{\alpha,\omega} \rceil
	\right)
	+
	\log
	\left(
		\frac{1}{\zeta}
	\right)
	\right)
$$
as $\zeta \rightarrow 0$ and satisfying
$$\norm{\mu_{U_{\alpha,\omega},\zeta}(x,y)-xy}_{L^\infty((-U_{\alpha,\omega},U_{\alpha,\omega})^2)}\leq \zeta.$$
Thus,
\begin{equation}\label{eq:error_mult_net}
	(\heartsuit)
	=
	\norm{
		\mu_{U_{\alpha,\omega},\zeta}
		\left(
			\Phi_{\alpha,\epsilon} 
			,
			\Phi_{\omega,\epsilon}  \circ \Phi_{\beta,\epsilon} 
		\right)
		 -
		\Phi_{\alpha,\epsilon} 
		(\Phi_{\omega,\epsilon}  \circ  \Phi_{\beta,\epsilon} )
	}_{L^\infty((-U_{\alpha,\omega},U_{\alpha,\omega})^2)}
	\leq
	\zeta
	=
	\frac{\epsilon}{3}
\end{equation}
as a consequence of \eqref{eq:bound_phi_beta} and 
\eqref{eq:bound_Phi_omega_Phi_beta}.

Recalling \eqref{eq:error_bound_alpha_omega},
\eqref{eq:bound_Phi_omega_Phi_beta}, and 
\eqref{eq:bound_Phi_omega_Phi_beta_diff}
\begin{equation}
	(\clubsuit)
	\leq
	2
	\norm{
		\omega
	}_{L^\infty(\mathcal{I}_{\beta})}
	\frac{
		\epsilon
	}{
		6
		\norm{
			\omega
		}_{L^\infty(\mathcal{I}_\beta)}
	}
	=
	\frac{\epsilon}{3}
	\quad
	\text{and}
	\quad
	(\spadesuit)
	\leq
	\norm{
		\alpha
	}_{L^\infty(\mathcal{I})}
	\frac{
		\epsilon
	}{
		3
		\norm{
			\alpha
		}_{L^\infty(\mathcal{I}_D)}
	}
	=
	\frac{\epsilon}{3}.
\end{equation}
Next, from \eqref{eq:error_approx_ot}
and \eqref{eq:error_mult_net} we conclude
\begin{equation}
	\norm{
		f
		-
		\Phi_{f,\epsilon}
	}_{L^\infty(\mathcal{I}_D)}
	\leq
	\epsilon.
\end{equation}
We proceed to bound the width and depth of $\Phi_{f,\epsilon}$.
An inspection of the definition of $\Phi_{f,\epsilon}$ reveals that
\begin{equation}
	\Phi_{f,\epsilon}(x)
	=
	\left(
		\Lambda_{2,\epsilon}
		\circ
		\Lambda_{1,\epsilon}
	\right)(x),
	\quad
	x \in \mathcal{I}_D,
\end{equation}
where
\begin{equation}
	\Lambda_{1,\epsilon}(x)
	=
	\begin{pmatrix}
		x \\
		\left(\Phi_{\omega,\epsilon}  \circ  \Phi_{\beta,\epsilon} \right)(x)
	\end{pmatrix}
	\quad
	\text{and}
	\quad
	\Lambda_{2,\epsilon}(x_1,x_2)
	=
	\mu_{U_{\alpha,\omega},\zeta}
	\left(
		\Phi_{\alpha,\epsilon} (x_1)
		,
		x_2
	\right).
\end{equation}
Thus, we have
\begin{align}
	\mathcal{M}
	\left(
		\Phi_{f,\epsilon}
	\right)
	=
	\max
	\left\{
		\mathcal{M}
		\left(
			\Lambda_{1,\epsilon}
		\right)
		,
		\mathcal{M}
		\left(
			\Lambda_{2,\epsilon}
		\right)
	\right\}
	=
	\max\{
	2
	+
	\max\{M_\omega,M_\beta\}
	,
	\max
	\{
	M_\alpha,
	5
	\}
	\},
\end{align}
and for the depth we have
\begin{equation}
\begin{aligned}
	\mathcal{L}
	\left(
		\Phi_{f,\epsilon}
	\right)
	\leq
	&
	\mathcal{L}
	\left(
		\Lambda_{1,\epsilon}
	\right)
	+
	\mathcal{L}
	\left(
		\Lambda_{2,\epsilon}
	\right)
	\\
	\leq 
	&
	C
	\left(
	\log
	\left(
		\lceil U_{\alpha,\omega} \rceil
	\right)
	+
	\log
	\left(
		\frac{3}{\epsilon}
	\right)
	\right)
	+
	L_\omega
	\left(
		\frac{
		\epsilon
		}{
		6
		\norm{
			\alpha
		}_{L^\infty(\mathcal{I}_D)}
		}
		,
		2\norm{\beta}_{L^\infty(\mathcal{I}_D)}
	\right)
	\\
	&
	+
	L_\alpha
	\left(
	\frac{
		\epsilon
	}{
		6
		\norm{
			\omega
		}_{L^\infty(\mathcal{I}_\beta)}
	}
	\right)
	+
	L_\beta
	\left(
		\frac{
		\epsilon
		}{
		6
		\norm{
			\alpha
		}_{L^\infty(\mathcal{I}_D)}
		\norm{
			\omega'
		}_{L^\infty(\mathcal{I}_\beta)}
		}
	\right).
\end{aligned}
\end{equation}
Finally, the weights of $\Phi_{f,\epsilon}$ 
are those of $\Phi_{\alpha,\epsilon}$,
$\Phi_{\beta,\epsilon} $, 
$\Phi_{\omega,\epsilon}$, 
and of
$\mu_{U_{\alpha,\omega},\zeta}$,
thus yielding \eqref{eq:bound_mult_comp_weights}.
\end{proof}

\begin{lemma}\label{lmm:product_of_complex_NNs}
Let $\mathcal{I}_{\mathcal{D}} = (-D,D)$ for $D>0$.
Let $\alpha,\beta: {\mathcal{I}}_D\rightarrow \mathbb{C}$ be complex-valued,
continuous functions in ${\mathcal{I}}_D$.
Assume there exist ReLU-NNs 
$\Phi_{\alpha,\epsilon} \in \mathcal{N\!N}_{L_\alpha(\epsilon),M_\alpha,1,2}$
and 
$\Phi_{\beta,\epsilon} \in \mathcal{N\!N}_{L_\beta(\epsilon),M_\beta,1,2}$
(i.e.~with depths $L_\alpha(\epsilon)$ and $L_\beta(\epsilon)$ 
depending on $\epsilon>0$ and fixed widths $M_\alpha$ and $M_\beta$) 
such that
\begin{equation}
	\norm{
		\alpha
		-
		\left(
			\Phi_{\alpha,\epsilon}
		\right)_1
		-
		\imath
		\left(
			\Phi_{\alpha,\epsilon}
		\right)_2
	}_{L^\infty(\mathcal{I}_D)} 
	\leq 
	\epsilon
	\quad
	\text{and}
	\quad
	\norm{
		\beta
		-
		\left(
			\Phi_{\beta,\epsilon}
		\right)_1
		-
		\imath
		\left(
			\Phi_{\beta,\epsilon}
		\right)_2
	}_{L^\infty(\mathcal{I}_D)} 
	\leq 
	\epsilon.
\end{equation}
Then, for $\epsilon>0$ small enough there exist ReLU-NNs
$\Phi_{\times,\epsilon} \in \mathcal{N\!N}_{L,M,1,2}$ 
such that
\begin{equation}
	\norm{
		\alpha
		\beta
		-
		\left(
			\Phi_{\times,\epsilon}
		\right)_1
		-
		\imath
		\left(
			\Phi_{\times,\epsilon}
		\right)_2
	}_{L^\infty(\mathcal{I}_D)} 
	\leq
	\epsilon,
\end{equation}
with $M$ given by
\begin{equation}
	M
	=
	\max
	\left \{
		4+M_\alpha,
		2+M_\beta, 
		2
		+
		\max
		\left\{
			10
			,
			M_\alpha
		\right\},
		2
		+
		\max
		\left\{
			10
			,
			M_\beta
		\right\}
	\right\},
\end{equation}
and with depth bound
\begin{equation}
	L
	\leq
	C
	\left(
	\log
	\left(
		\lceil U_{\alpha,\beta} \rceil
	\right)
	+
	\log
	\left(
		\frac{4}{\epsilon}
	\right)
	\right)
	+
	2
	\max
	\left\{
		L_\alpha
		\left(
			\frac{\epsilon}{4}
		\right)
		,
		2
	\right\}
	+
	2
	\max
	\left\{
		L_\beta
		\left(
			\frac{\epsilon}{4}
		\right)
		,
		2
	\right\},
\end{equation}
where 
$U_{\alpha,\beta} =2\max\left\{\norm{\alpha}_{L^\infty(\mathcal{I}_D)},\norm{\beta}_{L^\infty(\mathcal{I}_D)}\right\}$.

The weights of $\Phi_{\times,\epsilon}$ are bounded according to
\begin{equation}\label{eq:bound_mult_comp_weights_2}
	\mathcal{B}
	\left(
		\Phi_{\times,\epsilon}
	\right)
	\leq
	\max
	\left\{
		\mathcal{B}
		\left(
			\Phi_{\alpha,\epsilon}
		\right),
		\mathcal{B}
		\left(
			\Phi_{\beta,\epsilon}
		\right),
		1	
	\right\}.
\end{equation}
\end{lemma}

\begin{proof}
Observe that
\begin{equation}
\begin{aligned}
	\alpha
	\beta
	&
	=
	\left(
		\Re\{\alpha \}
		+
		\imath
		\Im\{\alpha \}
	\right)
	\left(
		\Re\{\beta \}
		+
		\imath
		\Im\{\beta \}
	\right)
	\\
	&
	=
	\Re\{\alpha \}
	\Re\{\beta \}
	-
	\Im\{\alpha \}
	\Im\{\beta \}
	+
	\imath
	\left(
		\Re\{\alpha \}
		\Im\{\beta \}
		+
		\Im\{\alpha \}
		\Re\{\beta \}
	\right).
\end{aligned}
\end{equation}
Define the ReLU-NNs
\begin{equation}
\begin{aligned}
	\Phi^{(1)}_{\times,\epsilon}
	(x)
	&
	=
	\mu_{U_{\alpha,\beta},\epsilon}
	\left(
		\left(
			\Phi_{\alpha,\epsilon}(x)
		\right)_1
		,
		\left(
			\Phi_{\beta,\epsilon}(x)
		\right)_1
	\right)
	,
	&
	\quad
	\Phi^{(2)}_{\times,\epsilon}
	(x)
	&
	=
	\mu_{U_{\alpha,\beta},\epsilon}
	\left(
		\left(
			\Phi_{\alpha,\epsilon}(x)
		\right)_2
		,
		\left(
			\Phi_{\beta,\epsilon}(x)
		\right)_2
	\right),
	\\
	\Phi^{(3)}_{\times,\epsilon}
	(x)
	&
	=
	\mu_{U_{\alpha,\beta},\epsilon}
	\left(
		\left(
			\Phi_{\alpha,\epsilon}(x)
		\right)_1
		,
		\left(
			\Phi_{\beta,\epsilon}(x)
		\right)_2
	\right)
	,
	&
	\quad
	\Phi^{(4)}_{\times,\epsilon}
	(x)
	&
	=
	\mu_{U_{\alpha,\beta},\epsilon}
	\left(
		\left(
			\Phi_{\alpha,\epsilon}(x)
		\right)_2
		,
		\left(
			\Phi_{\beta,\epsilon}(x)
		\right)_1
	\right),
\end{aligned}
\end{equation}
and $\mu_{U_{\alpha,\beta},\epsilon} \in\mathcal{N\!N}_{L,5,2,1}$ with
$
	L
	\leq
	C
	\left(
	\log
	\left(
		\lceil U_{\alpha,\beta} \rceil
	\right)
	+
	\log
	\left(
		\frac{4}{\epsilon}
	\right)
	\right)
$ for $\epsilon \in (0,1/8)$, and satisfying
\begin{equation}
	\norm{
		xy
		-
		\mu_{U_{\alpha,\beta},\epsilon}(x,y)
	}_{L^\infty(U_{\alpha,\beta})}
	\leq
	\frac{\epsilon}{4}
\end{equation}
is the multiplication ReLU-NN of Proposition \ref{prop:mult_network}.
Set
\begin{equation}
	\Phi_{\times,\epsilon}
	(x)
	=
	\begin{pmatrix}
		\Phi^{(1)}_{\times,\epsilon}
		(x)
		-
		\Phi^{(2)}_{\times,\epsilon}
		(x)
		\\
		\Phi^{(3)}_{\times,\epsilon}
		(x)
		+
		\Phi^{(4)}_{\times,\epsilon}
		(x)
	\end{pmatrix},
	\quad
	x \in \mathcal{I}_D.
\end{equation}
It holds that
\begin{equation}
	\norm{
		\Re\{\alpha \}
		\Re\{\beta \}
		-
		\Phi^{(1)}_{\times,\epsilon}
	}_{L^\infty(\mathcal{I}_D)}
	\leq
	\frac{\epsilon}{4}.
\end{equation}
The same holds for $\Phi^{(i)}_{\times,\epsilon}$, $i=2,3,4$,
thus
\begin{equation}
	\norm{
		\alpha
		\beta
		-
		\left(
			\left(\Phi_{\times,\epsilon}\right)_1
			+
			\imath
			\left(\Phi_{\times,\epsilon}\right)_2
		\right)
	}_{L^\infty(\mathcal{I}_D)}
	\leq
	\epsilon.
\end{equation}
Let us define the ReLU-NNs
\begin{equation}
	\Lambda_{1,\epsilon}
	(x_1,x_2,x_3)
	\coloneqq
	\begin{pmatrix}
		1 & -1   
	\end{pmatrix}
	\begin{pmatrix}
		\mu_{U_{\alpha,\beta},\epsilon}\left(\left(\Phi_{\alpha,\epsilon}(x_1)\right)_1,x_2\right) \\
		\mu_{U_{\alpha,\beta},\epsilon}\left(\left(\Phi_{\alpha,\epsilon}(x_1)\right)_2,x_3\right) 
	\end{pmatrix}
\end{equation}
and
\begin{equation}
	\Lambda_{2,\epsilon}
	(x_1,x_2,x_3)
	\coloneqq
	\begin{pmatrix}
		1 & 1
	\end{pmatrix}
	\begin{pmatrix}
		\mu_{U_{\alpha,\beta},\epsilon}\left(x_3,\left(\Phi_{\beta,\epsilon}(x_1)\right)_2\right) \\
		\mu_{U_{\alpha,\beta},\epsilon}\left(x_2,\left(\Phi_{\beta,\epsilon}(x_1)\right)_1\right) 
	\end{pmatrix}.
\end{equation}
Next, we set
\begin{equation}
\begin{aligned}
	\Theta_{1,\epsilon}(x)
	&
	\coloneqq
	\begin{pmatrix}
		x \\
		\left(\Phi_{\beta,\epsilon}(x)\right)_1 \\
		\left(\Phi_{\beta,\epsilon}(x)\right)_2
	\end{pmatrix},
	&
	\quad
	\Theta_{2,\epsilon}(x_1,x_2,x_3)
	&
	\coloneqq
	\begin{pmatrix}
		\Lambda_{1,\epsilon} (x_1,x_2,x_3) \\
		x_1
	\end{pmatrix},
	\\
	\Theta_{3,\epsilon}(x_1,x_2)
	&
	\coloneqq
	\begin{pmatrix}
		x_1 \\
		x_2 \\
		\left(\Phi_{\alpha,\epsilon}(x_2)\right)_1 \\
		\left(\Phi_{\alpha,\epsilon}(x_2)\right)_2
	\end{pmatrix},
	&
	\quad
	\Theta_{4,\epsilon}(x_1,x_2,x_3)
	&
	\coloneqq
	\begin{pmatrix}
		x _1\\
		\Lambda_{2,\epsilon}
		(x_2,x_3,x_4)
	\end{pmatrix}.
\end{aligned}
\end{equation}
Thus
\begin{equation}
	\Phi_{\times,\epsilon}
	(x)
	=
	\left(
		\Theta_{4,\epsilon}
		\circ
		\Theta_{3,\epsilon}
		\circ
		\Theta_{2,\epsilon}
		\circ
		\Theta_{1,\epsilon}
	\right)
	(x),
	\quad
	x \in \mathcal{I}_D.
\end{equation}
We have
\begin{equation}
\begin{aligned}
	\mathcal{M}
	\left(
		\Theta_{1,\epsilon}
	\right)
	&
	=
	2+M_\beta,
	&
	\quad
	\mathcal{M}
	\left(
		\Theta_{2,\epsilon}
	\right)
	&
	=
	2
	+
	\max
	\left\{
		10
		,
		M_\alpha
	\right\},
	\\
	\mathcal{M}
	\left(
		\Theta_{3,\epsilon}
	\right)
	&
	=
	4+M_\alpha,
	&
	\quad
	\mathcal{M}
	\left(
		\Theta_{4,\epsilon}
	\right)
	&
	=
	2
	+
	\max
	\left\{
		10
		,
		M_\beta
	\right\},
\end{aligned}
\end{equation}
and
\begin{equation}
\begin{aligned}
	\mathcal{L}
	\left(
		\Theta_{1,\epsilon}
	\right)
	&
	\leq
	\max
	\left\{
		L_\beta
		\left(
			\frac{\epsilon}{4}
		\right)
		,
		2
	\right\},
	\\
	\mathcal{L}
	\left(
		\Theta_{2,\epsilon}
	\right)
	&
	\leq
	C
	\left(
	\log
	\left(
		\lceil U_{\alpha,\beta} \rceil
	\right)
	+
	\log
	\left(
		\frac{4}{\epsilon}
	\right)
	\right)
	+
	\max
	\left\{
	L_\alpha
	\left(
	\frac{
		\epsilon
	}{
		12
		\norm{
			\beta
		}_{L^\infty(\mathcal{I}_\beta)}
	}
	\right)
		,
		2
	\right\},
	\\
	\mathcal{L}
	\left(
		\Theta_{3,\epsilon}
	\right)
	&
	\leq
	\max
	\left\{
		L_\alpha
		\left(
			\frac{\epsilon}{4}
		\right)
		,
		2
	\right\},
	\\
	\mathcal{L}
	\left(
		\Theta_{4,\epsilon}
	\right)
	&
	\leq
	C
	\left(
	\log
	\left(
		\lceil U_{\alpha,\beta} \rceil
	\right)
	+
	\log
	\left(
		\frac{4}{\epsilon}
	\right)
	\right)
	+
	\max
	\left\{
	L_\beta
	\left(
		\frac{
		\epsilon
		}{
		24
		\norm{
			\alpha
		}_{L^\infty(\mathcal{I}_D)}
		}
	\right)
		,
		2
	\right\}.
\end{aligned}
\end{equation}
Thus $\Phi_{\times,\epsilon} \in \mathcal{N\!N}_{L,M,1,2}$ with
\begin{equation}
	M
	=
	\max_{i=1,\dots,4}
	\left\{
		\mathcal{M}
		\left(
			\Theta_{i,\epsilon}
		\right)
	\right\},
\end{equation}
and
\begin{equation}
\begin{aligned}
	L
	&
	\leq
	\sum_{i=4}^{4}
	\mathcal{L}
	\left(
		\Theta_{i,\epsilon}
	\right)
	\\
	\leq
	&
	2
	C
	\left(
	\log
	\left(
		\lceil U_{\alpha,\beta} \rceil
	\right)
	+
	\log
	\left(
		\frac{4}{\epsilon}
	\right)
	\right)
	+
	2
	\max
	\left\{
	L_\alpha
	\left(
	\frac{
		\epsilon
	}{
		12
		\norm{
			\beta
		}_{L^\infty(\mathcal{I}_\beta)}
	}
	\right)
		,
		2
	\right\}
	\\
	&
	+
	2
	\max
	\left\{
	L_\beta
	\left(
		\frac{
		\epsilon
		}{
		24
		\norm{
			\alpha
		}_{L^\infty(\mathcal{I}_D)}
		}
	\right)
		,
		2
	\right\},
\end{aligned}
\end{equation}
with $C>0$ as in Proposition \ref{prop:mult_network}.
This concludes the proof of this result.
\end{proof}

Using Lemma \ref{lmm:general_result_oscillatory_texture}
and Lemma \ref{lmm:product_of_complex_NNs} we get the following result.

\begin{lemma}\label{lmm:approx_comp_mult_complex}
Let $\mathcal{I}_D = (-D,D)$ for $D>0$.
Let $\omega: \R \rightarrow \mathbb{C}$
and $\alpha: \overline{\mathcal{I}}_D\rightarrow \mathbb{C}$,
$\beta: \overline{\mathcal{I}}_D\rightarrow \R$
be continuously differentiable and continuous functions, respectively. 
Set $f(x) = \alpha(x)  (\omega \circ \beta)(x)$ for $x \in \mathcal{I}_D$.
Assume the following:
\begin{itemize}
	\item[(i)]
	There exist ReLU-NNs
	$\Phi_{\omega,\epsilon} \in \mathcal{N\!N}_{L_\omega(\epsilon,B),M_\omega,1,2}$
	(i.e.~with depth $L_\omega(\epsilon,B)$ depending
	on $\epsilon>0$ and $B>0$, and fixed width $M_\omega$)
	such that for $\epsilon >0$ and any $B>0$ it holds
	\begin{equation}
		\norm{
			\omega
			-
			\left(\Phi_{\omega,\epsilon} \right)_1
			-
			\imath
			\left(\Phi_{\omega,\epsilon} \right)_2
		}_{L^\infty((-B,B))} 
		\leq 
		\epsilon.
	\end{equation}
	\item[(ii)]
	In addition, assume that there exist ReLU-NNs
	$\Phi_{\alpha,\epsilon}  \in \mathcal{N\!N}_{L_\alpha(\epsilon),M_\alpha,1,2}$ 
        and
	$\Phi_{\beta,\epsilon}  \in \mathcal{N\!N}_{L_\beta(\epsilon),M_\beta,1,1}$
	(i.e.~with depths $L_\alpha(\epsilon)$ and $L_\beta(\epsilon)$ 
        depending on $\epsilon>0$ 
        and fixed widths $M_\alpha$ and $M_\beta$)
	such that for $\epsilon >0$ it holds
	\begin{equation}
		\norm{
			\alpha
			-
			\left(
				\left(\Phi_{\alpha,\epsilon} \right)_1
				+
				\imath
				\left(\Phi_{\alpha,\epsilon} \right)_2
			\right)
		}_{L^\infty(\mathcal{I}_D)} 
		\leq 
		\epsilon
		\quad
		\text{and}
		\quad
		\norm{\beta-\Phi_{\beta,\epsilon} }_{L^\infty(\mathcal{I}_D)} 
		\leq 
		\epsilon.
	\end{equation}
\end{itemize}
Then, 
there exist $C>0$ and ReLU-NNs
$\Phi_{f,\epsilon} \in \mathcal{N\!N}_{L(\epsilon),M,1,1}$
such that for $\epsilon>0$ small enough
\begin{equation}
	\norm{
		f
		-
		\Phi_{f,\epsilon}
	}_{L^\infty(\mathcal{I}_D)}
	\leq
	\epsilon,
\end{equation}
with
\begin{equation}
	M
	=
	\max\{
	2
	+
	\max\{M_\omega,M_\beta\}
	,
	2
	+
	\max
	\left\{
		10
		,
		M_\alpha
	\right\}
	,
	4+M_\alpha
	,
	2
	+
	\max
	\left\{
		10
		,
		M_\omega
		,
		M_\beta
	\right\}
	\},
\end{equation}
and
\begin{equation}
\begin{aligned}
	L(\epsilon)
	\leq 
	&
	C
	\left(
	\log
	\left(
		\lceil U_{\alpha,\omega} \rceil
	\right)
	+
	\log
	\left(
		\frac{12}{\epsilon}
	\right)
	\right)
	+
	2
	\max
	\left\{
		L_\alpha
		\left(
		\frac{
			\epsilon
		}{
			24
			\norm{
				\alpha
			}_{L^\infty(\mathcal{I})}
			\norm{
				\omega'
			}_{L^\infty(\mathcal{I}_\beta)}
		}
		\right)
		,
		2
	\right\},
	\\
	&
	+
	2
	\max
	\left\{
		L_\omega
		\left(
		\frac{
			\epsilon
		}{
			24
			\norm{
				\alpha
			}_{L^\infty(\mathcal{I}_D)}
		}
		,
		2\norm{\beta}_{L^\infty(\mathcal{I}_D)}
		\right)
		+
		L_\beta
		\left(
		\frac{
			\epsilon
		}{
			24
			\norm{
				\alpha
			}_{L^\infty(\mathcal{I})}
			\norm{
				\omega'
			}_{L^\infty(\mathcal{I}_\beta)}
		}
		\right)
		,
		2
	\right\},
\end{aligned}
\end{equation}
where $U_{\alpha,\omega} =
2 \max\left\{\norm{\alpha}_{L^\infty(\mathcal{I}_{D})},\norm{\omega}_{L^\infty(\mathcal{I}_{\beta})}\right\}$.

The weights of $\Phi_{f,\epsilon}$ are bounded according to
\begin{equation}\label{eq:bound_mult_comp_complex_weights}
	\mathcal{B}
	\left(
		\Phi_{f,\epsilon}
	\right)
	\leq
	\max
	\left\{
		\mathcal{B}
		\left(
			\Phi_{\alpha,\epsilon}
		\right),
		\mathcal{B}
		\left(
			\Phi_{\beta,\epsilon}
		\right),
		\mathcal{B}
		\left(
			\Phi_{\omega,\epsilon}
		\right),
		1	
	\right\}.
\end{equation}
\end{lemma}

\begin{proof}
The construction 
is similar to that of Lemma \ref{lmm:product_of_complex_NNs}.
Firstly
\begin{equation}
\begin{aligned}
	\alpha  (\omega \circ \beta)
	&
	=
	\left(
		\Re\{\alpha \}
		+
		\imath
		\Im\{\alpha \}
	\right)
	\left(
		\Re\{ \omega \circ \beta \}
		+
		\imath
		\Im\{\omega \circ \beta \}
	\right)
	\\
	&
	=
	\Re\{\alpha \}
	\Re\{\omega \circ \beta \}
	-
	\Im\{\alpha \}
	\Im\{\omega \circ \beta \}
	+
	\imath
	\left(
		\Re\{\alpha \}
		\Im\{\omega \circ \beta \}
		+
		\Im\{\alpha \}
		\Re\{\omega \circ \beta \}
	\right).
\end{aligned}
\end{equation}
As in the proof of Lemma \ref{lmm:general_result_oscillatory_texture}
we define the ReLU-NNs
\begin{equation}
\begin{aligned}
	\Phi^{(1)}_{\times,\epsilon}
	(x)
	&
	=
	\mu_{U_{\alpha,\omega},\epsilon}
	\left(
		\left(
			\Phi_{\alpha,\epsilon}(x)
		\right)_1
		,
		\left(
			\left(
			\Phi_{\omega,\epsilon}
			\circ
			\Phi_{\beta,\epsilon}
			\right)
			(x)
		\right)_1
	\right),
	\\
	\Phi^{(2)}_{\times,\epsilon}
	(x)
	&
	=
	\mu_{U_{\alpha,\omega},\epsilon}
	\left(
		\left(
			\Phi_{\alpha,\epsilon}(x)
		\right)_2
		,
		\left(
			\left(
			\Phi_{\omega,\epsilon}
			\circ
			\Phi_{\beta,\epsilon}
			\right)
			(x)
		\right)_2
	\right),
	\\
	\Phi^{(3)}_{\times,\epsilon}
	(x)
	&
	=
	\mu_{U_{\alpha,\omega},\epsilon}
	\left(
		\left(
			\Phi_{\alpha,\epsilon}(x)
		\right)_1
		,
		\left(
			\left(
			\Phi_{\omega,\epsilon}
			\circ
			\Phi_{\beta,\epsilon}
			\right)
			(x)
		\right)_2
	\right)
	,
	\\
	\Phi^{(4)}_{\times,\epsilon}
	(x)
	&
	=
	\mu_{U_{\alpha,\omega},\epsilon}
	\left(
		\left(
			\Phi_{\alpha,\epsilon}(x)
		\right)_2
		,
		\left(
			\left(
			\Phi_{\omega,\epsilon}
			\circ
			\Phi_{\beta,\epsilon}
			\right)
			(x)
		\right)_1
	\right),
\end{aligned}
\end{equation}
and 
$\mu_{U_{\alpha,\omega},\epsilon} \in\mathcal{N\!N}_{L,5,2,1}$ 
with
$
	L
	\leq 
	C
	\left(
	\log
	\left(
		\lceil U_{\alpha,\omega} \rceil
	\right)
	+
	\log
	\left(
		\frac{12}{\epsilon}
	\right)
	\right)
$, as $\epsilon \rightarrow 0$,
is the multiplication ReLU-NN of Proposition \ref{prop:mult_network}
satisfying
$$
	\norm{
		\mu_{U_{\alpha,\omega},\epsilon}(x,y)-xy
	}_{L^\infty((-U_{\alpha,\omega},U_{\alpha,\omega})^2)}
	\leq
	\frac{\epsilon}{12}.
$$
Set
\begin{equation}
	\Phi_{\times,\epsilon}
	(x)
	=
	\begin{pmatrix}
		\Phi^{(1)}_{\times,\epsilon}
		(x)
		-
		\Phi^{(2)}_{\times,\epsilon}
		(x)
		\\
		\Phi^{(3)}_{\times,\epsilon}
		(x)
		+
		\Phi^{(4)}_{\times,\epsilon}
		(x)
	\end{pmatrix}.
\end{equation}
From Lemma \ref{lmm:general_result_oscillatory_texture}
\begin{equation}\label{eq.accuracy_NNs_prod_comp}
\begin{aligned}
	\norm{
		\Re\{\alpha \}
		\Re\{\omega \circ \beta \}
		-
		\Phi^{(1)}_{\times,\epsilon}
	}_{L^\infty(\mathcal{I}_D)}
	&
	\leq
	\frac{\epsilon}{4},
	&
	\quad
	\norm{
		\Im\{\alpha \}
		\Im\{\omega \circ \beta \}
		-
		\Phi^{(2)}_{\times,\epsilon}
	}_{L^\infty(\mathcal{I}_D)}
	&
	\leq
	\frac{\epsilon}{4},
	\\
	\norm{
		\Re\{\alpha \}
		\Im\{\omega \circ \beta \}
		-
		\Phi^{(3)}_{\times,\epsilon}
	}_{L^\infty(\mathcal{I}_D)}
	&
	\leq
	\frac{\epsilon}{4},
	&
	\quad
	\norm{
		\Im\{\alpha \}
		\Re\{\omega \circ \beta \}
		-
		\Phi^{(4)}_{\times,\epsilon}
	}_{L^\infty(\mathcal{I}_D)}
	&
	\leq
	\frac{\epsilon}{4}.
\end{aligned}
\end{equation}
An inspection of the proof of Lemma \ref{lmm:general_result_oscillatory_texture}
reveals that to achieve the accuracies prescribed in \eqref{eq.accuracy_NNs_prod_comp}
with depths as in Lemma \ref{lmm:general_result_oscillatory_texture}
one requires the individual ReLU-NNs to satisfy
\begin{equation}
\begin{aligned}
	\norm{
		\alpha
		-
		\left(
			\left(\Phi_{\alpha,\epsilon}\right)_1
			+
			\imath
			\left(\Phi_{\alpha,\epsilon}\right)_2
		\right)
	}_{L^\infty(\mathcal{I}_D)}
	&
	\leq
	\frac{
		\epsilon
	}{
		24
		\norm{
			\omega
		}_{L^\infty(\mathcal{I}_\beta)}
	}
	, \\
	\norm{
		\beta
		-
		\Phi_{\beta,\epsilon} 
	}_{L^\infty(\mathcal{I}_D)}
	&
	\leq
	\frac{
		\epsilon
	}{
		24
		\norm{
			\alpha
		}_{L^\infty(\mathcal{I}_D)}
		\norm{
			\omega'
		}_{L^\infty(\mathcal{I}_\beta)}
	},
\end{aligned}
\end{equation}
and
\begin{equation}
	\norm{
		\omega
		-
		\left(
			\left(\Phi_{\omega,\epsilon}\right)_1
			+
			\imath
			\left(\Phi_{\omega,\epsilon}\right)_2
		\right)
	}_{L^\infty(\mathcal{I}_\beta)}
	\leq
	\frac{
		\epsilon
	}{
		24
		\norm{
			\alpha
		}_{L^\infty(\mathcal{I}_D)}
	},
\end{equation}
with $\mathcal{I}_\beta$ as in \eqref{eq:interval_beta}.

Let us define the ReLU-NN 
\begin{equation}
	\Lambda_{1,\epsilon}
	(x_1,x_2,x_3)
	\coloneqq
	\begin{pmatrix}
		1 & -1   
	\end{pmatrix}
	\begin{pmatrix}
		\mu_{U_{\alpha,\omega},\epsilon}\left(\left(\Phi_{\alpha,\epsilon}(x_1)\right)_1,x_2\right) \\
		\mu_{U_{\alpha,\omega},\epsilon}\left(\left(\Phi_{\alpha,\epsilon}(x_1)\right)_2,x_3\right) 
	\end{pmatrix}
\end{equation}
and
\begin{equation}
	\Lambda_{2,\epsilon}
	(x_1,x_2,x_3)
	\coloneqq
	\begin{pmatrix}
		1 & 1
	\end{pmatrix}
	\begin{pmatrix}
		\mu_{U_{\alpha,\omega},\epsilon}
		\left(
			x_3,
			\left(
			\left(
			\Phi_{\omega,\epsilon}
			\circ
			\Phi_{\beta,\epsilon}
			\right)
			(x_1)
			\right)_2
		\right) \\
		\mu_{U_{\alpha,\omega},\epsilon}
		\left(
			x_2,
			\left(
			\left(
			\Phi_{\omega,\epsilon}
			\circ
			\Phi_{\beta,\epsilon}
			\right)
			(x_1)
			\right)_1
		\right)
	\end{pmatrix}.
\end{equation}
Next, we set
\begin{equation}
\begin{aligned}
	\Theta_{1,\epsilon}(x)
	&
	\coloneqq
	\begin{pmatrix}
		x \\
		\left(
			\Phi_{\omega,\epsilon}
			\circ
			\Phi_{\beta,\epsilon}
			(x)
		\right)_1 \\
		\left(
			\Phi_{\omega,\epsilon}
			\circ
			\Phi_{\beta,\epsilon}
			(x)
		\right)_2
	\end{pmatrix},
	&
	\quad
	\Theta_{2,\epsilon}(x_1,x_2,x_3)
	&
	\coloneqq
	\begin{pmatrix}
		\Lambda_{1,\epsilon} (x_1,x_2,x_3) \\
		x_1
	\end{pmatrix},
	\\
	\Theta_{3,\epsilon}(x_1,x_2)
	&
	\coloneqq
	\begin{pmatrix}
		x_1 \\
		x_2 \\
		\left(\Phi_{\alpha,\epsilon}(x_2)\right)_1 \\
		\left(\Phi_{\alpha,\epsilon}(x_2)\right)_2
	\end{pmatrix},
	&
	\quad
	\Theta_{4,\epsilon}(x_1,x_2,x_3)
	&
	\coloneqq
	\begin{pmatrix}
		x _1\\
		\Lambda_{2,\epsilon}
		(x_2,x_3,x_4)
	\end{pmatrix}.
\end{aligned}
\end{equation}
We set
\begin{equation}
	\Phi_{\times,\epsilon}
	(x)
	=
	\left(
		\Theta_{4,\epsilon}
		\circ
		\Theta_{3,\epsilon}
		\circ
		\Theta_{2,\epsilon}
		\circ
		\Theta_{1,\epsilon}
	\right)
	(x),
	\quad
	x \in \mathcal{I}_D.
\end{equation}
We have
\begin{equation}
\begin{aligned}
	\mathcal{M}
	\left(
		\Theta_{1,\epsilon}
	\right)
	&
	=
	2+\max\{M_\omega,M_\beta\},
	&
	\quad
	\mathcal{M}
	\left(
		\Theta_{2,\epsilon}
	\right)
	&
	=
	2
	+
	\max
	\left\{
		10
		,
		M_\alpha
	\right\},
	\\
	\mathcal{M}
	\left(
		\Theta_{3,\epsilon}
	\right)
	&
	=
	4+M_\alpha,
	&
	\quad
	\mathcal{M}
	\left(
		\Theta_{4,\epsilon}
	\right)
	&
	=
	2
	+
	\max
	\left\{
		10
		,
		M_\omega
		,
		M_\beta
	\right\}.
\end{aligned}
\end{equation}
and
\begin{equation}
\begin{aligned}
	\mathcal{L}
	\left(
		\Theta_{1,\epsilon}
	\right)
	=
	&
	\max
	\left\{
		L_\omega
		\left(
		\frac{
			\epsilon
		}{
			24
			\norm{
				\alpha
			}_{L^\infty(\mathcal{I}_D)}
		}
		,
		2\norm{\beta}_{L^\infty(\mathcal{I}_D)}
		\right)
		+
		L_\beta
		\left(
		\frac{
			\epsilon
		}{
			24
			\norm{
				\alpha
			}_{L^\infty(\mathcal{I}_D)}
			\norm{
				\omega'
			}_{L^\infty(\mathcal{I}_\beta)}
		}
		\right)
		,
		2
	\right\},
	\\
	\mathcal{L}
	\left(
		\Theta_{2,\epsilon}
	\right)
	\leq
	&
	C
	\left(
	\log
	\left(
		\lceil U_{\alpha,\omega} \rceil
	\right)
	+
	\log
	\left(
		\frac{12}{\epsilon}
	\right)
	\right)
	+
	\max
	\left\{
		L_\alpha
		\left(
		\frac{
			\epsilon
		}{
			24
			\norm{
				\alpha
			}_{L^\infty(\mathcal{I}_D)}
			\norm{
				\omega'
			}_{L^\infty(\mathcal{I}_\beta)}
		}
		\right)
		,
		2
	\right\},
	\\
	\mathcal{L}
	\left(
		\Theta_{3,\epsilon}
	\right)
	=
	&
	\max
	\left\{
		L_\alpha
		\left(
		\frac{
			\epsilon
		}{
			24
			\norm{
				\alpha
			}_{L^\infty(\mathcal{I}_D)}
			\norm{
				\omega'
			}_{L^\infty(\mathcal{I}_\beta)}
		}
		\right)
		,
		2
	\right\},
	\quad
	\text{and},
	\\
	\mathcal{L}
	\left(
		\Theta_{4,\epsilon}
	\right)
	\leq
	&
	C
	\left(
	\log
	\left(
		\lceil U_{\alpha,\omega} \rceil
	\right)
	+
	\log
	\left(
		\frac{12}{\epsilon}
	\right)
	\right)
	\\
	&
	+
	\max
	\left\{
		L_\omega
		\left(
		\frac{
			\epsilon
		}{
			24
			\norm{
				\alpha
			}_{L^\infty(\mathcal{I}_D)}
		}
		,
		2\norm{\beta}_{L^\infty(\mathcal{I}_D)}
		\right)
		+
		L_\beta
		\left(
		\frac{
			\epsilon
		}{
			24
			\norm{
				\alpha
			}_{L^\infty(\mathcal{I}_D)}
			\norm{
				\omega'
			}_{L^\infty(\mathcal{I}_\beta)}
		}
		\right)
		,
		2
	\right\}.
\end{aligned}
\end{equation}
Thus $\Phi_{\times,\epsilon} \in \mathcal{N\!N}_{L,M,1,2}$ with
\begin{equation}
	L
	\leq
	\sum_{i=1}^{4}
	\mathcal{L}
	\left(
		\Theta_{i,\epsilon}
	\right)
	\quad
	\text{and}
	\quad
	M
	=
	\max_{i=1,\dots,4}
	\left\{
		\mathcal{M}
		\left(
			\Theta_{i,\epsilon}
		\right)
	\right\}.
\end{equation}
Finally, the weights of $\Phi_{f,\epsilon}$ are those of $\Phi_{\alpha,\epsilon}$,
$\Phi_{\beta,\epsilon} $, $\Phi_{\omega,\epsilon}$, 
and of $\mu_{U_{\alpha,\omega},\epsilon}$,
thus yielding \eqref{eq:bound_mult_comp_complex_weights}.
\end{proof}

\subsection{Trading off bound of weights for depth}
\label{sec:BdDepth}
\begin{proposition}[{\cite[Corollary A.2]{elbrachter2021deep}}]
\label{prop:trading_weights}
Let $d,d' \in \mathbb{N}$, $a \in \mathbb{R}_+$, $A \in [-a,a]^{d'\times d}$,
and $b \in [-a,a]^{d'}$. 
There exist ReLU-NNs $\Phi_{A,b} \in \mathcal{N\!N}_{L,M,d,d'}$
satisfying $\Phi_{A,b}(x) = Ax+b$, with 
\begin{equation}
	L
	\leq
	\left \lfloor \log(\snorm{a}) \right \rfloor+5,
	\quad
	M \leq \max\{d,3d'\},
	\quad
	\text{and}
	\quad
	\mathcal{B}(\Phi_{A,b}) \leq 1.
\end{equation}
\end{proposition}

\section{Proof of Propositions \ref{prop:approx_FockInt_NN} and \ref{prop:approx_FockInt_NN_der}}
\label{sec:proof_FockInt}
\begin{proof}[Proof of Propositions \ref{prop:approx_FockInt_NN}]
We divide the proof in five steps.
We construct ReLU-NNs
emulating $\Psi$ in three different regions
of the interval $\mathcal{I}_\kappa$
by leveraging on the asymptotic expansion
described in Theorem \ref{thm:expansion_V}.
These three constructions are described in steps
{\sf \encircle{1}} through {\sf \encircle{3}}. 
Next, in step {\sf \encircle{4}}, using a suitable
partition of unity defined over the interval $\mathcal{I}_\kappa$,
we put together the aforementioned localized ReLU-NNs to construct 
a global approximation. 
Finally, in step {\sf \encircle{5}} we provide
wavenumber-robust bounds for the depth, width, and weights of the
constructed ReLU-NNs.\\

{\sf \encircle{1}}  {\sf ReLU-NN Emulation of $\Psi(\tau)$ for $\tau \rightarrow +\infty$}.
According to Theorem \ref{thm:expansion_V} for each $n\in \mathbb{N}$
there exists $C_n>0$ such that
\begin{align}
	\snorm{
	\Psi(\tau)
	-
	\sum_{j=0}^{n}a_j \tau^{1-3j}} 
	\leq  C_n\tau^{1-3(n+1)}
	\quad
	\text{as}\quad \tau\rightarrow\infty.
\end{align}
For each $n\in \mathbb{N}$ and $\epsilon>0$ 
we set
\begin{align}\label{eq:upper_D}
	D^+_{{n,\epsilon}}
	\coloneqq
	2
	\max
	\left\{
	 \left(
	 	2
	 	\frac{C_n}{\epsilon}
	\right)^{\frac{1}{3(n+1)-1}}
	,
	1
	\right\}.
\end{align}
For each $n\in\mathbb{N}$, $\epsilon>0$,
and $D^+_{{n,\epsilon}}$ as in 
\eqref{eq:upper_D} it holds
\begin{align}\label{eq:bound_Psi_a}
	\norm{
		\Psi(\tau)
		-
		\sum_{j=0}^{n}a_j \tau^{1-3j}
	}_{
		L^\infty
		\left(
		\left(
			\frac{1}{2}D^+_{{n,\epsilon}},\kappa^{\frac{1}{3}}A
		\right)
		\right)
	} 
	\leq
	\frac{\epsilon}{2}.
\end{align}
Lemma \ref{lmm:approx_NN_inverse} 
asserts the existence of $C>0$ and ReLU-NNs
$\upphi_{D_\kappa,\zeta}\in \mathcal{N\!N}_{L,13,1,1}$ 
with
\begin{equation}
\begin{aligned}
	L
	\leq
	C
	\left(\log\left(\frac{1}{\zeta}\right)\right)^2
	\left(
		\left(\log\left(\kappa^{\frac{1}{3}}A\right)\right)^2
		+
		\left(
			\log
			\left(
				\frac{1}{\zeta}
			\right)
		\right)^2
	\right)
	\quad
	\text{as }
	\zeta \rightarrow 0
\end{aligned}
\end{equation}
such that
\begin{align}
	\norm{
		\frac{1}{\tau} 
		- 
		\upphi_{D_\kappa,\zeta}(\tau)
	}_{L^\infty\left(\left(1,\kappa^{\frac{1}{3}}A\right)\right)}
	\leq 
	\zeta.
\end{align}

According to Lemma \ref{lmm:inverse_pol_NN} there exist $C>0$ 
and for each $j\in \{1,\dots,n\}$ ReLU-NNs
$\uppsi_{3j-1,\epsilon} \in \mathcal{N\!N}_{L_j,13,1,1}$
with
\begin{equation}
\begin{aligned}
	L_j
	\leq
	&
	C
	\left(
	\log
	\left(
		\frac{1}{\epsilon}
	\right)
	+
	\log(6j-2)
	\right)^2
	\left(
		\left(\log\left(\kappa^{\frac{1}{3}}A\right)\right)^2
		+
		\left(
			\log
			\left(
				\frac{2}{\epsilon}
			\right)
			+
			\log(6j-2)
		\right)^2
	\right)
	\\
	&
	+
	C
	(3j-1)
	\left(
	\log
	\left(
		\frac{1}{\epsilon}
	\right)
	+
	\log(6j-2)
	\right),
\end{aligned}
\end{equation}
and satisfying 
\begin{equation}
	\norm{
		\uppsi_{3j-1,\epsilon}(\tau)
		-
		\frac{1}{\tau^{3j-1}}
	}_{L^\infty
	\left(\left(
		\frac{1}{2}D^+_{{n,\epsilon}},\kappa^{\frac{1}{3}}A
	\right)\right)}
	\leq
	\norm{
		\uppsi_{3j-1,\epsilon}(\tau)
		-
		\frac{1}{\tau^{3j-1}}
	}_{L^\infty
	\left(\left(
		1,\kappa^{\frac{1}{3}}A
	\right)\right)}
	\leq 
	\epsilon.
\end{equation}
We shortly recall the construction of the ReLU-NN $\uppsi_{3j-1,\epsilon}$ from Lemma \ref{lmm:inverse_pol_NN}.
Let $\mu_{D_{j,\zeta},\zeta} \in \mathcal{N\!N}_{L,5,2,1}$ with 
$L\leq C\left(\log(D_{j,\zeta})+ \log\left(\frac{1}{\zeta}\right)\right)$
as $\zeta\rightarrow 0$ be the multiplication ReLU-NN from
Proposition \ref{prop:mult_network} satisfying 
\begin{equation}
	\norm{
		\mu_{D_{j,\zeta} ,\zeta}(x,y)-xy
	}_{L^\infty
	\left(\left(-D_{j,\zeta},D_{j,\zeta} \right)^2\right)} 
	\leq 
	\zeta,
\end{equation}
with $D_{j,\zeta} \coloneqq 1+ (2j-1)(1+\zeta)^{j-1}\zeta$.
Set
\begin{align}\label{eq:net_psi_mius_1}
	\Theta_{j,\zeta}(x_1,x_2,x_3,x_4)
	\coloneqq
	A_j
	\begin{pmatrix}
	x_1 \\
	x_2 \\
	x_3 \\
	\mu_{D_{j,\zeta} ,\zeta}(x_3,x_4)
	\end{pmatrix},
	\quad
	j =3,\dots,3n-1,
\end{align}
together with 
\begin{equation}
	\Theta_{1,\zeta}(x)
	\coloneqq
	\begin{pmatrix}
	\Re\{a_0\} x \\
	\Im\{a_0\} x \\
	\upphi_{D_\kappa,\zeta} (x)
	\end{pmatrix}
	\quad
	\text{and}
	\quad
	\Theta_{2,\zeta}(x_1,x_2,x_3)
	\coloneqq
	A_j
	\begin{pmatrix}
	x_1 \\
	x_2 \\
	x_3 \\
	\mu_{D_{1,\zeta} ,\zeta}(x_3,x_3)
	\end{pmatrix},
\end{equation}
with
\begin{equation}
	A_j
	\coloneqq
	\begin{pmatrix}
		1 & 0 & 0 & \Re\{\widetilde{a}_j\}\\
		0 & 1 & 0 & \Im\{\widetilde{a}_j\}\\
		0 & 0 & 1 & 0\\
		0 & 0 & 0 & 1
	\end{pmatrix}
	,
	\quad
	\widetilde{a}_j
	\coloneqq
	\left\{
	\begin{array}{cl}
		a_{\frac{j+1}{3}}, & \text{mod}_3(j+1) = 0, \\
		0, & \text{otherwise},
	\end{array}
	\right.
	\quad
	j=1,\dots,3n-1.
\end{equation}
Define
\begin{align}\label{eq:net_psi_mius_2}
	\Psi^+_{n,\epsilon}(\tau)
	\coloneqq
	\begin{pmatrix}
	1 & 0 & 0 \\
	0 & 1 & 0 \\
	\end{pmatrix}
	\left(
		\Theta_{3n-1,\zeta} \circ \Theta_{n-1,\zeta} \circ\cdots\circ\Theta_{2,\zeta}\circ\Theta_{1,\zeta}
	\right)(\tau),
\end{align}
with $\epsilon$ depending on $\zeta$ in a form to be specified ahead.

One can readily see that
\begin{equation}
	\left(
		\Psi^+_{n,\zeta}(\tau)
	\right)_1
	+
	\imath
	\left(
		\Psi^+_{n,\epsilon}(\tau)
	\right)_2
	=
	a_0
	\tau
	+
	\sum_{j=1}^{3n-1}
	\widetilde{a}_j
	\uppsi_{j,D_\kappa,\epsilon}(\tau)
	=
	a_0
	\tau
	+
	\sum_{j=1}^{n}
	a_j
	\uppsi_{3j-1,D_\kappa,\zeta}(\tau).
\end{equation}
Next
\begin{equation}
\begin{aligned}
	\norm{
		\left(
			\Psi^+_{n,\epsilon}(\tau)
		\right)_1
		+
		\imath
		\left(
			\Psi^+_{n,\epsilon}(\tau)
		\right)_2
		-
		\sum_{j=0}^{n}a_j \tau^{1-3j}
	}_{
		L^\infty
		\left(
		\left(
			\frac{1}{2}D^+_{{n,\epsilon}},\kappa^{\frac{1}{3}}A
		\right)
		\right)
	}
	&
	\\
	&
	\hspace{-4.5cm}
	\leq 
	\sum_{j=1}^{n}
	\snorm{
		a_j
	}
	\norm{
		\frac{
			1
		}{
			\tau^{3j-1}
		}
		-
		\uppsi_{3j-1,D_\kappa,\zeta}(\tau)
	}_{
		L^\infty
		\left(
		\left(
			\frac{1}{2}D^+_{{n,\zeta}},\kappa^{\frac{1}{3}}A
		\right)
		\right)
	}
	\\
	&
	\hspace{-4.5cm}
	\leq 
	\sum_{j=1}^{n}
	\snorm{
		a_j
	}
	(6j-1)
	(1+\zeta)^{3j-1}
	\zeta
	\\
	&
	\hspace{-4.5cm}
	\leq 
	(6n-1)
	(1+\zeta)^{3n-1}
	\zeta
	U_n,
\end{aligned}
\end{equation}
with $U_n = \sum_{j=1}^{n}\snorm{a_j}$.

Let us set $\epsilon = 2(6n-1) (1+\zeta)^{3n-1} \zeta U_n$. Therefore,
recalling \eqref{eq:bound_Psi_a} and using the triangle inequality we get
\begin{equation}
\begin{aligned}
	\norm{
		\Psi
		-
		\left(
			\Psi^+_{n,\epsilon}(\tau)
		\right)_1
		-
		\imath
		\left(
			\Psi^+_{n,\epsilon}(\tau)
		\right)_2
	}_{L^\infty
	\left(\left(
		\frac{1}{2}D^+_{{n,\epsilon}}
		,
		\kappa^{\frac{1}{3}}
		A
	\right)\right)} 
	\leq
	\epsilon.
\end{aligned}
\end{equation}
An inspection of \eqref{eq:net_psi_mius_1} and
\eqref{eq:net_psi_mius_2} reveals that
the width of $\Psi^+_{n,\epsilon}$ is
\begin{equation}
\begin{aligned}
	\mathcal{M}
	\left(
		\Psi^+_{n,\epsilon}
	\right)
	=
	\max_{j \in \{1,\dots,3n-1\}}
	\mathcal{M}
	\left(
		 \Theta_{j,\zeta} 
	\right)
	=
	\mathcal{M}
	\left(
		 \Theta_{1,\zeta} 
	\right)
	=
	4
	+
	\mathcal{M}
	\left(
		\upphi_{D_\kappa,\zeta} 
	\right)
	=
	17.
\end{aligned}
\end{equation}
The depth of $\Psi^+_{n,\zeta}(\tau)$ is bounded according to
{\small
\begin{equation}
\begin{aligned}
	\mathcal{L}
	\left(
		\Psi^+_{n,\zeta}
	\right)
	\leq
	&
	\sum_{j=1}^{3n-1}
	\mathcal{L}
	\left(
		\Theta_{j,\zeta}
	\right)
	\\
	\leq
	&
	\mathcal{L}
	\left(
		\upphi_{D_\kappa,\zeta}
	\right)
	+
	\sum_{j=2}^{3n-1}
	\mathcal{L}
	\left(
		\mu_{D_{j,\zeta} ,\zeta}
	\right)
	\\
	\leq
	&
	C
	\left(\log\left(\frac{1}{\zeta}\right)\right)^2
	\left(
		\left(\log\left(\kappa^{\frac{1}{3}}A\right)\right)^2
		+
		\left(
			\log
			\left(
				\frac{2}{\zeta}
			\right)
		\right)^2
	\right)
	\\
	&
	+
	C
	(3n-1)
 	\left(
		\log\left(D_{3n-1,\zeta}\right)
		+ 
		\log\left(\frac{1}{\zeta}\right)
	\right)
	\\
	\leq
	&
	C
	\left(
	\log
	\left(
		\frac{U_n}{\epsilon}
	\right)
	+
	\log(12n-2)
	\right)^2
	\left(
		\left(\log\left(\kappa^{\frac{1}{3}}A\right)\right)^2
		+
		\left(
			\log
			\left(
				\frac{2}{\epsilon}
				U_n
			\right)
			+
			\log(12n-2)
		\right)^2
	\right)
	\\
	&
	+
	C
	(3n-1)
	\left(
	\log
	\left(
		\frac{U_n}{\epsilon}
	\right)
	+
	\log(12n-2)
	\right),
	\quad
	\text{as }
	\epsilon \rightarrow 0,
\end{aligned}
\end{equation}
}%
where we have used that
$\epsilon = 2(6n-1) (1+\zeta)^{3n-1} \zeta U_n \geq 2(6n-1)  \zeta U_n$,
the fact that $D_{3n-1,\zeta} \leq D_{3n,\zeta} = 1+ \frac{\epsilon}{2 U_n}$,
and
$
	\log
	\left(
		\frac{1}{\zeta}
	\right)
	\leq
	\log
	\left(
		\frac{U_n}{\epsilon}
	\right)
	+
	\log(12n-2)
	+
	\epsilon
	\frac{3n-1}{(12n-2)U_n}.
$
The weights of this ReLU-NN
are bounded in absolute value by a constant independent of $\kappa$, however depending on $n\in\IN$.
Indeed, the weights of $\Psi^+_{n,\zeta}$ are those of $\upphi_{D_\kappa,\zeta}$ and $\mu_{D_{j,\zeta} ,\zeta}$. 
%

{\sf \encircle{2}} {\sf Approximation of $\Psi(\tau)$ for $\tau \rightarrow -\infty$}.
According to Theorem \ref{thm:expansion_V} 
\begin{align}
	\Psi(\tau)
	=
	c_0 \exp(-\imath \tau^3/3-\imath \tau \alpha_1) 
	(1+ \mathcal{O}(\exp(-\beta \snorm{\tau}))
	\quad
	\text{as } \tau\rightarrow-\infty.
\end{align}
There exists $C>0$ independent of $\kappa$ such that
\begin{align}\label{eq:asymp_Psi}
	\snorm{\Psi(\tau) - c_0 \exp(-\imath \tau^3/3-\imath \tau \alpha_1) } 
	\leq 
	C
	\exp\left(\left(\frac{\nu_1\sqrt{3}}{2}-\beta\right)\snorm{\tau}\right), 
	\quad
	\tau \in \mathbb{R}_{-}.
\end{align}
Recall that $\frac{\nu_1\sqrt{3}}{2}-\beta<0$, $\nu_1 <0$,
$\alpha_1=\exp(-2\pi/3 \imath)\nu_1$, 
and that $\nu_1\in \mathbb{R}_{-}$ is the right-most root of $\normalfont\text{Ai}$,
as explained in Theorem \ref{thm:expansion_V}.

Let us set
\begin{align}
	D^{-}_{\epsilon}
	\coloneqq
	2
	\max
	\left\{
	\snorm{
		\frac{\nu_1\sqrt{3}}{2}-\beta
	}^{-1} 
	\log 
	\left(
		2
		\frac{C}{\epsilon}
	\right)
	,
	1
	\right\},
\end{align}
with $C>0$ as in \eqref{eq:asymp_Psi}.
Then, we have
\begin{align}\label{eq:approx_complex_exp_Psi}
	\norm{
		\Psi(\tau) 
		- 
		c_0 
		\exp(-\imath \tau^3/3-\imath \tau \alpha_1)
	}_{
		L^\infty
		\left(
			\left(
				-\kappa^{\frac{1}{3}}A
				,
				-\frac{1}{2} D^-_\epsilon
			\right)
		\right)
	}
	\leq
	\frac{\epsilon}{2}.
\end{align}
Let us set $ \phi(\tau) \coloneqq c_0  \exp(-\imath \tau^3/3-\imath \tau \alpha_1)	$.
Observe that 
\begin{equation}\label{eq:real_img_app}
\begin{aligned}
	\Re\{\phi(\tau)\}
	&
	=
	\left(
	\Re\{c_0\}
	\cos
	\left(
		\frac{\tau \nu_1}{2} 
		- 
		\frac{\tau^3}{3}
	\right)
	-
	\Im\{c_0\}
	\sin
	\left(
		\frac{\tau \nu_1}{2} 
		- 
		\frac{\tau^3}{3}
	\right)
	\right)
	\exp
	\left(
		\frac{\nu_1\sqrt{3}}{2}\snorm{\tau}
	\right),
	\quad \text{and} \\
	\Im\{\phi(\tau)\}
	&
	=
	\left(
	\Re\{c_0\}
	\sin
	\left(
		\frac{\tau \nu_1}{2} 
		- 
		\frac{\tau^3}{3}
	\right)
	+
	\Im\{c_0\}
	\cos
	\left(
		\frac{\tau \nu_1}{2} 
		- 
		\frac{\tau^3}{3}
	\right)
	\right)
	\exp
	\left(
		\frac{\nu_1\sqrt{3}}{2}\snorm{\tau}
	\right).
\end{aligned}
\end{equation}
Define
\begin{equation}\label{eq:def_phi}
\begin{aligned}
	\phi_1(\tau)
	=
	&
	\cos
	\left(
		\frac{\tau \nu_1}{2} 
		- 
		\frac{\tau^3}{3}
	\right)
	\exp
	\left(
		\frac{\nu_1\sqrt{3}}{2}\snorm{\tau}
	\right),
	\quad
	\text{and}\\
	\phi_2(\tau)
	=
	&
	\sin
	\left(
		\frac{\tau \nu_1}{2} 
		- 
		\frac{\tau^3}{3}
	\right)
	\exp
	\left(
		\frac{\nu_1\sqrt{3}}{2}\snorm{\tau}
	\right).
\end{aligned}
\end{equation}
Hence, 
\begin{align}
	\norm{
		\Re\{\Psi(\tau)\}
		- 
		\left(
			\Re\{c_0\}
			\phi_1(\tau)
			-
			\Im\{c_0\}
			\phi_2(\tau)
		\right)
	}_{L^\infty\left(\left(-\kappa^{\frac{1}{3}}A,-\frac{1}{2}D^-_\epsilon\right)\right)}
	\leq
	\frac{\epsilon}{2},
\end{align}
and
\begin{align}
	\norm{
		\Im\{\Psi(\tau)\}
		- 
		\left(
			\Re\{c_0\}
			\phi_2(\tau)
			+
			\Im\{c_0\}
			\phi_1(\tau)
		\right)
	}_{L^\infty\left(\left(-\kappa^{\frac{1}{3}}A,-\frac{1}{2}D^-_\epsilon\right)\right)}
	\leq
	\frac{\epsilon}{2}.
\end{align}

For the ReLU-NN emulation of $\phi_1(\tau)$ and $\phi_2(\tau)$
we use Theorem \ref{lmm:general_result_oscillatory_texture}.
We proceed to verify the assumptions of said result item-by-item.
\begin{itemize}
	\item[(i)]
	According to Proposition \ref{prop:aprox_sc},
	there exist $C>0$ and ReLU-NNs
	$$
		\Phi^{\cos}_{B,\epsilon}, \Phi^{\sin}_{B,\epsilon} 
		\in 
		\mathcal{N\!N}_{L_{\sin,\cos}(\epsilon,B),9,1,1}
	$$
	such that for any $B>0$
	\begin{equation}
	L_{\sin,\cos}(\epsilon,B)
	\leq 
	C
	\left(
	\left(
		\log
		\left(
			\frac{1}{\epsilon}
		\right)
	\right)^2
	+
	\log
	\left(
		\lceil B \rceil
	\right)
	\right),
	\quad
	\text{as }
	\epsilon \rightarrow 0,
	\end{equation}
	such that 
	\begin{equation}
		\norm{
			\cos(x)
			-
			\Phi^{\cos}_{B,\epsilon}(x)
		}_{L^\infty\left(\left(-B,B\right)\right)} 
		\leq 
		\epsilon
		\;\;\;
		\text{and}
		\;\;\;
		\norm{
			\sin(x)
			-
			\Phi^{\sin}_{B,\epsilon}(x)
		}_{L^\infty\left(\left(-B,B\right)\right)} 
		\leq 
		\epsilon.
	\end{equation}
	\item[(ii)]
	According to Proposition \ref{prop:approximation_polynomials},
	there exist ReLU-NNs
	$\vartheta_{\epsilon} \in \mathcal{N\!N}_{L_\vartheta(\epsilon),9,1,1}$ 
	with
	\begin{equation}
	L_\vartheta(\epsilon)
	\leq 
	C
	\left(
		\log
		\left(
			\frac{1}{\epsilon}
		\right)
		+
		\log
		\left(
		\kappa^{\frac{1}{3}} A
		\right)
	\right)
	\quad
	\text{as }
	\epsilon \rightarrow 0,
	\end{equation}	
	such that
	\begin{equation}
		\norm{
			\vartheta_{\epsilon} (x)
			-
			\left(
			\frac{x \nu_1}{2} - \frac{x^3}{3}
			\right)
		}_{L^\infty\left(\left(-\kappa^{\frac{1}{3}}A,-\frac{1}{2}D^-_\epsilon\right)\right)} \leq \epsilon.
	\end{equation}
	\item[(iii)]
	According to Proposition \ref{prop:approx_exp_NN}
	there exist ReLU-NN
	$\Phi^{\exp}_{a,D,\epsilon} \in \mathcal{N\!N}_{L_{\exp}(\epsilon),9,1,1}$ 
	with $a = \frac{\nu_1\sqrt{3}}{2}$, $D = \kappa^{\frac{1}{3}}A$, and
	\begin{align}
	L_{\exp}(\epsilon)
	\leq
	C
	\left(
		\left(
		\log
		\left(
			\frac{\snorm{\nu_1}\sqrt{3}}{2}
			\kappa^{\frac{1}{3}} A
		\right)
		\right)^2
		+
		\left(\log\left(\frac{1}{\epsilon}\right)\right)^2
	\right),
	\quad
	\text{as }
	\epsilon \rightarrow 0,
	\end{align}
	such that
	\begin{equation}
		\norm{
			\Phi^{\exp}_{a,D,\epsilon}  (x)
			-
			\exp\left(\frac{\nu_1\sqrt{3}}{2}\snorm{x}\right)
		}_{L^\infty\left(\left(-\kappa^{\frac{1}{3}}A,-\frac{1}{2}D^-_\epsilon\right)\right)} \leq \epsilon.
	\end{equation}
\end{itemize}
Consequently, Lemma \ref{lmm:general_result_oscillatory_texture} asserts 
the existence of a constant $C>0$ and ReLU-NNs 
$\phi_{1,\epsilon},\phi_{2,\epsilon} \in \mathcal{N\!N}_{L,14,1,1}$ 
with
\begin{equation}
\begin{aligned}
	L
	\leq
	&
	C
	\log
	\left(
		\frac{1}{\epsilon}
	\right)
	+
	L_{\sin,\cos}
	\left(
		\frac{
			\epsilon
		}{
			6
		}
		,
		\snorm{\nu_1}
		\kappa^{\frac{1}{3}} A
		+
		\frac{2}{3}
		\kappa A^3
	\right)
	+
	L_{\exp}
	\left(
	\frac{
		\epsilon
	}{
		6
	}
	\right)
	+
	L_\vartheta
	\left(
		\frac{
		\epsilon
		}{
		6
		}
	\right)
	\\
	\leq
	&
	C
	\left(
	\log
	\left(
		\frac{1}{\epsilon}
	\right)
	+
	\left(
		\log
		\left(
			\frac{6}{\epsilon}
		\right)
	\right)^2
	+
	\log
	\left(
		\snorm{\nu_1}
		\kappa^{\frac{1}{3}} A
		+
		\frac{2}{3}
		\kappa A^3
	\right)
	\right)
	\\
	&
	+
	C
	\left(
		\log
		\left(
			\frac{6}{\epsilon}
		\right)
		+
		\log
		\left(
		\kappa^{\frac{1}{3}} A
		\right)
	\right)
	+
	C
	\left(
		\left(
		\log
		\left(
			\frac{\snorm{\nu_1}\sqrt{3}}{2}
			\kappa^{\frac{1}{3}} A
		\right)
		\right)^2
		+
		\left(\log\left(\frac{6}{\epsilon}\right)\right)^2
	\right)
\end{aligned}
\end{equation}
such that for $i=1,2$
\begin{equation}\label{eq:error_bound_Phi}
	\norm{
		\phi_{i}
		-
		\phi_{i,\epsilon}
	}_{L^\infty\left(\left(-\kappa^{\frac{1}{3}}A,-\frac{1}{2}D^-_\epsilon\right)\right)}
	\leq
	\epsilon.
\end{equation}
Let us set
\begin{equation}
	\Theta_0(x)
	=
	\begin{pmatrix}
	0 \\
	0 \\
	x
	\end{pmatrix}
	\quad
	\text{and}
	\quad
	\Theta_{i,\epsilon}(x_1,x_2,x_3)
	=
	A_i
	\begin{pmatrix}
		x_1 \\
		x_2 \\
		x_3 \\
		\phi_{i,\epsilon}(x_3)
	\end{pmatrix}
	\quad
	\text{for }\\
	i=1,2,
\end{equation}
with
\begin{equation}
	A_1
	\coloneqq
	\begin{pmatrix}
		1 & 0 & 0 & \Re\{c_0\} \\
		0 & 1 & 0 & \Im\{c_0\} \\
		0 & 0 & 1 & 0
	\end{pmatrix}
	\in 
	\mathbb{R}^{3\times 4}
	\quad
	\text{and}
	\quad
	A_2
	\coloneqq
	\begin{pmatrix}
		1 & 0 & 0 & -\Im\{c_0\} \\
		0 & 1 & 0 & \Re\{c_0\} \\
		0 & 0 & 1 & 0
	\end{pmatrix}
	\in 
	\mathbb{R}^{3\times 4}
\end{equation}	
Define the ReLU-NN 
\begin{equation}
	\Psi^-_{\epsilon}(\tau)
	=
	\begin{pmatrix}
	1 & 0 & 0  \\
	0 & 1 & 0 
	\end{pmatrix}
	\left(
		\Theta_{2,\epsilon}
		\circ
		\Theta_{1,\epsilon}
		\circ
		\Theta_0
	\right)(\tau).
\end{equation}
Observe that
\begin{equation}
\begin{aligned}
	\left(
		\Psi^-_{\epsilon}(\tau)
	\right)_1
	+
	\imath
	\left(
		\Psi^-_{\epsilon}(\tau)
	\right)_2
	=
	\Re\{c_0\} 
	\phi_{1,\epsilon}(\tau)
	-
	\Im\{c_0\} 
	\phi_{2,\epsilon}(\tau)
	+
	\imath
	\left(
		\Re\{c_0\} 
		\phi_{2,\epsilon}(\tau)
		+
		\Im\{c_0\} 
		\phi_{1,\epsilon}(\tau)
	\right).
\end{aligned}
\end{equation}
Recalling \eqref{eq:approx_complex_exp_Psi}
and \eqref{eq:error_bound_Phi}
\begin{equation}
\begin{aligned}
	\norm{
		\Psi
		-
		\left(
		\left(
			\Psi^-_{\epsilon}(\tau)
		\right)_1
		+
		\imath
		\left(
			\Psi^-_{\epsilon}(\tau)
		\right)_2
		\right)
	}_{L^\infty\left(\left(-\kappa^{\frac{1}{3}}A,-\frac{1}{2}D^-_\epsilon\right)\right)}
	\\
	&
	\hspace{-7.5cm}
	\leq
	\norm{
		\Psi(\tau) 
		- 
		c_0 
		\exp(-\imath \tau^3/3-\imath \tau \alpha_1)
	}_{
		L^\infty
		\left(
			\left(
				-\kappa^{\frac{1}{3}}
				A
				,
				-
				\frac{1}{2}
				D^-_\epsilon
			\right)
		\right)
	}
	\\
	&
	\hspace{-7.2cm}
	+
	\norm{
		c_0 
		\exp(-\imath \tau^3/3-\imath \tau \alpha_1)
		-
		\left(
		\left(
			\Psi^-_{n,\epsilon}(\tau)
		\right)_1
		+
		\imath
		\left(
			\Psi^-_{n,\epsilon}(\tau)
		\right)_2
		\right)
	}_{
		L^\infty
		\left(
			\left(
				-\kappa^{\frac{1}{3}}
				A
				,
				-
				\frac{1}{2}
				D^-_\epsilon
			\right)
		\right)
	}
	\\
	&
	\hspace{-7.5cm}
	\leq
	\norm{
		\Psi(\tau) 
		- 
		c_0 
		\exp(-\imath \tau^3/3-\imath \tau \alpha_1)
	}_{
		L^\infty
		\left(
			\left(
				-\kappa^{\frac{1}{3}}
				A
				,
				-
				\frac{1}{2}
				D^-_\epsilon
			\right)
		\right)
	}
	\\
	&
	\hspace{-7.2cm}
	+
	\left( \snorm{\Re\{c_0\}}+\snorm{\Im\{c_0\}} \right)
	\left(
		\norm{
		\phi_{1}
		-
		\phi_{1,\epsilon}
		}_{L^\infty\left(\left(-\kappa^{\frac{1}{3}}A,-\frac{1}{2}D^-_\epsilon\right)\right)}
	\right.
	\\
	&
	\hspace{-7.2cm}
	+
	\left.
		\norm{
		\phi_{2}
		-
		\phi_{2,\epsilon}
		}_{L^\infty\left(\left(-\kappa^{\frac{1}{3}}A,-\frac{1}{2}D^-_\epsilon\right)\right)}
	\right)
	\\
	&
	\hspace{-7.5cm}
	\leq
	\left(
		\frac{1}{2}
		+
		\snorm{\Re\{c_0\}}+\snorm{\Im\{c_0\}}
	\right)
	\epsilon.
\end{aligned}
\end{equation}
Consequently, there exist ReLU-NNs 
$\Psi^-_{\epsilon} $ of depth
\begin{equation}
\begin{aligned}
	\mathcal{L}
	\left(
		\Psi^-_{\epsilon} 
	\right)
	\leq
	&
	\sum_{i=1}^{2}
	\mathcal{L}
	\left(
		\phi_{i,\epsilon}
	\right),
\end{aligned}
\end{equation}
thus
\begin{equation}
\begin{aligned}
	\mathcal{L}
	\left(
		\Psi^-_{\epsilon} 
	\right)
	\leq
	&
	C
	\left(
	\log
	\left(
		\frac{U_0}{\epsilon}
	\right)
	+
	\left(
		\log
		\left(
			\frac{6}{\epsilon}
			U_0
		\right)
	\right)^2
	+
	\log
	\left(
		\snorm{\nu_1}
		\kappa^{\frac{1}{3}} A
		+
		\frac{2}{3}
		\kappa A^3
	\right)
	\right)
	\\
	&
	+
	C
	\left(
		\log
		\left(
			\frac{6}{\epsilon}
			U_0
		\right)
		+
		\log
		\left(
		\kappa^{\frac{1}{3}} A
		\right)
	\right)
	+
	C
	\left(
		\left(
		\log
		\left(
			\frac{\snorm{\nu_1}\sqrt{3}}{2}
			\kappa^{\frac{1}{3}} A
		\right)
		\right)^2
		+
		\left(\log\left(\frac{6}{\epsilon}U_0\right)\right)^2
	\right)
	\\
	\leq
	&
	C'
	\left(
		\log
		\left(
			\frac{1}{\epsilon}
		\right)
		+
		\left(
		\log
		\left(
			\frac{1}{\epsilon}
		\right)
		\right)^2
		+
		\log
		\left(
			\kappa^{\frac{1}{3}}
			A
		\right)
		+
		\left(
		\log
		\left(
			\kappa^{\frac{1}{3}}
			A
		\right)
		\right)^2
	\right),
	\quad
	\text{as }
	\epsilon \rightarrow 0,
\end{aligned}
\end{equation}
with $U_0 = \frac{1}{2}+\snorm{\Re\{c_0\}}+\snorm{\Im\{c_0\}}$
and such that 
\begin{equation}
\norm{
		\Psi(\tau)
		-
		\left(
		\left(
			\Psi^-_{\epsilon}(\tau)
		\right)_1
		+
		\imath
		\left(
			\Psi^-_{\epsilon}(\tau)
		\right)_2
		\right)
	}_{L^\infty\left(\left(-\kappa^{\frac{1}{3}}A_Z,-\frac{1}{2}D^-_\epsilon\right)\right)}
	\leq
	\epsilon.
\end{equation}
The width of $\Psi^-_{\epsilon}$ is
\begin{equation}
	\mathcal{M}
	\left(
		\Psi^-_{\epsilon}
	\right)
	=
	\max
	\left\{
	\mathcal{M}
	\left(
		\Theta_{0}
	\right)
	,
	\mathcal{M}
	\left(
		\Theta_{1,\epsilon}
	\right)
	,
	\mathcal{M}
	\left(
		\Theta_{2,\epsilon}
	\right)
	\right\}
	=
	20.
\end{equation}

{\sf \encircle{3}}
{\sf Approximation of $\Psi(\tau)$ for $\tau$ bounded.}
Set $$ D_{n,\varepsilon} \coloneqq \max \left\{ D^+_{{n,\epsilon}},D^-_{{\epsilon}}\right\}.$$
Remark \ref{rmk:bound_psi} implies that
$\Psi \in \mathcal{P}_{\boldsymbol{\lambda},D_{n,\epsilon},\mathbb{C}}$
with $\boldsymbol{\lambda}\coloneqq\{\lambda_n\}_{n\in \mathbb{N}_0}$ 
defined as
$\lambda_0\leq C_0\left(1+D_{n,\epsilon}\right)$
and $\lambda_n\leq C_n$ for all $n\in \IN_0$
with $C_n>0$ as in Remark \ref{rmk:bound_psi},
thus independent of $\kappa$.

According to Proposition \ref{prop:approximation_P}
for each $n\in \IN_{\geq 2}$ there exist ReLU-NNs 
$\Psi^\cc_{n,\epsilon} \in \mathcal{N\!N}_{L,15,1,2}$
with 
\begin{equation}
	L
	\leq 
	B_{n,\epsilon}
	\left(
		B_{n,\epsilon}
		\epsilon^{-\frac{2}{n-1}}
		+
		\epsilon^{-\frac{1}{n-1}}
		\log\left(\frac{1}{\epsilon}\right) 
		+
		\epsilon^{-\frac{2}{n-1}}
		\log
		\left(
			D_{n,\epsilon}
		\right)
		+
		\epsilon^{-\frac{1}{n-1}}
		\log
		\left(
			\lambda_n
		\right)
	\right),
\end{equation}
as $\epsilon \rightarrow 0$, such that 
\begin{equation}
	\norm{
		\Psi
		-
		\left(
			(\Psi^\cc_{n,\epsilon})_1
			+
			\imath
			(\Psi^\cc_{n,\epsilon})_2
		\right)
	}_{L^\infty((-D_{n,\varepsilon},D_{n,\varepsilon}))} 
	\leq 
	\epsilon,
\end{equation}
with $B_{n,\epsilon}$ independent of $\kappa$
and bounded for each $n \in \IN$ according to
\begin{equation}
	B_{n,\epsilon}
	\leq
	C
	\max\left\{
		n+2,
		\left(
			\widetilde{C}_n 
			D^{n+1}_{n,\varepsilon}
			\lambda_{n+1}
		\right)^{\frac{1}{n-1}}
	\right\},
\end{equation}
with $C,\widetilde{C}_n>0$ independent of $\kappa$
and $C$ independent of $n$.
Furthermore, 
we have\footnote{We have reasonably assumed $D^+_{{n,\epsilon}}\geq D^-_{{\epsilon}}$. }
\begin{equation}\label{eq:bound_Bnepsilon}
	B_{n,\epsilon}
	\leq
	C_n
	\epsilon^{-\frac{n+1}{(n-1)(3n+2)}}.
\end{equation}

{\sf \encircle{4}}
{\sf Approximation of $\Psi(\tau)$ in $\mathcal{I}_\kappa$.}
In steps {\sf \encircle{1}} through {\sf \encircle{3}} we have
constructed local approximations for $\Psi(\tau)$ in three different
regions of the interval $\mathcal{I}_\kappa$.
Here we combine them using a suitable partition of unity
over the wavenumber-dependent interval $\mathcal{I}_\kappa$, 
as defined in \eqref{eq:interval_I_k}.

Let us set
\begin{align}
	\chi(x)
	\coloneqq 
	\varrho(x+1)
	-
	2\varrho(x)
	+
	\varrho(x-1)
	\in 
	\mathcal{N\!N}_{2,3,1,1},
\end{align}
and, for $\ell\in \mathbb{N}$, $\chi_\ell(x)\coloneqq\chi(x-\ell)\in \mathcal{N\!N}_{2,3,1,1}$.
Observe that
\begin{align}
\label{eq:partition_unity}
	\sum_{\ell=-3}^{3}
	\chi_\ell(x)
	=1,
	\quad
	x
	\in[-3,3],
\end{align}
thus yielding a partition of unity of the bounded interval $[-3,3]$.

Define the map
$T_{\kappa,n,\epsilon}: \overline{\mathcal{I}_\kappa} \rightarrow [-3,3]$
as
\begin{equation}\label{eq:T_kappa_n_eps}
	T_{\kappa,n,\epsilon}(x)
	=
	\left\{
	\begin{array}{cl}
		-3+\frac{x+\kappa^{\frac{1}{3}} A}{-D^{-}_{\epsilon}+\kappa^{\frac{1}{3}} A},
		& 
		x\in [-\kappa^{\frac{1}{3}} A,-D^{-}_{\epsilon}),
		\\
		\frac{x}{D^{-}_{\epsilon}/2}, 
		& 
		x\in [-D^{-}_{\epsilon},0),
		\\
		\frac{x}{D^+_{{n,\epsilon}}/2}, 
		& 
		x\in [0,D^+_{{n,\epsilon}}),
		\\
		2+\frac{x-D^+_{{n,\epsilon}}}{\kappa^{\frac{1}{3}} A-D^+_{{n,\epsilon}}},
		& 
		x\in [D^+_{{n,\epsilon}},\kappa^{\frac{1}{3}} A].
	\end{array}
	\right.
\end{equation}
In the definition of $T_{\kappa,n,\epsilon}$ in \eqref{eq:T_kappa_n_eps}
we have implicitely assumed that $\kappa^\frac{1}{3}A \geq \max\{D^+_{{n,\epsilon}},D^{-}_{\epsilon}\}$. 
Therefore, under the reasonable assumption that $D^+_{{n,\epsilon}}\geq D^{-}_{\epsilon}$, 
this implies that
the bounds presented here are valid for $\kappa$ satisfying \eqref{eq:condition_kappa}.

For $x \in \mathcal{I}_\kappa$
\begin{equation}
\begin{aligned}
	T_{\kappa,n,\epsilon}(x)
	=
	&
	-3
	+
	\frac{\varrho(x+\kappa^{\frac{1}{3}} A)}{-D^{-}_{\epsilon}+\kappa^{\frac{1}{3}} A}
	+
	\left(
		\frac{1}{D^{-}_{\epsilon}}
		-
	 	\frac{1}{-D^{-}_{\epsilon}+\kappa^{\frac{1}{3}} A}
	\right)
	\varrho
	\left(
		x+D^{-}_{\epsilon}
	\right)
	\\
	&
	+
	\left(
	 	\frac{1}{D^{+}_{n,\epsilon}}
		-
		\frac{1}{D^{-}_{\epsilon}}
	\right)
	\varrho
	\left(
		x
	\right)
	+
	\left(
		\frac{1}{\kappa^{\frac{1}{3}} A-D^{+}_{n,\epsilon}}
		-
	 	\frac{1}{D^{+}_{n,\epsilon}}
	\right)
	\varrho
	\left(
		x-D^{+}_{n,\epsilon}
	\right).
\end{aligned}
\end{equation}
Thus
\begin{equation}\label{eq:T_NN}
	T_{\kappa,n,\epsilon}(x) 
	=
	{\bf W}_2
	\varrho
	\left(
		{\bf W}_1
		x
		+
		{\bf b}_1
	\right)
	+
	\mathbf{b}_2
	\in \mathcal{N\!N}_{2,4,1,1},
\end{equation}
where
\begin{equation}
	\mathbf{W}_1
	=
	\begin{pmatrix}
		1 \\
		1 \\
		1 \\
		1
	\end{pmatrix}
	\in 
	\mathbb{R}^{4\times 1}
	,
	\quad
	\mathbf{W}_2
	=
	\begin{pmatrix}
	\frac{1}{-D^{-}_{\epsilon}+\kappa^{\frac{1}{3}} A} \\
	\frac{1}{D^{-}_{\epsilon}}
	-
	\frac{1}{-D^{-}_{\epsilon}+\kappa^{\frac{1}{3}} A} \\
	\frac{1}{D^{+}_{n,\epsilon}}
	-
	\frac{1}{D^{-}_{\epsilon}} \\
	\frac{1}{\kappa^{\frac{1}{3}} A-D^{+}_{n,\epsilon}}
	-
	\frac{1}{D^{+}_{n,\epsilon}}
	\end{pmatrix}^\top
	\in
	\mathbb{R}^{1\times 4},
\end{equation}
and\footnote{Clearly $\mathbf{b}_2$ is a scalar value.
However, we use this notation to comply with the structure of Definition \ref{def:DNN}.}
\begin{equation}
	\mathbf{b}_1
	=
	\begin{pmatrix}
		\kappa^{\frac{1}{3}} A  \\
		D^{-}_{\epsilon}  \\
		0 \\
		-D^{+}_{n,\epsilon} 
	\end{pmatrix}
	\in
	\mathbb{R}^{4\times 1}
	,
	\quad
	\mathbf{b}_2
	=
	-
	3
	\in 
	\mathbb{R}^{1\times 1}.
\end{equation}
However, 
one can readily see that
\begin{equation}
	\norm{{\bf W}_1}_\infty
	\leq
	1,
	\quad
	\norm{{\bf W}_2}_\infty
	\leq
	\frac{1}{D^{+}_{n,\epsilon}}
	+
	\frac{1}{D^{-}_{\epsilon}}
	=
	2,
	\quad
	\norm{{\bf b}_1}_\infty
	\leq
	\kappa^{\frac{1}{3}} A,
	\quad
	\text{and}
	\quad
	\norm{{\bf b}_2}_\infty
	\leq
	3,
\end{equation}
which is not satisfactory as the weights are bounded in absolute value 
by $\kappa^{\frac{1}{3}} A$.

By recalling Proposition \ref{prop:trading_weights}, we conclude that 
there exists a ReLU-NN $T_\kappa \in \mathcal{N\!N}_{L,3,1,1}$ 
with 
$L\leq \log(\kappa^{\frac{1}{3}} A)+5$ 
and weights bounded in absolute value
by $1$ such that $T_\kappa(x) =\varrho\left(x+\kappa^{\frac{1}{3}} A\right)$, 
thus
\begin{equation}\label{eq:representation_T}
\begin{aligned}
	T_{\kappa,n,\epsilon}(x)
	=
	&
	-3
	+
	\frac{T_\kappa(x)}{-D^{-}_{\epsilon}+\kappa^{\frac{1}{3}} A}
	+
	\left(
		\frac{1}{D^{-}_{\epsilon}}
		-
	 	\frac{1}{-D^{-}_{\epsilon}+\kappa^{\frac{1}{3}} A}
	\right)
	\varrho
	\left(
		x+D^{-}_{\epsilon}
	\right)
	\\
	&
	+
	\left(
	 	\frac{1}{D^{+}_{n,\epsilon}}
		-
		\frac{1}{D^{-}_{\epsilon}}
	\right)
	\varrho
	\left(
		x
	\right)
	+
	\left(
		\frac{1}{\kappa^{\frac{1}{3}} A-D^{+}_{n,\epsilon}}
		-
	 	\frac{1}{D^{+}_{n,\epsilon}}
	\right)
	\varrho
	\left(
		x-D^{+}_{n,\epsilon}
	\right)
	\\
	=
	&
	{\bf  W}_2
	\begin{pmatrix}
		T_\kappa(x) \\
		\varrho
		\left(
			{\bf A}
			x
			+
			{\bf c}
		\right)
	\end{pmatrix}
	+
	\mathbf{b}_2
\end{aligned}
\end{equation}
with 
\begin{equation}
	\mathbf{A}
	=
	\begin{pmatrix}
		1 \\
		1 \\
		1
	\end{pmatrix}
	\in 
	\mathbb{R}^{3\times 1}
	\quad
	\text{and}
	\quad
	\mathbf{c}
	=
	\begin{pmatrix}
		D^{-}_{\epsilon}  \\
		0 \\
		-D^{+}_{n,\epsilon} 
	\end{pmatrix}
	\in
	\mathbb{R}^{3\times 1}.
\end{equation}
Therefore, \eqref{eq:representation_T} 
reveals that 
$T_{\kappa,n,\epsilon} \in \mathcal{N\!N}_{L,6,1,1}$
with 
$L\leq \log(\kappa^{\frac{1}{3}} A)+5$ 
and with weights bounded in absolute value by $1$.

Observe that
\begin{align}\label{eq:partition_unity_T}
	\sum_{\ell=-3}^{3}
	\left(
		\chi_\ell
		\circ
		T_{\kappa,n,\epsilon}
	\right)(\tau)
	=1,
	\quad
	\tau
	\in
	\mathcal{I}_\kappa.
\end{align}
The function $\Psi$ admits the representation
\begin{equation}\label{eq:partition_unity_T_Psi}
	\Psi(\tau)
	=
	\sum_{k=-3}^{3}
	\left(
		\chi_{\ell}
		\circ
		T_{\kappa,n,\epsilon}
	\right)
	(\tau)
	\Psi(\tau),
	\quad
	\tau
	\in 
	\mathcal{I}_\kappa.
\end{equation}

For each $n \in \IN_{\geq2}$ let us set 
\begin{equation}
	\Psi_{\kappa,n,\ell,\epsilon}
	\coloneqq
	\left\{
	\begin{array}{cl}
	\Psi^-_{\epsilon}, & \ell \in \{-3,-2\}, \\
	\Psi^\cc_{n,\epsilon}, & \ell \in \{-1,0,1\}, \\
	\Psi^+_{n,\epsilon}, & \ell \in \{2,3\}. \\
	\end{array}
	\right.
\end{equation}

Set $\lambda_{\kappa}:=\norm{\Psi}_{L^\infty\left(\mathcal{I}_\kappa\right)}$. 
According to Remark \ref{rmk:bound_psi} it holds
$\lambda_{\kappa} \leq C_{\Psi}(1+\kappa)^{\frac{1}{3}}$ with $C_{\Psi}$
independent of $\kappa$.
Let $\mu_{\lambda_{\kappa},\epsilon}\in \mathcal{N\!N}_{L,5,2,1}$ with 
$L\leq C\left(\log( \lambda_{\kappa})+ \log\left(\frac{1}{\epsilon}\right)\right)$
as $\epsilon\rightarrow 0$
be the multiplication network from Proposition \ref{prop:mult_network} satisfying
\begin{align}\label{eq:error_mult_network}
	\norm{
		\mu_{\lambda_\kappa,\epsilon}(x,y)-xy
	}_{L^\infty
	\left(\left(-2\lambda_\kappa,2\lambda_\kappa\right)^2\right)} 
	\leq 
	\epsilon.
\end{align}
For $\epsilon>0$ and $n\in \IN$ we set
\begin{equation}
\label{eq:network_approximation_D}
	\Psi_{\kappa,n,\epsilon}(\tau)
	\coloneqq
	\sum_{\ell=-3}^{3} 
	\begin{pmatrix}
	\mu_{\lambda_\kappa,\epsilon}
	\left(
		\left(
			\chi_{\ell}\circ T_{\kappa,n,\epsilon}
		\right)(\tau)
		,
		\left(\Psi_{\kappa,n,\ell,\epsilon}(\tau)\right)_1
	\right) \\
	\mu_{\lambda_\kappa,\epsilon}
	\left(
		\left(
			\chi_{\ell}\circ T_{\kappa,n,\epsilon}
		\right)(\tau)
		,
		\left(\Psi_{\kappa,n,\ell,\epsilon}(\tau)\right)_2
	\right)
	\end{pmatrix},
	\quad
	\tau
	\in 
	\mathcal{I}_\kappa.
\end{equation}

We argue that $\Psi_{\kappa,n,\epsilon}(\tau)$ as in \eqref{eq:network_approximation_D}
is indeed a ReLU-NN of a fixed width to be determined.
Firstly, for $\ell \in \{-3,\dots,3\}$ we set
\begin{align}
	\Theta_0(x)
	\coloneqq
	\begin{pmatrix}
	x \\
	x \\
	0 \\
	0
	\end{pmatrix}
	\quad
	\text{and}
	\quad
	\Theta_{\ell,\epsilon}(x_1,x_2,x_3,x_4)
	\coloneqq
	A
	\begin{pmatrix}
	x_1 \\
	\mu_{\lambda_\kappa,\epsilon}
	\left(
		\left(
			\chi_{\ell}\circ T_{\kappa,n,\epsilon}
		\right)(x_2)
		,
		\left(\Psi_{\kappa,\ell,n,\epsilon}(x_2)\right)_1
	\right)
	\\
	\mu_{\lambda_\kappa,\epsilon}
	\left(
		\left(
			\chi_{\ell}\circ T_{\kappa,n,\epsilon}
		\right)(x_2)
		,
		\left(\Psi_{\kappa,\ell,n,\epsilon}(x_2)\right)_2
	\right)
	\\
	x_3 \\
	x_4
	\end{pmatrix}
\end{align}
where $A\in \mathbb{R}^{5\times 4}$ is the matrix that
performs the operation $A(y_1,y_2,y_3,y_4,y_5)^\top = (y_1,y_1,y_2+y_4,y_3+y_5)^\top$,
thus
\begin{equation}
\begin{aligned}
	\Psi_{\kappa,n,\epsilon}(\tau)
	=
	\begin{pmatrix}
		0 & 0 & 1 & 0 \\
		0 & 0 & 0 & 1 \\
	\end{pmatrix}
	\left(
	\Theta_{3,\epsilon}
	\circ
	\cdots
	\circ
	\Theta_{-3,\epsilon}
	\circ
	\Theta_0
	\right)(\tau),
	\quad
	\tau \in \mathcal{I}_\kappa,
\end{aligned}
\end{equation}
thus rendering $\Psi_{\kappa,n,\epsilon}$ a ReLU-NN itself.

{\sf \encircle{5}}
{\sf Bounds for the Width, Depth, and Weights}.
Recalling \eqref{eq:partition_unity_T} and \eqref{eq:partition_unity_T_Psi}
we obtain for $\tau \in \mathcal{I}_{\kappa}$
\begin{equation}
\begin{aligned}
	\Psi(\tau)
	-
	&
	\left(
		\Psi_{\kappa,n,\epsilon}(\tau)
	\right)_1
	-
	\imath
	\left(
		\Psi_{\kappa,n,\epsilon}(\tau)
	\right)_2
	\\
	=
	&
	\sum_{\ell=-3}^{3} \Psi(\tau) \chi_\ell(\tau) 
	-
	\left(
		\Psi_{\kappa,n,\epsilon}(\tau)
	\right)_1
	-
	\imath
	\left(
		\Psi_{\kappa,n,\epsilon}(\tau)
	\right)_2 \\
	=
	&
	\sum_{\ell=-3}^{3}
	\left(
	\Psi(\tau) \chi_{\ell}(\tau)
	-
	\mu_{\lambda_\kappa,\epsilon}
	\left(
	\chi_{\ell}(\tau)
	,
	\left(\Psi_{\kappa,\ell,n,\epsilon}(\tau)\right)_1
	\right)
	-
	\imath
	\mu_{\lambda_\kappa,\epsilon}
	\left(
	\chi_{\ell}(\tau)
	,
	\left(\Psi_{\kappa,\ell,n,\epsilon}(\tau)\right)_2
	\right)
	\right).
\end{aligned}
\end{equation}
Next, for each $\ell \in \{-3,\dots,3\}$
\begin{equation}
\begin{aligned}
	\norm{
		\Psi \chi_{\ell}
		-
		\mu_{\lambda_\kappa,\epsilon}
		\left(
		\chi_{\ell}
		,
		\left(
			\Psi_{\kappa,\ell,n,\epsilon}\right)_1
		\right)
		-
		\imath
		\mu_{\lambda_\kappa,\epsilon}
		\left(
		\chi_{\ell}
		,
		\left(
			\Psi_{\kappa,\ell,n,\epsilon}\right)_2
		\right)
	}_{L^\infty\left(\mathcal{I}_\kappa\right)}
	\\
	&
	\hspace{-4cm}
	\leq
	\norm{
		\chi_\ell
		\left(
			\Psi
			-
			\left(
				\Psi_{\kappa,\ell,n,\epsilon}
			\right)_1
			-
			\imath
			\left(
				\Psi_{\kappa,\ell,n,\epsilon}
			\right)_2
		\right)
	}_{L^\infty\left(\mathcal{I}_\kappa\right)}
	\\
	&
	\hspace{-3.7cm}
	+
	\norm{
		\mu_{\lambda_\kappa,\epsilon}
		\left(
		\chi_{\ell}
		,
		\left(
			\Psi_{\kappa,\ell,n,\epsilon}
		\right)_1
		\right)
		-
		\chi_{\ell}
		\left(
			\Psi_{\kappa,\ell,n,\epsilon}
		\right)_1
	}_{L^\infty\left(\mathcal{I}_\kappa\right)}
	\\
	&
	\hspace{-3.7cm}
	+
	\norm{
		\mu_{\lambda_\kappa,\epsilon}
		\left(
		\chi_{\ell}
		,
		\left(
			\Psi_{\kappa,\ell,n,\epsilon}
		\right)_2
		\right)
		-
		\chi_{\ell}
		\left(
			\Psi_{\kappa,\ell,n,\epsilon}
		\right)_2
	}_{L^\infty\left(\mathcal{I}_\kappa\right)}.
\end{aligned}
\end{equation}
As a consequence of steps {\sf \encircle{1}} through {\sf \encircle{3}},
and the compactness of the support of $\chi_\ell$
\begin{equation}
	\norm{
		\chi_\ell
		\left(
			\Psi
			-
			\left(
				\Psi_{\kappa,\ell,n,\epsilon}
			\right)_1
			-
			\imath
			\left(
				\Psi_{\kappa,\ell,n,\epsilon}
			\right)_2
		\right)
	}_{L^\infty\left(\mathcal{I}_\kappa\right)}
	\leq
	\epsilon,
	\quad
	\ell
	\in \{-3,\dots,3\}.
\end{equation}
and recalling \eqref{eq:error_mult_network}, we get
for $\ell\in \{-3,\dots,3\}$
\begin{equation}
	\norm{
		\Psi \chi_{\ell}
		-
		\mu_{\lambda_\kappa,\epsilon}
		\left(
		\chi_{\ell}
		,
		\left(
			\Psi_{\kappa,\ell,n,\epsilon}\right)_1
		\right)
		-
		\imath
		\mu_{\lambda_\kappa,\epsilon}
		\left(
		\chi_{\ell}
		,
		\left(
			\Psi_{\kappa,\ell,n,\epsilon}\right)_2
		\right)
	}_{L^\infty\left(\mathcal{I}_\kappa\right)}
	\leq
	3\epsilon.
\end{equation}
The ReLU-NNs $\chi_\ell$ with $\ell$ even do not overlap. 
The same consideration holds for those $\chi_\ell$ with $\ell$ odd. 
Therefore, we obtain
\begin{equation}
\begin{aligned}
	\norm{
		\Psi
		-
		\left(
			\Psi_{\kappa,n,\epsilon}
		\right)_1
		-
		\imath
		\left(
			\Psi_{\kappa,n,\epsilon}
		\right)_2
	}_{L^\infty\left(\mathcal{I}_\kappa\right)}
	\\
	&
	\hspace{-5cm}
	\leq
	\max_{\substack{\ell \in\{-3,\dots,3\} \\ \ell \text{ even}}}
	\norm{
		\Psi \chi_{\ell}
		-
		\mu_{\lambda_\kappa,\epsilon}
		\left(
		\chi_{\ell}
		,
		\left(
			\Psi_{\kappa,\ell,n,\epsilon}\right)_1
		\right)
		-
		\imath
		\mu_{\lambda_\kappa,\epsilon}
		\left(
		\chi_{\ell}
		,
		\left(
			\Psi_{\kappa,\ell,n,\epsilon}\right)_2
		\right)
	}_{L^\infty\left(\mathcal{I}_\kappa\right)}
	\\
	&
	\hspace{-4.7cm}
	+
	\max_{\substack{\ell \in\{-3,\dots,3\} \\ \ell \text{ odd}}}
		\norm{
		\Psi \chi_{\ell}
		-
		\mu_{\lambda_\kappa,\epsilon}
		\left(
		\chi_{\ell}
		,
		\left(
			\Psi_{\kappa,\ell,n,\epsilon}\right)_1
		\right)
		-
		\imath
		\mu_{\lambda_\kappa,\epsilon}
		\left(
		\chi_{\ell}
		,
		\left(
			\Psi_{\kappa,\ell,n,\epsilon}\right)_2
		\right)
	}_{L^\infty\left(\mathcal{I}_\kappa\right)}.
\end{aligned}
\end{equation}
Hence,
\begin{equation}
	\norm{
		\Psi
		-
		\left(
			\Psi_{\kappa,n,\epsilon}
		\right)_1
		-
		\imath
		\left(
			\Psi_{\kappa,n,\epsilon}
		\right)_2
	}_{L^\infty\left(\mathcal{I}_\kappa\right)}
	\leq
	6\epsilon.
\end{equation}
Next, we calculate the width of $\Psi_{\kappa,n,\epsilon}$.
For each $\ell \in \{-3,\dots,3\}$ one has
\begin{equation}
	\mathcal{M}
	\left(
		\Theta_{\ell,\epsilon}
	\right)
	=
	6
	+
	\max
	\left\{
		2
		\mathcal{M}
		\left(
			\mu_{\lambda_\kappa,\epsilon}
		\right)
		,
		\mathcal{M}
		\left(
			\chi_{\ell}\circ T_{\kappa,n,\epsilon}
		\right)
		+
		\mathcal{M}
		\left(
			\Psi_{\kappa,\ell,n,\epsilon}
		\right)
	\right\},
\end{equation}
thus
\begin{equation}
	\mathcal{M}
	\left(
		 \Psi_{\kappa,n,\epsilon}
	\right)
	=
	\max_{\ell \in \{-3,\dots,3\}}
	\mathcal{M}
	\left(
		\Theta_{\ell,\epsilon}
	\right)
	=
	30.
\end{equation}
The depth of $ \Psi_{\kappa,n,\epsilon}$ is bounded according to
\begin{equation}
\begin{aligned}
	\mathcal{L}
	\left(
		 \Psi_{\kappa,n,\epsilon}
	\right)
	\leq
	&
	\mathcal{L}
	\left(
		 \Theta_0
	\right)
	+
	\sum_{\ell=-3}^{3}
	\mathcal{L}
	\left(
		 \Theta_{\ell,\epsilon}
	\right)
	\\
	\leq
	&
	4
	+
	\sum_{\ell=-3}^{3}
	\left(
	\mathcal{L}
	\left(
		\mu_{\lambda_\kappa,\epsilon}
	\right)
	+
	\max
	\left\{
	\mathcal{L}
	\left(
		 \chi_{\ell}\circ T_{\kappa,n,\epsilon}
	\right)
	,
	\mathcal{L}
	\left(
		\Psi_{\kappa,\ell,n,\epsilon}
	\right)
	\right\}
	\right)
\end{aligned}
\end{equation}
thus
{\small
\begin{equation}
\begin{aligned}	
	\mathcal{L}
	\left(
		 \Psi_{\kappa,n,\epsilon}
	\right)
	\leq
	&
	C\left(\log( \lambda_{\kappa})+ \log\left(\frac{6}{\epsilon}\right)\right)
	\\
	&
	+
	C
	\left(
	\log
	\left(
		\frac{U_n}{\epsilon}
	\right)
	+
	\log(12n-2)
	\right)^2
	\left(
		\left(\log\left(\kappa^{\frac{1}{3}}A\right)\right)^2
		+
		\left(
			\log
			\left(
				\frac{2}{\epsilon}
				U_n
			\right)
			+
			\log(12n-2)
		\right)^2
	\right)
	\\
	&
	+
	C
	(3n-1)
	\left(
	\log
	\left(
		\frac{U_n}{\epsilon}
	\right)
	+
	\log(12n-2)
	\right)
	\\
	&
	+
	B_{n,\epsilon}
	\left(
		B_{n,\epsilon}
		\epsilon^{-\frac{2}{n-1}}
		+
		\epsilon^{-\frac{1}{n-1}}
		\log\left(\frac{1}{\epsilon}\right) 
		+
		\epsilon^{-\frac{2}{n-1}}
		\log
		\left(
			D_{n,\epsilon}
		\right)
		+
		\epsilon^{-\frac{1}{n-1}}
		\log
		\left(
			\lambda_n
		\right)
	\right)
	\\
	&
	+
	C
	\left(
		\log
		\left(
			\frac{1}{\epsilon}
		\right)
		+
		\left(
		\log
		\left(
			\frac{1}{\epsilon}
		\right)
		\right)^2
		+
		\log
		\left(
			\kappa^{\frac{1}{3}}A
		\right)
		+
		\left(
		\log
		\left(
			\kappa^{\frac{1}{3}}A
		\right)
		\right)^2
	\right),
	\quad
	\text{as }
	\epsilon \rightarrow 0,
\end{aligned}
\end{equation}
}%
with $B_n$ and $\lambda_n$ as in step {\sf \encircle{3}}.

By combining terms in the previous expression, 
recalling that $\lambda_n$ does not depend on $\kappa$,
and 
that\footnote{Again, we have reasonably assumed $D^+_{{n,\epsilon}}\geq D^-_{{\epsilon}}$.}
$\log(D_{n,\epsilon})\leq C_n \log\left(\frac{1}{\epsilon} \right)$
for some 
constant independent of $\kappa$, 
we obtain
\begin{equation}
\begin{aligned}
	\mathcal{L}
	\left(
		 \Psi_{\kappa,n,\epsilon}
	\right)
	\leq
	&
	C_n
	\left(
	\log
		\left(
			\frac{1}{\epsilon}
		\right)
	\right)^2
	\left(
		\log(\kappa)
		+
		\log
		\left(
			\kappa^{\frac{1}{3}}A
		\right)
		+
		\log
		\left(
			\kappa^{\frac{1}{3}}A
		\right)^2
		+
			\log
		\left(
			\frac{1}{\epsilon}
		\right)
		+
		\left(
		\log
		\left(
			\frac{1}{\epsilon}
		\right)
		\right)^2
	\right)
	\\
	&
	+
	B_{n,\epsilon}
	\left(
		B_{n,\epsilon}
		\epsilon^{-\frac{2}{n-1}}
		+
		\epsilon^{-\frac{1}{n-1}}
		\log\left(\frac{1}{\epsilon}\right) 
		+
		\epsilon^{-\frac{1}{n-1}}
		\log
		\left(
			\frac{1}{\epsilon}
		\right)
		+
		\epsilon^{-\frac{1}{n-1}}
	\right),
\end{aligned}
\end{equation}
as $\epsilon \rightarrow 0$,
where for each $n \in \IN$ the constant $C_n>0$ does not depend neither
on $\epsilon>0$ nor $\kappa$, and 
$B_{n,\epsilon}$ is bounded as in \eqref{eq:bound_Bnepsilon}.

Finally, an inspection of the proof reveals that the weights are bounded 
in absolute value by a constant independent of $\kappa$, but depending on $n$.
\end{proof}

\begin{proof}[Proof of Proposition \ref{prop:approx_FockInt_NN_der}]
As in Proposition \ref{prop:approx_FockInt_NN},
we divide the proof in five steps. Throughout this proof,
$\ell \in \mathbb{N}$ corresponds to the order of differentiation of
$\Psi(\tau)$.

{\sf \encircle{1} Approximation of $\Psi^{(\ell)}(\tau)$ for $\tau \rightarrow \infty$}.
Theorem \ref{thm:expansion_V} provides the following asymptotic expansion of $\Psi^{(\ell)}$:
For each $n \in \mathbb{N}$ it holds
\begin{align}
	\snorm{
	\Psi^{(\ell)}(\tau)
	-
	\widetilde{a}_{0,\ell}
	-
	\sum_{j=1}^{n}
	a_{j,\ell}
	\tau^{1-3j-\ell} }
	\leq  
	C_{n,\ell}\tau^{1-3(n+1)-\ell}
	\quad
	\text{as}\quad \tau\rightarrow\infty,
\end{align}
where
\begin{equation}
	\widetilde{a}_{0,\ell}
	=
	\left\{
	\begin{array}{cl}
		a_0, & \ell =1, \\
		0, & \ell \in \IN_{\geq 2}, \\
	\end{array}
	\right.
\end{equation}
where $a_{j,\ell} =a_j(1-3j)\cdots(1-3j-\ell+1)$,
and $a_j$ are exactly as in \eqref{eq:asymptotic_psi_tau_large}.

Let us set 
\begin{align}
	D^+_{{n,\ell,\epsilon}}
	\coloneqq
	2
	\max
	\left\{
		\left(2\frac{C_{n,\ell}}{\epsilon}\right)^{\frac{1}{3(n+1)+\ell-1}}
		,
		1
	\right \}.
\end{align}
Then, for $n,\ell \in \IN$
\begin{align}
	\norm{
	\Psi^{(\ell)}(\tau)
	-
	\widetilde{a}_{0,\ell}
	-
	\sum_{j=1}^{n}
	a_{j,\ell} \tau^{1-3j-\ell}
	}_{L^\infty\left(\left(\frac{1}{2}D^+_{{n,\ell,\epsilon}},\kappa^{\frac{1}{3}}A\right)\right)} 
	\leq  
	\frac{\epsilon}{2}.
\end{align}

By following the construction in step {\sf \encircle{1}} of Proposition 
\ref{prop:approx_FockInt_NN} we can conclude that there exist ReLU-NNs
$\Psi^+_{n,\ell,\epsilon} \in \mathcal{N\!N}_{L_\ell,17,1,2} $
such that
\begin{equation}
	\left(
		\Psi^+_{n,\ell,\epsilon}(\tau)
	\right)_1
	+
	\imath
	\left(
		\Psi^+_{n,\ell,\epsilon}(\tau)
	\right)_2
	=
	\widetilde{a}_{0,\ell}
	+
	\sum_{j=1}^{n}
	a_{j,\ell}
	\uppsi_{3j+\ell-1,\epsilon},
\end{equation}
with 
{\small
\begin{equation}
\begin{aligned}
	L_\ell
	\leq
	&
	C
	\left(
	\log
	\left(
		\frac{U_{n,\ell}}{\epsilon}
	\right)
	+
	\log(12n+4\ell-2)
	\right)^2
	\left(
		\left(\log\left(\kappa^{\frac{1}{3}}A\right)\right)^2
		+
		\left(
			\log
			\left(
				\frac{2}{\epsilon}
				U_{n,\ell}
			\right)
			+
			\log(12n+4\ell-2)
		\right)^2
	\right)
	\\
	&
	+
	C
	(3n+\ell-1)
	\left(
	\log
	\left(
		\frac{U_{n,\ell}}{\epsilon}
	\right)
	+
	\log(12n+4\ell-2)
	\right),
	\quad
	\text{as }
	\epsilon \rightarrow 0,
\end{aligned}
\end{equation}
}%
with $U_{n,\ell} = \sum_{j=1}^{n} \snorm{a_{j,\ell}} $, and satisfying
\begin{equation}
	\norm{
		\left(
			\Psi^+_{n,\ell,\epsilon}(\tau)
		\right)_1
		+
		\imath
		\left(
			\Psi^+_{n,\ell,\epsilon}(\tau)
		\right)_2
		-
		\widetilde{a}_{0,\ell}
		-
		\sum_{j=0}^{n}a_j \tau^{1-3j-\ell}
	}_{
		L^\infty
		\left(
		\left(
			\frac{1}{2}D^+_{{n,\epsilon}},\kappa^{\frac{1}{3}}A
		\right)
		\right)
	}
	\leq
	\frac{\epsilon}{2}.
\end{equation}
We conclude that
\begin{align}
	\norm{
	\Psi^{(\ell)}(\tau)
	-
	\left(
		\Psi^+_{n,\ell,\epsilon}(\tau)
	\right)_1
	-
	\imath
	\left(
		\Psi^+_{n,\ell,\epsilon}(\tau)
	\right)_2
	}_{L^\infty\left(\left(\frac{1}{2}D^+_{{n,\ell,\epsilon}},\kappa^{\frac{1}{3}}A\right)\right)} 
	\leq  
	\epsilon.
\end{align}

{\sf \encircle{2} Approximation of $\Psi^{(\ell)}(\tau)$ for $\tau \rightarrow -\infty$}.
Theorem \ref{thm:expansion_V} provides the following asymptotic expansion of $\Psi^{(\ell)}$
\begin{align}
	\Psi^{(\ell)}(\tau)
	=
	c_0 
	D^\ell_\tau 
	\left\{
		\exp(-\imath \tau^3/3-\imath \tau \alpha_1) 
	\right \}
	(1+ \mathcal{O}(\exp(-\beta \snorm{\tau}))
	\quad
	\text{as}\quad \tau\rightarrow-\infty.
\end{align}
One can readily observe that
\begin{equation}
	c_0
	D^\ell_\tau 
	\left\{
		\exp(-\imath \tau^3/3-\imath \tau \alpha_1) 
	\right \}
	=
	q_{2\ell}(\tau)
	\exp(-\imath \tau^3/3-\imath \tau \alpha_1) 
\end{equation}
where $q_{2\ell}$ is a polynomial of degree $2\ell$ 
(we have absorbed the constant $c_0$ into $q_{2\ell}$) with 
complex coefficients that are independent of $\kappa$.
Let us set
\begin{equation}\label{eq:phi_ell_def}
	\psi_{\ell}(\tau)
	\coloneqq
	q_{2\ell}(\tau)
	\exp(-\imath \tau^3/3-\imath \tau \alpha_1).
\end{equation}
It holds that
\begin{align}
	\snorm{
		\Psi^{(\ell)}(\tau) 
		- 
		\psi_{\ell}(\tau)
	} 
	\leq 
	C_{\ell,\kappa}
	\exp\left(\left(\frac{\nu_1\sqrt{3}}{2}-\beta\right)\snorm{\tau}\right)
	\quad
	\text{as } \tau \rightarrow-\infty,
\end{align}
with $C_{\ell,\kappa}>0$ given by
\begin{equation}\label{eq:constant_pol}
	C_{\ell,\kappa}
	\coloneqq
	\max_{\tau \in 
	\left[
		-\kappa^{\frac{1}{3}} A,\kappa^{\frac{1}{3}} A
	\right]}
	\snorm{
		q_{2\ell}(\tau)
	}
	\leq
	C_\ell
	\left(
		\kappa^{\frac{1}{3}} A
	\right)^{2\ell}.
\end{equation}
Observe that $\frac{\nu_1\sqrt{3}}{2}-\beta<0$.
Let us set 
\begin{align}
	D^{-}_{\ell,\kappa,\epsilon}
	\coloneqq
	2
	\max
	\left\{
		\snorm{\frac{\nu_1\sqrt{3}}{2}-\beta}^{-1} 
		\log \left (2\frac{C_{\ell,\kappa}}{\epsilon}\right)  
		,
		1
	\right\}.
\end{align}
Therefore
\begin{align}
	\norm{
		\Psi^{(\ell)}(\tau) 
		- 
		\psi_{\ell}(\tau)
	}_{L^\infty\left(
	\left(
		-\kappa^{\frac{1}{3}}A
		,
		-
		\frac{1}{2}D^-_{\ell,\kappa,\epsilon}
	\right)\right)}
	\leq
	\frac{\epsilon}{2}.
\end{align}
Observe that 
\begin{equation}\label{eq:real_img_app_q2l}
\begin{aligned}
	\Re\{\psi_{\ell}(\tau)\}
	&
	=
	\left(
	\Re\{q_{2\ell}(\tau)\}
	\cos
	\left(
		\frac{\tau \nu_1}{2} 
		- 
		\frac{\tau^3}{3}
	\right)
	-
	\Im\{q_{2\ell}(\tau)\}
	\sin
	\left(
		\frac{\tau \nu_1}{2} 
		- 
		\frac{\tau^3}{3}
	\right)
	\right)
	\exp
	\left(
		\frac{\nu_1\sqrt{3}}{2}\snorm{\tau}
	\right),
\end{aligned}
\end{equation}
and
\begin{equation}
\begin{aligned}
	\Im\{\psi_{\ell}(\tau)\}
	&
	=
	\left(
	\Re\{q_{2\ell}(\tau)\}
	\sin
	\left(
		\frac{\tau \nu_1}{2} 
		- 
		\frac{\tau^3}{3}
	\right)
	+
	\Im\{q_{2\ell}(\tau)\}
	\cos
	\left(
		\frac{\tau \nu_1}{2} 
		- 
		\frac{\tau^3}{3}
	\right)
	\right)
	\exp
	\left(
		\frac{\nu_1\sqrt{3}}{2}\snorm{\tau}
	\right).
\end{aligned}
\end{equation}
Thus,
\begin{align}\label{eq:error_real_DtPsi}
	\norm{
		\Re\left\{ \Psi^{(\ell)}(\tau)\right\}
		- 
		\left(
			\Re\left\{q_{2\ell}(\tau)\right\}
			\phi_{1}(\tau)
			-
			\Im\left\{q_{2\ell}(\tau)\right\}
			\phi_{2}(\tau)
		\right)
	}_{L^\infty\left(\left(-\kappa^{\frac{1}{3}}A,-\frac{1}{2}D^-_{\ell,\kappa,\epsilon}\right)\right)}
	\leq
	\frac{
		\epsilon
	}{
		2
	},
\end{align}
and
\begin{align}\label{eq:error_imag_DtPsi}
	\norm{
		\Im\left\{\Psi^{(\ell)}(\tau) \right\}
		- 
		\left(
			\Re\left\{q_{2\ell}(\tau)\right\}
			\phi_{2}(\tau)
			+
			\Im\left\{q_{2\ell}(\tau)\right\}
			\phi_{1}(\tau)
		\right)
	}_{L^\infty\left(\left(-\kappa^{\frac{1}{3}}A,-\frac{1}{2}D^-_{\ell,\kappa,\epsilon}\right)\right)}
	\leq
	\frac{
		\epsilon
	}{
		2
	},
\end{align}
where $\phi_1(\tau)$ and $\phi_2(\tau)$ are as in 
\eqref{eq:def_phi}. The ReLU-NN emulation of these
functions has been done in step {\sf \encircle{2}} of
Proposition \ref{prop:approx_FockInt_NN}.


Next, we use Lemma \ref{lmm:product_of_complex_NNs}
to approximate $\psi_{\ell}(\tau)$ as defined in \eqref{eq:phi_ell_def}.
We verify the hypothesis item-by-item
\begin{itemize}
\item[(i)]
According to Proposition \ref{prop:approximation_polynomials}
and Remark \ref{def:DNN} there exist a constant $C>0$, 
independent of $\kappa$, and ReLU-NNs $\Phi_{q_{2\ell},\epsilon} \in \mathcal{N\!N}_{L,11,1,2}$ 
with
\begin{equation}
	L_{q_{2\ell}}
	(\epsilon)
	\leq
	C
	\left(
		2\ell
		\log
		\left(
			\frac{1}{\epsilon}
		\right)
		+
		4\ell^2
		\log
		\left(
			\kappa
		\right)
		+
		2\ell \log(2\ell)
		+
		2\ell 
		\log
		\left(
			C_{\ell,\kappa}
		\right)
	\right),
	\quad
	\text{as }
	\epsilon \rightarrow 0,
\end{equation}
which satisfies
\begin{equation}
	\norm{
		q_{2\ell}
		-
		\left(
		\left(
			 \Phi_{q_{2\ell},\epsilon} 
		\right)_1
		+
		\imath
		\left(
			 \Phi_{q_{2\ell},\epsilon} 
		\right)_2
		\right)
	}_{L^\infty(\mathcal{I}_\kappa)}
	\leq
	\epsilon,
\end{equation}
with weights bounded in absolute value
by a constant independent of $\kappa$.
\item[(ii)]
An inspection of step {\sf \encircle{2}} in Proposition \ref{prop:approx_FockInt_NN}
reveals the existence of ReLU-NNs emulating the map 
$\tau \mapsto \left(\phi_1(\tau),\phi_2(\tau) \right)^\top$.
Let us set
\begin{equation}
	\Theta_0(x)
	=
	\begin{pmatrix}
	0 \\
	0 \\
	x
	\end{pmatrix}
	\quad
	\text{and}
	\quad
	\Theta_{i,\epsilon}(x_1,x_2,x_3)
	=
	A_i
	\begin{pmatrix}
		x_1 \\
		x_2 \\
		x_3 \\
		\phi_{i,\epsilon}(x_3)
	\end{pmatrix}
	\quad
	\text{for }\\
	i=1,2,
\end{equation}
with
\begin{equation}
	A_1
	\coloneqq
	\begin{pmatrix}
		1 & 0 & 0 & 1 \\
		0 & 1 & 0 & 0 \\
		0 & 0 & 1 & 0
	\end{pmatrix}
	\in 
	\mathbb{R}^{3\times 4}
	\quad
	\text{and}
	\quad
	A_2
	\coloneqq
	\begin{pmatrix}
		1 & 0 & 0 & 0 \\
		0 & 1 & 0 & 1 \\
		0 & 0 & 1 & 0
	\end{pmatrix}
	\in 
	\mathbb{R}^{3\times 4}.
\end{equation}	
Define the ReLU-NN 
\begin{equation}
	\Phi_{\epsilon}(\tau)
	=
	\begin{pmatrix}
	1 & 0 & 0  \\
	0 & 1 & 0 
	\end{pmatrix}
	\left(
		\Theta_{2,\epsilon}
		\circ
		\Theta_{1,\epsilon}
		\circ
		\Theta_0
	\right)(\tau).
\end{equation}
There exist $C>0$ and ReLU-NNs 
$\Phi_{\epsilon} \in \mathcal{N\!N}_{L,20,1,2}$ with
\begin{equation}
	L_{\phi_1,\phi_2}
	(\epsilon)
	\leq
	C
	\left(
		\left(
		\log
		\left(
			\frac{1}{\epsilon}
		\right)
		\right)^2
		+
		\left(
		\log
		\left(
		\kappa
		\right)
		\right)^2
	\right)
	\quad
	\text{as }
	\epsilon
	\rightarrow 0,
\end{equation}
such that
\begin{equation}
	\norm{
		\left(
			\Phi_{\epsilon}(\tau)
		\right)_1
		+
		\imath
		\left(
			\Phi_{\epsilon}(\tau)
		\right)_2
		-
		\phi_1(\tau)
		-
		\imath
		\phi_2(\tau)
	}_{L^\infty(\mathcal{I}_\kappa)}
	\leq
	\epsilon.
\end{equation}
\end{itemize}
It follows from Lemma \ref{lmm:product_of_complex_NNs} that
there exist ReLU-NNs $\Psi^{-}_{\ell,\epsilon}\in \mathcal{N\!N}_{L,22,1,2}$ 
such that
\begin{equation}
	\norm{
		q_{2\ell}(\tau)
		\left(
			\phi_1(\tau)
			+
			\imath
			\phi_2(\tau)
		\right)
		-
		\left(
			\Psi^{-}_{\ell,\epsilon}(\tau)
		\right)_1
		-
		\imath
		\left(
			\Psi^{-}_{\ell,\epsilon}(\tau)
		\right)_2
	}_{L^\infty(\mathcal{I}_\kappa)} 
	\leq
	\frac{\epsilon}{2}.
\end{equation}
and
\begin{equation}\label{eq:bound_L_kappa}
\begin{aligned}
	L
	\leq
	&
	C
	\left(
	\log
	\left(
		\lceil  U_{\ell,\kappa}  \rceil
	\right)
	+
	\log
	\left(
		\frac{4}{\epsilon}
	\right)
	\right)
	+
	2
	\max
	\left\{
		L_{q_{2\ell}}
		\left(
			\frac{\epsilon}{4}
		\right)
		,2
	\right\}
	+
	2
	\max
	\left\{
		L_{\phi_1,\phi_2}
		\left(
			\frac{\epsilon}{4}
		\right)
		,2
	\right\}
	\\
	\leq
	&
	C
	\left(
	\log
	\left(
		\lceil U_{\ell,\kappa} \rceil
	\right)
	+
	\log
	\left(
		\frac{4}{\epsilon}
	\right)
	\right)
	\\
	&
	+
	C
	\left(
		\ell
		\log
		\left(
			\frac{1}{\epsilon}
		\right)
		+
		\ell^2
		\log
		\left(
			\kappa
		\right)
		+
		\ell \log(2\ell)
		+
		\ell 
		\log
		\left(
			C_{\ell,\kappa}
		\right)
	\right)
	\\
	&
	+
	C
	\left(
		\left(
		\log
		\left(
			\frac{1}{\epsilon}
		\right)
		\right)^2
		+
		\left(
		\log
		\left(
		\kappa
		\right)
		\right)^2
	\right)
\end{aligned}
\end{equation}
with 
\begin{equation}
	U_{\ell,\kappa} 
	= 
	2\max\left\{\norm{q_{2\ell}}_{L^\infty(\mathcal{I}_\kappa)},\norm{\phi_1(\tau)+\imath\phi_2(\tau)}_{L^\infty(\mathcal{I}_\kappa)}\right\} 
	= 
	2C_{\ell,\kappa}, 
\end{equation}
with $C_{\ell,\kappa}$ as in
\eqref{eq:constant_pol}, and $C>0$ independent of $\kappa$ and $\ell \in \IN$.

Recalling \eqref{eq:error_real_DtPsi} and \eqref{eq:error_imag_DtPsi} 
we conclude that $\Psi^{-}_{\ell,\epsilon}\in \mathcal{N\!N}_{L,24,1,2}$
satisfies
\begin{equation}
	\norm{
		\Psi^{(\ell)}(\tau) 
		-
		\left(
			\Psi^{-}_{\ell,\epsilon}(\tau)
		\right)_1
		-
		\imath
		\left(
			\Psi^{-}_{\ell,\epsilon}(\tau)
		\right)_2
	}_{L^\infty\left(\left(-\kappa^{\frac{1}{3}}A,-\frac{1}{2}D^-_{\ell,\kappa,\epsilon}\right)\right)}
	\leq
	\epsilon,
\end{equation}
with $L$ bounded as in \eqref{eq:bound_L_kappa}.

{\sf \encircle{3} 
Approximation of $\Psi^{(\ell)}(\tau)$ for $\tau$ bounded.}
Set 
\begin{equation}
	D_{n,\ell,\kappa,\epsilon} 
	\coloneqq 
	\max \left\{ D^+_{{n,\ell,\epsilon}},D^-_{{\ell,\kappa,\epsilon}}\right\}.
\end{equation}
Remark \ref{rmk:bound_psi} implies that
$\Psi ^{(\ell)}\in \mathcal{P}_{\boldsymbol{\lambda},D_{n,\ell,\kappa,\epsilon},\mathbb{C}}$
with $\boldsymbol{\lambda}_\ell \coloneqq\{\lambda_{n,\ell}\}_{n\in \mathbb{N}_0}$ 
defined as
$\lambda_{n,\ell}\leq C_{n,\ell}$
with $C_{n,\ell}>0$ as in Remark \ref{rmk:bound_psi},
thus independent of $\kappa$.
According to Proposition \ref{prop:approximation_P}
for each $n\in \IN_{\geq 2}$ there exist ReLU-NNs 
$\Psi^\cc_{n,\epsilon} \in \mathcal{N\!N}_{L,15,1,2}$
with 
\begin{equation}
	L
	\leq 
	B_{n,\ell,\kappa,\epsilon}
	\left(
		B_{n,\ell,\kappa,\epsilon}
		\epsilon^{-\frac{2}{n-1}}
		+
		\epsilon^{-\frac{1}{n-1}}
		\log\left(\frac{1}{\epsilon}\right) 
		+
		\epsilon^{-\frac{2}{n-1}}
		\log
		\left(
			D_{n,\ell,\kappa,\epsilon}
		\right)
		+
		\epsilon^{-\frac{1}{n-1}}
		\log
		\left(
			\lambda_{n,\ell}
		\right)
	\right)
\end{equation}
as $\epsilon\rightarrow0$ such that 
\begin{equation}
	\norm{
		\Psi^{(\ell)}
		-
		\left(
			\left(\Psi^\cc_{n,\ell,\epsilon}\right)_1
			+
			\imath
			\left(\Psi^\cc_{n,\ell,\epsilon}\right)_2
		\right)
	}_{L^\infty((-D_{n,\ell,\kappa,\epsilon},D_{n,\ell,\kappa,\epsilon}))} 
	\leq 
	\epsilon,
\end{equation}
with $B_{n,\ell,\kappa,\epsilon}>0$ bounded according to
\begin{equation}
	B_{n,\ell,\kappa,\epsilon}
	\leq
	C
	\max\left\{
		n+2,
		\left(
			\widetilde{C}_n 
			D^{n+1}_{n,\ell,\kappa,\epsilon}
			\lambda_{n+1,\ell}
		\right)^{\frac{1}{n-1}}
	\right\},
\end{equation}
and $C,\widetilde{C}_n>0$ independent of $\kappa$
and $C$ independent of $n$.
Furthermore, we have that
\begin{equation}\label{eq:bound_Bnepsilon_ell}
	B_{n,\ell,\kappa,\epsilon}
	\leq
	C_{n,\ell}
	\left(
		\log\left(\kappa^\frac{1}{3}A\right)
		+
		\log \left (\frac{1}{\epsilon}\right)  
	\right)^{\frac{n+1}{n-1}}
	\epsilon^{-\frac{n+1}{(n-1)(3n+\ell+2)}},
\end{equation}
with $C_{n,\ell}$ not depending on $\kappa$ or $\epsilon$.

{\sf \encircle{4}}
{\sf Approximation of $\Psi^{(\ell)}(\tau)$ in $\mathcal{I}_\kappa$.}
As in the proof of Proposition \ref{prop:approx_FockInt_NN},
in steps {\sf \encircle{1}} through {\sf \encircle{3}} we have
constructed local approximations for $\Psi^{(\ell)}$ in three different
regions of the interval $\mathcal{I}_\kappa$.
We combine them using a suitable partition of unity, as in step 
{\sf \encircle{4}} of the proof of Proposition \ref{prop:approx_FockInt_NN},
over the wavenumber-dependent interval $\mathcal{I}_\kappa$, which 
has been defined in \eqref{eq:interval_I_k}.

Define the map $T_{\kappa,n,\ell,\epsilon}:{\mathcal{I}_\kappa \rightarrow [-3,3]}$
as
\begin{equation}\label{eq:T_kappa_n_eps_l}
	T_{\kappa,n,\ell,\epsilon}(x)
	=
	\left\{
	\begin{array}{cl}
		-3+\frac{x+\kappa^{\frac{1}{3}} A}{-D^{-}_{\ell,\kappa,\epsilon}+\kappa^{\frac{1}{3}} A},
		& 
		x\in [-\kappa^{\frac{1}{3}} A,-D^{-}_{\ell,\kappa,\epsilon}),
		\\
		\frac{x}{D^{-}_{\ell,\kappa,\epsilon}/2}, 
		& 
		x\in [-D^{-}_{\ell,\kappa,\epsilon},0),
		\\
		\frac{x}{D^+_{{n,\ell,\epsilon}}/2}, 
		& 
		x\in [0,D^+_{{n,\ell,\epsilon}}),
		\\
		2+\frac{x-D^+_{{n,\ell,\epsilon}}}{\kappa^{\frac{1}{3}} A-D^+_{{n,\ell,\epsilon}}},
		& 
		x\in [D^+_{{n,\ell,\epsilon}},\kappa^{\frac{1}{3}} A].
	\end{array}
	\right.
\end{equation}

Observe that in the definition of $T_{\kappa,n,\epsilon}$ in \eqref{eq:T_kappa_n_eps_l}
we have implicitely assumed that $\kappa^\frac{1}{3}A \geq \max\{D^+_{{n,\ell,\epsilon}},D^{-}_{\ell,\kappa,\epsilon}\}$. 
Therefore, under the reasonable assumption that $D^+_{{n,\ell,\epsilon}}\geq D^{-}_{\ell,\kappa,\epsilon}$, this implies that
the bounds presented here are valid for $\kappa$ satisfying \eqref{eq:condition_kappa_l}.

The function $\Psi^{(\ell)}$ admits the representation
\begin{equation}\label{eq:partition_unity_T_2}
	\Psi^{(\ell)}(\tau)
	=
	\sum_{j=-3}^{3}
	\left(
		\chi_{j}
		\circ
		T_{\kappa,n,\ell,\epsilon}
	\right)
	(\tau)
	\Psi^{(\ell)}(\tau),
	\quad
	\tau
	\in 
	\mathcal{I}_\kappa.
\end{equation}

Let us set\footnote{
	The notation is slightly abused 
	in this definition. In \eqref{eq:Psi_NN_l}, we use
	$\Psi^{(\ell)}_{\kappa,n,\epsilon}$ to denote 
	a ReLU-NN emulating $\Psi^{(\ell)}$, not to denote
	the $\ell$-order derivative of a ReLU-NN. 
}
\begin{equation}\label{eq:Psi_NN_l}
	\Psi^{(\ell)}_{\kappa,n,j,\epsilon}
	\coloneqq
	\left\{
	\begin{array}{cl}
	\Psi^-_{\ell,\epsilon}, & j \in \{-3,-2\}, \\
	\Psi^\cc_{n,\ell,\epsilon}, & j \in \{-1,0,1\}, \\
	\Psi^+_{n,\ell,\epsilon}, & j \in \{2,3\}. \\
	\end{array}
	\right.
\end{equation}
As in \eqref{eq:network_approximation_D}, for each $n,\ell \in \IN$
and $\epsilon>0$ we set
\begin{equation}
\label{eq:network_approximation_D_DerLPsi}
	\Psi^{(\ell)}_{\kappa,n,\epsilon}(\tau)
	\coloneqq
	\sum_{j=-3}^{3} 
	\begin{pmatrix}
	\mu_{\lambda_\kappa,\epsilon}
	\left(
		\left(
			\chi_{j}\circ T_{\kappa,n,\ell,\epsilon}
		\right)(\tau)
		,
		\left(\Psi^{(\ell)}_{\kappa,n,j,\epsilon}(\tau)\right)_1
	\right) \\
	\mu_{\lambda_\kappa,\epsilon}
	\left(
		\left(
			\chi_{j}\circ T_{\kappa,n,\ell,\epsilon}
		\right)(\tau)
		,
		\left(\Psi^{(\ell)}_{\kappa,n,j,\epsilon}(\tau)\right)_2
	\right)
	\end{pmatrix},
	\quad
	\tau
	\in 
	\mathcal{I}_\kappa.
\end{equation}
As in step {\sf \encircle{4}} of the proof of Proposition \ref{prop:approx_FockInt_NN},
we can conclude that $\Psi^{(\ell)}_{\kappa,n,\epsilon}(\tau)$ can be expressed as 
a ReLU-NN.

{\sf \encircle{5}}
{\sf Bounds for the Width, Depth, and Weights}.
As in step {\sf \encircle{5}} of the proof of Proposition \ref{prop:approx_FockInt_NN}, 
we have
\begin{equation}
	\mathcal{M}
	\left(
		\Psi^{(\ell)}_{\kappa,n,\epsilon}
	\right)
	=
	6
	+
	\max_{j \in \{-3,\dots,3\}}
	\left\{
		2
		\mathcal{M}
		\left(
			\mu_{\lambda_\kappa,\epsilon}
		\right)
		,
		\mathcal{M}
		\left(
			\chi_{\ell}\circ T_{\kappa,n,\epsilon}
		\right)
		+
		\mathcal{M}
		\left(
			\Psi^{(\ell)}_{\kappa,n,j,\epsilon}
		\right)
	\right\}
	=
	30.
\end{equation}
The depth of $\Psi^{(\ell)}_{\kappa,n,\epsilon}(\tau)$ is bounded according to
\begin{equation}
	\mathcal{L}
	\left(
		\Psi^{(\ell)}_{\kappa,n,\epsilon}
	\right)
	\leq
	4
	+
	\sum_{\ell=-3}^{3}
	\left(
	\mathcal{L}
	\left(
		\mu_{\lambda_\kappa,\epsilon}
	\right)
	+
	\max
	\left\{
	\mathcal{L}
	\left(
		 \chi_{\ell}\circ T_{\kappa,n,\epsilon}
	\right)
	,
	\mathcal{L}
	\left(
		\Psi^{(\ell)}_{\kappa,n,j,\epsilon}
	\right)
	\right\}
	\right),
\end{equation}
thus yielding the bound stated in \eqref{eq:depth_Psi_L_der}.
\end{proof}

\bibliographystyle{siam}
\bibliography{ref}{}

\end{document}